\documentclass[11pt]{amsart}

\usepackage{mathtools}
\usepackage{amsmath}
\usepackage{amssymb}
\usepackage{yhmath}
\usepackage{graphicx}
\usepackage{ mathrsfs }
\usepackage{bbm}
\usepackage{xcolor}
\usepackage{dblfloatfix}

\usepackage{amscd,amssymb, bbm}

\usepackage{graphicx, amsmath, amssymb, amscd, amsthm, euscript, psfrag, amsfonts,bm}

\DeclareMathAlphabet{\mathpzc}{OT1}{pzc}{m}{it}

\numberwithin{equation}{section}

\numberwithin{equation}{section}

\newtheorem{theorem}{Theorem}[section]
\newtheorem{lemma}[theorem]{Lemma}

\newtheorem{dfn}[theorem]{Definition}\newtheorem{Def}[theorem]{Definition}
\newcommand{\BMS}{\operatorname{BMS}}

\usepackage{amssymb}

\DeclareMathAlphabet{\mathpzc}{OT1}{pzc}{m}{it}

\newtheorem{thm}[theorem]{Theorem}\newtheorem{Thm}[theorem]{Theorem}

\newtheorem{Rmk}[theorem]{Remark}

\newtheorem{Prop}[theorem]{Proposition}\newtheorem{prop}[theorem]{Proposition}

\newtheorem{cor}[theorem]{Corollary}
\newtheorem{lem}[theorem]{Lemma}\newtheorem{Lem}[theorem]{Lemma}

\newcommand{\be}{begin{equation}}

\newcommand{\bH}{\mathbb H}

\newcommand{\e}{{\e}}
\newcommand{\z}{\mathbb{Z}}

\renewcommand{\c}{\mathbb{C}}
\newcommand{\br}{\mathbb{R}}

\newcommand{\R}{\mathbb{R}}

\newcommand{{\grinv}}{{\Cal G}^{-r}}

\newcommand{\ba}{\backslash}

\newcommand{\G}{\Gamma}

\newcommand{\Cal}{\mathcal}

\newcommand{\SL}{\operatorname{SL}}

\newcommand{\bp}{\begin{pmatrix}}
\newcommand{\ep}{\end{pmatrix}}

\renewcommand{\bp}{{\rm bp}}

\newcommand{\SO}{\operatorname{SO}}
\newcommand{\SU}{\operatorname{SU}}

\newcommand{\T}{\operatorname{T}}

\newcommand{\bi}{\begin{itemize}}
\newcommand{\ei}{\end{itemize}}

\newcommand{\ben}{\begin{enumerate}}
\newcommand{\een}{\end{enumerate}}

\newcommand{\PSL}{\op{PSL}}

\newcommand{\RFM}{\op{RF}{M}}
\newcommand{\RF}{\op{RF}}

\newcommand{\GG}{\Gamma\ba G}

\newcommand{\op}{\operatorname}

\renewcommand{\be}{\begin{equation}}
\newcommand{\ee}{\end{equation}}

\newcommand{\C}{\mathcal C}
\newcommand{\oxH}{\overline{xH}}
\newcommand{\Inv}{\op{Inv}}
\newcommand{\hull}{\op{hull}}
\newcommand{\cal}{\mathcal}
\renewcommand{\e}{\varepsilon}
\newcommand{\La}{\Lambda}
\newcommand{\core}{\op{core}}
\newcommand{\bS}{\mathbb S}\newcommand{\mT}{\mathsf T}

\newcommand{\mC}{\mathcal C}
\newcommand{\FM}{\op{F} M}
\newcommand{\cl}{\overline}
\renewcommand{\R}{\mathcal R}

\newcommand{\RFPM}{\RF_+M}
\newcommand{\BR}{\operatorname{BR}}

\newcommand{\mS}{\mathscr{S}}
\newcommand{\mG}{\mathscr{G}}\newcommand{\mH}{\mathscr{H}}

\renewcommand{\epsilon}{\e}

\title[Discrete subgroups of $\SL_2(\c)$]{Dynamics for Discrete subgroups of $\SL_2(\c)$}

\author{Hee Oh}
\address{Mathematics department, Yale university, New Haven, CT 06511 and Korea Institute for Advanced Study, Seoul, Korea}
\email{hee.oh@yale.edu}

\thanks{Supported in part by NSF Grant \#1900101.}

\begin{document}\maketitle
\small{\it\quad\quad\quad\quad  Dedicated to Gregory Margulis with affection and admiration}

\begin{abstract} Margulis wrote in the preface of his book {\it Discrete subgroups of semisimple Lie groups} \cite{Ma3}: 
``A number of important topics have been omitted. The most significant of these is the theory of Kleinian groups and Thurston's theory of 3-dimensional manifolds: these two theories can be united under the common title {\it Theory of discrete subgroups of $\SL_2(\c)$}".

In this article, we will discuss a few recent advances regarding this missing topic from his book, which were influenced by his earlier works.
\end{abstract}

\tableofcontents
\section{Introduction}
A discrete subgroup of $\PSL_2(\c)$ is called a Kleinian group. In this article, we discuss dynamics of unipotent flows on 
the homogeneous space $\Gamma\ba \PSL_2(\c)$ for a Kleinian group $\Gamma$ which is not necessarily a lattice of $\PSL_2(\c)$.
Unlike the lattice case, the geometry and topology of the associated hyperbolic $3$-manifold $M=\Gamma\ba \bH^3$ influence
both topological and measure theoretic rigidity properties of unipotent flows. 

Around 1984-6, Margulis settled  the Oppenheim conjecture by proving that every bounded $\SO(2,1)$-orbit
in the space $\SL_3(\z)\ba \SL_3(\br)$ is compact (\cite{Ma1}, \cite{Ma}).  His proof  was topological, using minimal sets and the polynomial divergence property of unipotent flows.  With Dani (\cite{DM1}, \cite{DM2}), he also gave a classification of orbit closures for a certain family of one-parameter unipotent
subgroups of $\SL_3(\br)$. Based on Margulis' topological approach, Shah \cite{Sht}
 obtained a classification of orbit closures for the action of any connected closed subgroup generated by unipotent elements in the space $\Gamma\ba \PSL_2(\c)$ when $\Gamma$ is a lattice. This result in a much greater generality, as conjectured by Raghunathan, was proved by Ratner  using
 her measure rigidity theorem (\cite{Ra}, \cite{Ra1}).
 
 The relation between invariant measures and orbit closures for unipotent flows is not as tight in the infinite volume case as it is in
 the finite volume case. Meanwhile, the topological approach in the orbit closure classification can be extended to the class of rigid acylindrical hyperbolic $3$-manifolds, yielding the complete classification of orbit closures for the action of any connected closed subgroup generated by
 unipotent elements. This  was done jointly with McMullen and Mohammadi (\cite{MMO}, \cite{MMO2}). Much
 of this article  is devoted to explaining these results, although we present slightly different viewpoints in certain parts of the proof. Remarkably, this approach can handle the entire quasi-isometry class of rigid acylindrical hyperbolic $3$-manifolds,
 as far as the action of the subgroup $\PSL_2(\br)$ is concerned \cite{MMO1}. An immediate geometric consequence 
 is that for any convex cocompact acylindrical hyperbolic $3$-manifold $M$, any geodesic plane is either closed or dense
 inside the interior of the convex core of $M$; thereby
 producing the first continuous family of locally symmetric manifolds for which such a strong rigidity theorem for geodesic planes holds.
 This result extends to geometrically finite acylindrical hyperbolic $3$-manifolds as shown in joint work with Benoist  \cite{BO}.
 We also present a continuous family of quasifuchsian $3$-manifolds containing
 geodesic planes with wild closures \cite{MMO1}, which indicates the influence of
 the topology of the associated $3$-manifold in the rigidity problem at hand.
 
We call a higher dimensional analogue of a rigid acylindrical hyperbolic $3$-manifold
{ a convex cocompact hyperbolic $d$-manifold with Fuchsian ends}, following Kerckhoff and Storm \cite{KS}. For these manifolds $\Gamma\ba \bH^d$,
in joint work with Lee \cite{LO}, we have established a complete classification of orbit closures in $\Gamma\ba \SO^\circ (d,1)$
 for the action of
any connected closed subgroup
of $\SO^\circ (d,1)$ generated by unipotent elements. The possibility of accumulation on closed orbits of intermediate subgroups
presents new challenges, and the avoidance theorem and the induction arguments involving equidistribution statement
are major new ingredients in higher dimensional cases (Theorems \ref{hard} and \ref{mainth}).
We note that these manifolds do not admit any non-trivial local deformations for $d\ge 4$ \cite{KS}.


\bigskip

\noindent{\bf Acknowledgement} This survey article is mostly based on the papers
 \cite{MMO}, \cite{MMO1}, \cite{MMO2}, \cite{BO}, and \cite{LO}. I am grateful to my co-authors Curt McMullen, Amir Mohammadi, Yves Benoist and Minju Lee.
I would like to thank
 Yair Minsky and Amir Mohammadi  for helpful  comments on the preliminary version of this article.

\section{Kleinian groups}
 We give a brief introduction to Kleinian groups, including some basic notions and examples.
General references for this section
include \cite{Ra}, \cite{Mar}, \cite{MR}, \cite{MT}, \cite{Se}, \cite{Ha} and \cite{Ca}. In particular,
all theorems stated in this section with no references attached can be found  in \cite{Mar} and \cite{MT}.

We will  use the upper half-space model for hyperbolic $3$-space:
 $$\bH^3=\{(x_1, x_2, y):y>0\}, \quad ds=\frac{\sqrt{dx_1^2+dx_2^2+dy^2}}{y} .$$

In this model of $\bH^3$, a geodesic  is  either a vertical line or a vertical semi-circle. 
The geometric boundary of $\bH^3$ is given by the Riemann sphere $\bS^2=\hat{\c}$, when we identify
the plane $(x_1, x_2, 0)$ with the complex plane $\c$.

The group $G:=\op{PSL_2}(\mathbb C)$ acts on $\hat \c$ by M\"obius transformations: 

$$\begin{pmatrix} a& b\\ c& d\end{pmatrix} z=\frac{az+b}{cz+d} \quad \text{ with $a,b,c,d\in \c$ such that $ad-bc=1$} .$$

This action of $G$ extends to an isometric action on $\bH^3$ as follows: each $g\in G$
can be expressed as a composition $\Inv_{ C_1}\circ  \cdots \circ  \Inv_{ C_k}$, where $\Inv_{C}$ denotes the inversion with respect to a circle
$C\subset \hat \c$.\footnote{ If $C=\{z: |z-z_0|=r\}$, then $\Inv_C(z)$ is the unique point on the ray
$\{t z:t\ge 0\}$, satisfying the equation 
$|z-z_0| \cdot |\Inv_C(z)-z_0|=r^2$ for all $z\ne z_0$, and $\Inv_C(z_0)=\infty$.}
If we set $\Phi (g)=\Inv_{\hat C_1} \circ \cdots \circ  \Inv_{\hat C_k}$ where
 $\Inv_{\hat C}$ denotes the inversion with respect to the
 sphere $\hat C$ in $\br^3$
which is orthogonal to $\c $ and $\hat C \cap \c =C$, then $\Phi(g)$ preserves $(\bH^3, ds)$. 
Moreover, the Poincar\'e extension theorem says that $\Phi$
is an isomorphism between the two real Lie groups:
 $$\op{PSL}_2(\mathbb C)=\op{Isom}^+(\bH^3),$$ 
where  $ \op{PSL_2}(\mathbb C)$ is regarded as a $6$-dimensional real Lie group
and $\op{Isom}^+(\bH^3)$ denotes the group of all orientation preserving isometries of $\bH^3$.


\begin{Def} A discrete subgroup $\Gamma$ of $G$ is called a Kleinian group.
\end{Def}
For a  (resp. torsion-free) Kleinian group $\Gamma$, the quotient $\Gamma\ba \bH^3$ is a hyperbolic orbifold (resp. manifold).
Conversely, any  complete hyperbolic $3$-manifold $M$ can be presented as a quotient 
$$M=\Gamma \ba \bH^3$$ for a torsion-free Kleinian group $\Gamma$.
The study of hyperbolic manifolds is therefore directly related to the study of Kleinian groups.

Throughout the remainder of the article, we assume that a Kleinian group $\Gamma$ is non-elementary  i.e.,
$\Gamma$ does not contain an abelian subgroup of finite index. By Selberg's lemma, every Kleinian group has
a torsion-free subgroup of finite index. We will henceforth treat the torsion-free condition loosely.

\subsection{Lattices} The most well-studied Kleinian groups are lattices of $G$: a Kleinian group $\Gamma <G$ is a lattice if $M=\Gamma\ba \bH^3$
has finite volume. When $M$ is compact, $\Gamma$ is called a uniform or cocompact lattice. 
If $d>0$ is a square-free integer, then $\PSL_2(\z [\sqrt{- d}])$ is a non-uniform lattice of $G$.
More lattices, including uniform ones, can be constructed by number theoretic methods using
 the Lie group isomorphism
 $G\simeq \SO^\circ (3,1)$.

Let  $Q(x_1, x_2, x_3, x_4)$ be a quadratic form with  coefficients over a totally real number field $k$ of degree $n$
such that $Q$ has signature $(3,1)$ and
for any non-trivial embedding  $\sigma: k\to \br$, $Q^\sigma$ has signature $(4,0)$ or $(0,4)$; the orthogonal group
$\SO(Q^\sigma)$ is thus compact.

Then for $G= \SO^\circ (Q)$ and for the ring $\mathfrak o$ of integers of $k$, the subgroup
\be \label{ex} \Gamma:= G\cap \SL_4(\mathfrak o)\ee
 is a lattice in $G$ by a theorem of Borel and Harish-Chandra \cite{BHC}. 
Moreover, if $Q$ does not represent $0$ over $k$ (which is always the case if the degree of $k$ is bigger than $1$),  then
$\Gamma$ is a uniform
 lattice in $G$ by the Godement's criterion.
These examples contain all {\it arithmetic} lattices (up to a commensurability) which 
contain cocompact Fuchsian subgroups, that is, uniform lattices of $\SO^\circ(2,1)\simeq\PSL_2(\br)$ \cite{MR}.

Take two arithmetic non-commensurable hyperbolic $3$-manifolds $N_1$ and $N_2$ which 
 share a common properly imbedded closed geodesic surface $S$, up to an isometry. 
 We cut each $N_i$ along $S$, which results in one or two connected components.
Let  $M_i$ be the metric completion of a component of $N_i-S$, which has geodesic boundary isometric to
  one or two copies of $S$.
We now glue one or two copies of $M_1$ and $M_2$ together  along their geodesic boundary and get a
a connected finite-volume hyperbolic $3$-manifold with no boundary. The resulting $3$-manifold is a
non-arithmetic hyperbolic $3$-manifold,  and its fundamental group is  an example of the so-called 
hybrid  lattices constructed by Gromov and Piatetski-Schapiro \cite{GPS}.

Mostow rigidity theorem  says that any two isomorphic lattices of $G$ are conjugate to each other. Since a lattice is finitely presented, it follows that
 a  conjugacy class of a lattice  is determined by its presentation. Hence, despite the presence of non-arithmetic lattices in $G$,
  there are only countably many lattices of $G$ up to conjugation, or equivalently, there are only countably many hyperbolic manifolds of finite volume, up to isometry.


\subsection{Finitely generated Kleinian groups} 
 We will mostly focus on {\it finitely generated} Kleinian groups. When studying a finitely generated Kleinian group $\Gamma$, the associated limit set and the convex core play fundamental roles. 

Using the M\"obius transformation action of $\Gamma$ on $\bS^2$, we define:
\begin{Def} The limit set $\Lambda\subset \bS^2$ of $\Gamma$ is  the set of all accumulation points
of $\Gamma(z)$ for $z\in \bH^3\cup \bS^2$. 
\end{Def}
This definition is independent of the choice of $z\in \bH^3\cup \bS^2$, and  $\Lambda$ is a minimal $\Gamma$-invariant closed subset of
$\bS^2$.

\begin{Def} The convex core of $M$ is the convex submanifold of $M$ given by 
 $$\op{core} M :=  \Gamma\ba \op{hull} \Lambda \; \subset  \; M=\Gamma\ba \bH^3 $$
 where $\op{hull} \Lambda \subset \bH^3$ is the smallest convex subset containing all geodesics connecting two points in $\Lambda$.
\end{Def}

 \begin{figure}\label{SS}
 \includegraphics[height=4cm]{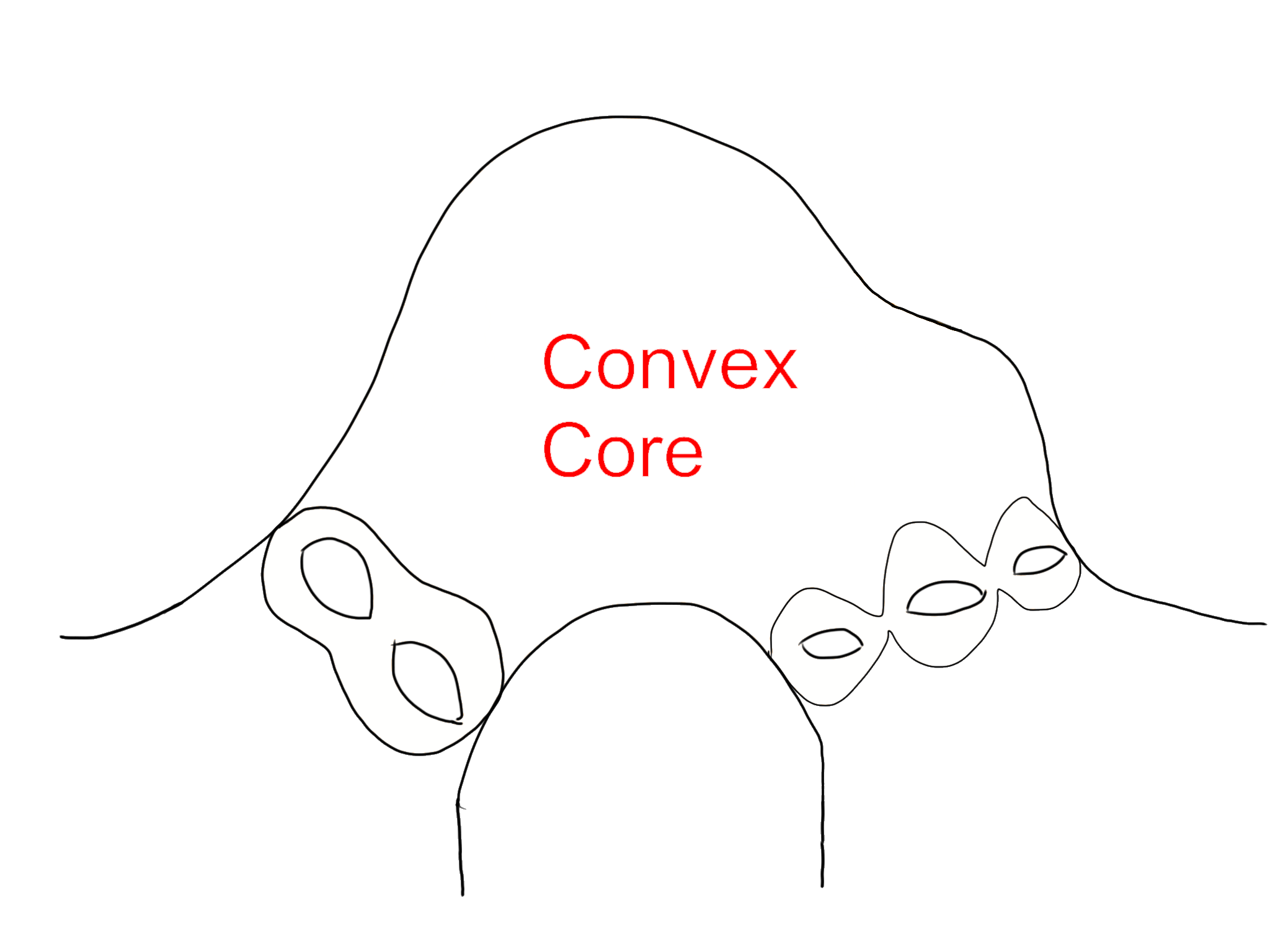}
\caption{The convex core}
\end{figure}

  If $\text{Vol}(M)<\infty$, then $\Lambda=\bS^2$ and hence $M$ is equal to its convex core.

\begin{Def}\label{def}
\begin{enumerate}
\item
A Kleinian group $\Gamma$ is called {\it geometrically finite}  if the unit neighborhood of $\core{M}$ has finite volume.
\item A  Kleinian group $\Gamma$ is called {\it convex cocompact}
 if  $\core{M}$ is compact, or equivalently, if $\Gamma$ is geometrically finite without any parabolic elements.
\end{enumerate} \end{Def}
An element $g\in G$ is either hyperbolic (if it is conjugate to a diagonal element whose entries
have modulus  not equal to $1$), elliptic (if it is conjugate to a diagonal element whose entries
have modulus $1$) or parabolic (if it is conjugate to a strictly upper triangular matrix).
By discreteness, an element of a torsion-free Kleinian group is either hyperbolic or parabolic.

Geometrically finite  (resp. convex cocompact)  Kleinian groups are natural generalization of (resp. cocompact) lattices of $G$.
Moreover, the convex core of a geometrically finite hyperbolic manifold admits a thick-thin decomposition: there exists a constant
$\e>0$ such that $\core M$ is the union of a compact subset of injectivity radius at least $\e>0$
 and finitely many cusps.
In the class of geometrically finite groups, lattices are  characterized by the property that their limit sets are the whole of $\bS^2$, and the limit sets of
other geometrically finite groups have Hausdorff dimension strictly smaller than $2$ 
(\cite{Su},\cite{Su2}). 

The group $G=\PSL_2(\c)$ can be considered as a real algebraic subgroup, more precisely, the group of real points of an algebraic group 
$\bf G$ defined over $\br$. A subset $S\subset G$ is called {\it Zariski dense} if $S$ is not contained in any proper real algebraic subset of $G$.
The Zariski density of a Kleinian group $\Gamma$ in
$G$ is equivalent to the property that its limit set $\Lambda$ is not contained in any circle of $\bS^2$.
When $\Lambda $ is contained in a circle, $\Gamma$ is conjugate to a discrete subgroup of $\PSL_2(\br)$; such Kleinian groups are referred to
as {\it{Fuchsian}} groups.
Geometrically finite Kleinian groups are always finitely generated, but the converse is not true in general; see \eqref{gi}.


\subsection{Examples of geometrically finite groups} 
Below we give examples of three different kinds of geometrically finite groups which are relevant to subsequent discussion.
Their limit sets are respectively totally disconnected, Jordan curves, and Sierpinski carpets.
We note that a geometrically finite non-lattice Zariski dense Kleinian group $\Gamma$
is determined by its limit set $\Lambda$ up to commensurability, more precisely, $\Gamma$ is a subgroup of finite index
in the {\it discrete} subgroup $\op{Stab}(\Lambda)=\{g\in G: g(\Lambda)=\Lambda\}$.

\subsubsection{Schottky groups}
The simplest examples of geometrically finite groups are Schottky groups.
A subgroup $\Gamma <G$ is called (classical) Schottky if
$\Gamma$ is generated by hyperbolic elements $g_1, \cdots, g_k\in G$, $k\ge 2$, satisfying that
 there exist mutually disjoint closed round disks $B_1, \cdots, B_k$ and $ B_1', \cdots, B_k'$ in $\bS^2$
such that each $g_i$ maps the exterior of $B_i$ onto the interior of $B_i'$.

If $g_1, \cdots, g_k$ are hyperbolic elements of $G$ whose fixed points in $\bS^2$ are mutually disjoint,
then $g_1^N, \cdots, g_k^N$ generate a Schottky group for all $N$ large enough.
A Schottky group $\Gamma$ is discrete and free; the common exterior  of  the hemi-spheres bounded by $B_i$, $B_i'$
is a fundamental domain $F$ of $\Gamma$. Since the limit set of $\Gamma$, which is totally disconnected,
 is contained in the union of interiors of $B_i$ and $B_i$'s,
it is easy to see that the intersection of the hull of $\Lambda$ and the fundamental domain $F$ is a bounded subset of $F$. Hence $\Gamma$
 is a convex cocompact subgroup. 
 Its convex core is the  handle body of genus $k$; in particular, the boundary of $\core M$ is a closed surface of genus $k$.
 
 Any  Kleinian group $\Gamma$ contains a Schottky subgroup  which has the same Zariski closure.
 If $\Gamma$ is Zariski dense, take any two hyperbolic elements $\gamma_1$ and $\gamma_2$ of $\Gamma$
 with disjoint sets of fixed points.
 Suppose that  all of four fixed points lie in a circle, say, $C\subset \bS^2$; note that $C$ is uniquely determined. 
Since the set of  fixed points of hyperbolic elements of $\Gamma$ forms a dense subset
 of $\Lambda $,  there exists a hyperbolic element $\gamma_3\in \Gamma$
 whose fixed points are not contained in $C$. Now, for any $N\ge 1$, the subgroup generated by
 $\gamma_1^N, \gamma_2^N, \gamma_3^N$ is Zariski dense, as its limit set cannot be contained in a circle.
 By taking $N$ large enough, we get a Zariski dense Schottky subgroup of $\Gamma$. 
  This in particular implies that any Kleinian group contains a convex cocompact subgroup, which is as large as itself in the algebraic sense.

\subsubsection{Fuchsian groups and deformations: quasifuchsian groups}
An orientation preserving homeomorphism $f:\mathbb S^2\to \mathbb S^2$ is called $\kappa$-quasiconformal
if for any $x\in \bS^2$,
$$\limsup_{r\to 0}\frac{\sup \{|f(y)-f(x)|: |y-x|=r\} }{\inf  \{|f(y)-f(x)|: |y-x|=r\}} \le \kappa .$$
The $1$-quasiconformal maps are precisely conformal maps \cite[Sec.2]{Mar}.
The group $G=\PSL_2(\c)$ is precisely the group of all conformal automorphisms of $\mathbb S^2$.

A Kleinian group $\Gamma$ is called {\it quasifuchsian} if it is a quasiconformal deformation of a (Fuchsian) lattice of $\PSL_2(\br)$, i.e.,
there exists a quasiconformal map $f$ and a lattice $\Delta <\PSL_2(\br)$ such that
$\Gamma= \{f \circ \delta\circ  f^{-1}: \delta\in \Delta\}$.
Any quasi-conformal deformation of a geometrically finite group is known to be geometrically finite; so a quasifuchsian group
is geometrically finite.

A quasifuchsian group is also characterized as a finitely generated Kleinian group whose limit set $\Lambda$ is a Jordan curve and which preserves each component of $\bS^2-\Lambda$. If $\Omega_{\pm}$ are components of $\bS^2-\Lambda$, then
$S_{\pm}:=\Gamma\ba \Omega_{\pm}$ admits a hyperbolic structure by the uniformization theorem, and
the product $\op{Teich}(S_+)\times  \op{Teich}(S_-)$ of Teichmuller spaces
gives a parameterization of all quasifuchsian groups which are quasiconformal deformations of a fixed lattice of $\PSL_2(\br)$.

 \begin{figure}\label{SSS3}
 \includegraphics[height=6cm]{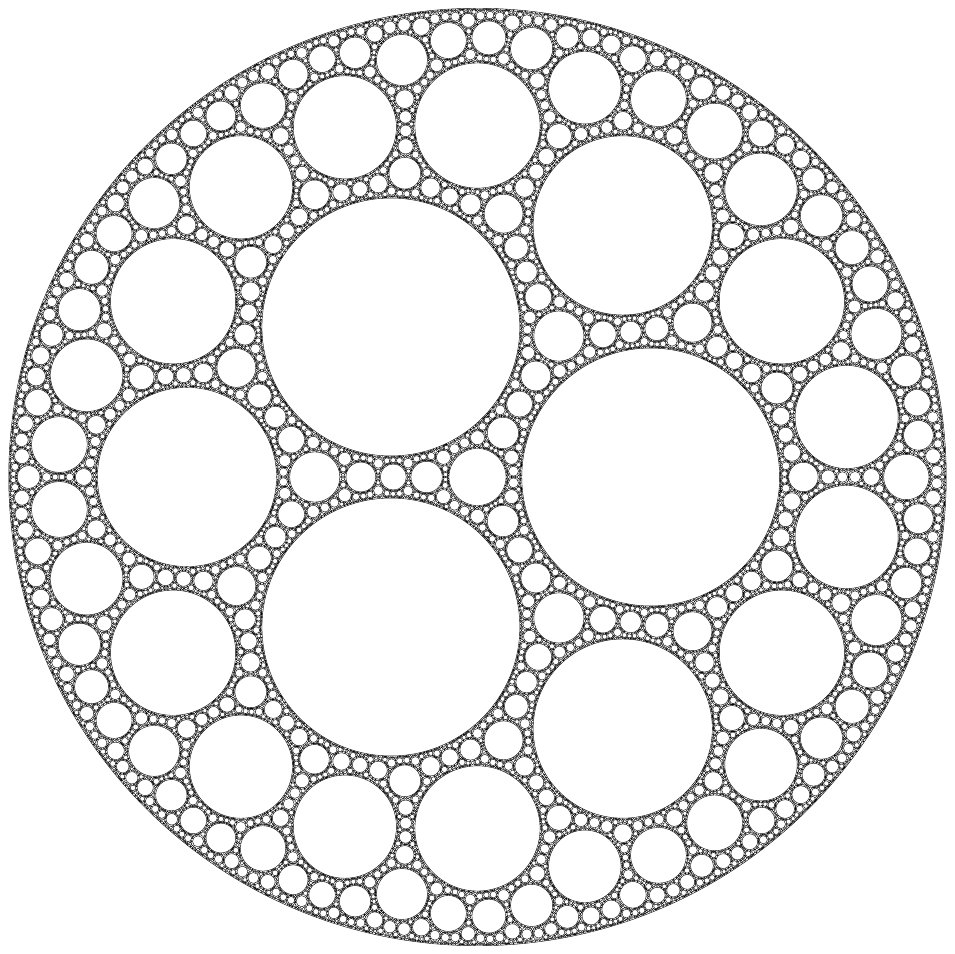} 
\caption{Limit set of a rigid acylindrical group, by McMullen}
\end{figure}

\subsubsection{Rigid acylindrical  groups and their deformations}\label{sell}
A Kleinian group $\Gamma<G$ is called {\it rigid acylindrical} if 
the convex core of the associated hyperbolic manifold $M=\Gamma\ba \bH^3$ is a
compact manifold with non-empty interior and with totally geodesic boundary.
If $\core M$ has empty boundary, then $M$ is compact and hence $\Gamma$ is  a uniform lattice. 
Rigid acylindrical non-lattice groups are characterized as
 convex cocompact Kleinian groups   whose limit set satisfies that
  $$\bS^2 -\Lambda=\bigcup B_i$$
where $B_i$'s are round disks  with mutually disjoint closures.

If $M$ is a rigid acylindrical hyperbolic $3$-manifold of infinite volume then the double
of $\core M$ is a closed hyperbolic $3$-manifold; hence
 any rigid acylindrical group is a subgroup of a uniform lattice of $G$, which contains a co-compact Fuchsian lattice
 $\pi_1(S)$ for a component $S$ of $\partial \core M$.
Conversely, if $\Gamma_0$ is a torsion-free uniform  lattice of $G$ such that $\Delta:=\Gamma_0\cap \PSL_2(\br)$ is a uniform lattice in
$\PSL_2(\br)$, then $M_0=\Gamma_0\ba \bH^3$ is a closed hyperbolic $3$-manifold which contains a properly immersed totally geodesic
surface $\Delta\ba \bH^2$. By passing to a finite cover of $M_0$, $M_0$ contains a properly embedded totally geodesic surface, say $S$ \cite[Theorem 5.3.4]{MR}.
Now the metric completion of a component of $M_0-S$ is a compact hyperbolic $3$-manifold with
totally geodesic boundary, and  its fundamental group, which
injects to $\Gamma_0=\pi_1(M_0)$, is
 a rigid acylindrical Kleinian group.


Rigid acylindrical Kleinian groups admit a huge deformation space, comprised of convex cocompact acylindrical groups.
We begin with the notion of acylindricality for a compact $3$-manifold.
 Let $D^2$ denote a closed $2$-disk and let $C^2=S^1\times [0,1]$ be a cylinder.
A compact $3$-manifold $N$ is called {\it acylindrical}
\begin{enumerate}
\item if $\partial N$ is incompressible, i.e., any continuous map $f:(D^2, \partial D^2)\to (N, \partial N)$
can be deformed  into $\partial{N}$ or equivalently if the inclusion $\pi_1(S)\to \pi_1(N)$
is injective for any component $S$ of $\partial N$; and
\item if any essential cylinder of $N$ is boundary parallel, i.e., any
continuous map  $f:(C^2, \partial C^2)\to (N, \partial {N})$, injective on $\pi_1$, can be deformed into
  $\partial{N}$.
\end{enumerate}

\begin{figure} \label{sier} \begin{center}
 \includegraphics [height=3cm]{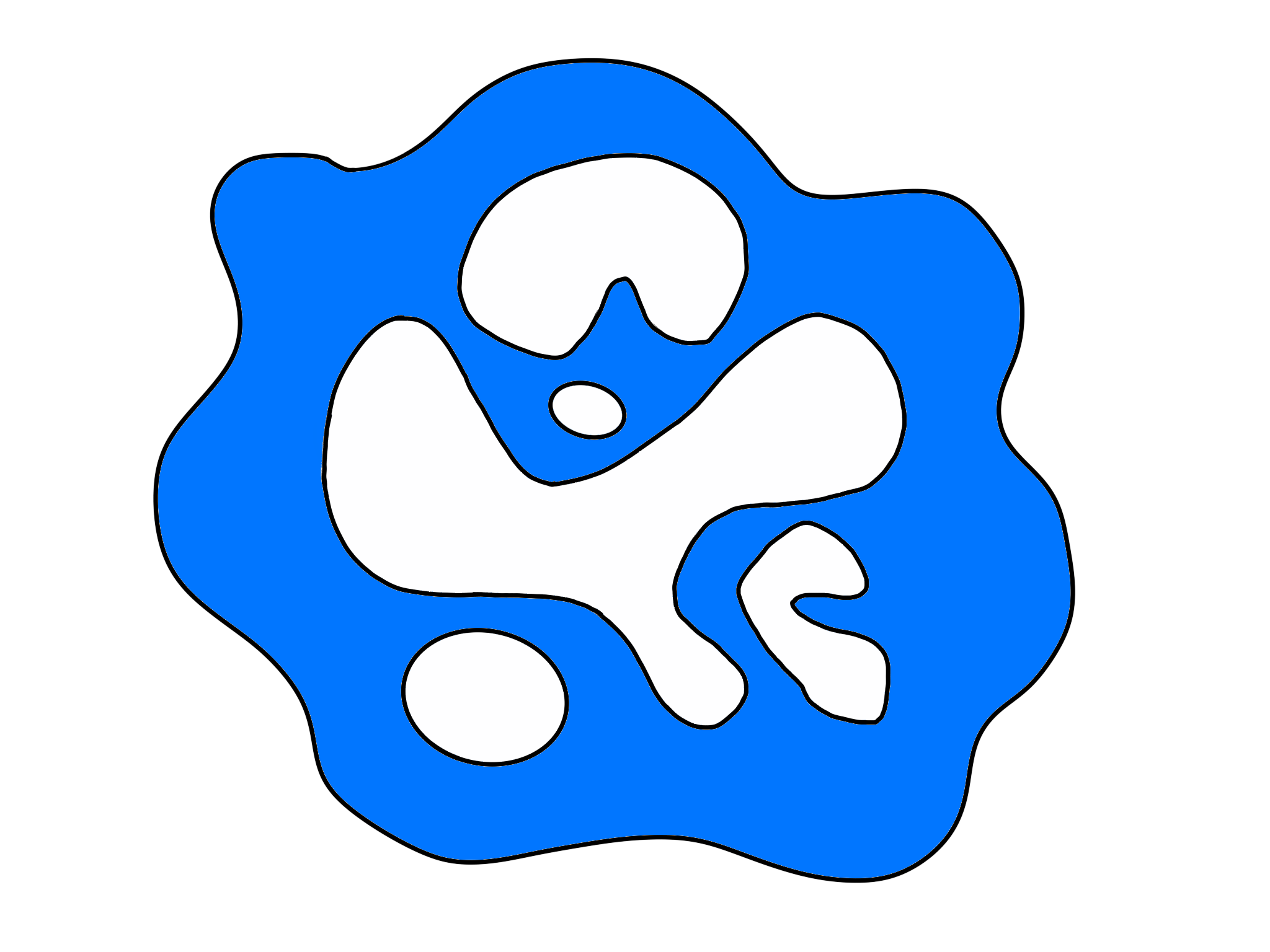} 
\end{center}\caption{Sierpinski carpet}
\end{figure}


A convex cocompact hyperbolic $3$-manifold $M$ is called {\it acylindrical} if its convex core is  acylindrical.
When $M$ has infinite volume, it is also described by
the property that its limit set is a Sierpinski carpet: $\bS^2 -\Lambda=\bigcup B_i$
 is  a dense union of Jordan disks $B_i$'s  with mutually disjoint closures and with  $\op{diam}(B_i)\to 0$. By Whyburn \cite{Wh},
 all Sierpinski carpets are known to be homeomorphic to each other.
 We refer to a recent preprint \cite{Zh} for a beautiful picture of the limit set of a convex cocompact (non-rigid) acylindrical group.

Any  convex cocompact acylindrical Kleinian group $\Gamma$  is a quasi-conformal deformation of a unique rigid acylindrical
Kleinian group $\Gamma_0$, and its quasi-conformal class is parametrized by
the product  $\prod_i {\op {Teich}}(S_i)$ where $S_i$'s are components of $\partial \core (\Gamma_0\ba \bH^3)$ (\cite{Th} and \cite{M}).
In terms of a manifold,
any  convex cocompact acylindrical hyperbolic  $3$-manifold is  quasi-isometric  to a unique  {\it rigid} acylindrical hyperbolic
 $3$-manifold $M$, and  its quasi-isometry class is parametrized by
 $\prod_i {\op {Teich}}(S_i)$.

The definition of acylindricality can be extended to geometrically finite groups with cusps using the notion of a compact core.
If $M$ is a hyperbolic $3$-manifold with finitely generated $\pi_1(M)$,
then there exists a a compact connected submanifold $\mathcal C\subset M$ (with boundary)
such that the inclusion $\mathcal C\subset  M$ induces an isomorphism $\pi_1(\cal C)\simeq  \pi_1(M)$; such $\mathcal C$ exists
uniquely, up to homeomorphism, and is called the compact core of $M$.
 Now a geometrically finite hyperbolic $3$-manifold $M$ is called {acylindrical} if its compact core  is an
acylindrical compact $3$-manifold.



\subsection{Thurston's geometrization theorem}
The complement $\Omega:=\bS^2-\Lambda$ is called the set of discontinuity.
Let $\Gamma$ be a finitely generated Kleinian group.
Ahlfors finiteness theorem says that $\Gamma\ba \Omega$ is a union of finitely many closed Riemann surfaces with at most a finite number of punctures.
The Kleinian manifold associated to $\Gamma$ is defined by {\it adding} $\Gamma\ba \Omega$ to $\Gamma\ba \bH^3$ on the conformal boundary at infinity:
$$\mathcal M(\Gamma)=\Gamma\ba \bH^3\cup \Omega, \quad \partial \mathcal M(\Gamma)=\Gamma\ba \Omega.$$

The convex cocompactness of  $\Gamma$ is equivalent to the compactness of $\mathcal M(\Gamma)$. If $\Gamma$
is geometrically finite with cusps, then $\mathcal M(\Gamma)$ is compact except possibly for a finite number of rank one and rank two cusps.
We denote by $\mathcal M_0(\Gamma)$ the compact submanifold of  $\mathcal M(\Gamma)$
obtained by removing the interiors of  solid pairing tubes corresponding to rank one cusps 
and solid cusp tori corresponding to rank two cusps (cf. \cite{Mar}).

The following is  a special case of Thurston's geometrization theorem under the extra non-empty boundary condition  (cf. \cite{Ka}):
\begin{Thm} Let $N$ be a compact irreducible\footnote{every $2$-sphere bounds a ball}
 orientable atoroidal\footnote{ any $\mathbb Z^2$ subgroup comes from boundary tori} $3$-manifold with non-empty boundary. Then
$N$ is homeomorphic to
$\mathcal M_0(\Gamma)$
for some geometrically finite Kleinian group $\Gamma$.
\end{Thm}

We remark that if $\partial N$ is incompressible and $N$ does not have any essential cylinders, then $\Gamma$ is a geometrically finite acylindrical group.

By applying Thurston's theorem to the compact core of
$\Gamma\ba \bH^3$, we deduce that every finitely generated Kleinian group $\Gamma$ is isomorphic to a geometrically finite group.

\subsection{Density of geometrically finite groups} 
The density conjecture of Bers, Sullivan and Thurston says that
most of Kleinian groups are geometrically finite. This is now a theorem whose proof combines the work
of many authors with the proof in full generality
due to Namazi-Souto and Ohshika (we refer to  \cite[Sec. 5.9]{Mar}
for more details and background).
\begin{thm} [Density theorem] The class of  geometrically finite Kleinian  groups is  open and dense 
 in the space of all finitely generated Kleinian groups. \end{thm}
In order to explain the topology used in the above theorem, let $\Gamma$ be a finitely generated Kleinian group. 
By Thurston's geometrization theorem, there exists a geometrically finite Kleinian group $\Gamma_0$ and an isomorphism $\rho:\Gamma_0\to\Gamma$. In fact, a more refined version gives that
$\rho$ is type-preserving, i.e., $\rho$ maps a parabolic element to a parabolic element. Fix a finite generating set $\gamma_1, \cdots, \gamma_k$ of $\Gamma_0$. The density theorem says there exists
a sequence of geometrically finite groups $\Gamma_n<G $, and isomorphisms $\rho_n: \Gamma_0\to \Gamma_n$ such that
$\rho_n $ converges to $\rho$ as $n\to \infty$, in the sense that 
$\rho(\gamma_i)=\lim_{n} \rho_n (\gamma_i)$ for each $i=1, \cdots, k$.

Here is an alternative way to describe the density theorem:
 Fix a geometrically finite Kleinian group $\Gamma$ with a fixed
set of generators $\gamma_1, \cdots, \gamma_k$ and relations $\omega_1, \cdots, \omega_r$.
Define $$\mathfrak R(\Gamma):=\{\rho:\Gamma\to G \text{ homomorphism}\}/\sim$$
with the equivalence relation given by conjugation by elements of $G$.
The set $\mathfrak R(\Gamma)$ can be identified with the algebraic variety
 $\{(g_1, \cdots, g_k)\in G\times \cdots \times G: \omega_i(g_1, \cdots, g_k)=e\;\;\text{ for $1\le i\le r$} \}/\sim$
where $\sim$ is given by conjugation by an element
of $G$ under the diagonal embedding. This defines a topology on $\mathfrak R(\Gamma)$, called
the algebraic convergence topology.

The discrete locus  is then defined by
the subcollection of discrete and faithful representations:
$$AH(\Gamma):=\{\rho\in \mathfrak R(\Gamma): \text{ type preserving isomorphism to a Kleinian group}\}.$$
Then $AH(\Gamma)$ is a closed subset, which parametrizes
hyperbolic structures on $\Gamma\ba \bH^3$.
The interior of $AH(\Gamma)$ consists of geometrically finite
Kleinian groups, and the density theorem says that 
$$\overline{\text{Int}AH(\Gamma)}= AH(\Gamma).$$

When $\Gamma$ is a lattice in $G$, $AH(\Gamma)$ is a single point by Mostow rigidity theorem.
For all other geometrically finite Kleinian groups, $AH(\Gamma)$ is {\it huge};
the quasiconformal deformation space of $\Gamma$
given by $$\mathfrak T (\Gamma)=\{\rho\in AH(\Gamma): \rho\text{ is induced by
a quasiconformal deformation of $\Gamma$}\}$$
is a connected component of the interior of $AH(\Gamma)$ and is a complex analytic manifold
of dimension same as the dimension of $\op{Teich}(\Gamma\ba \Omega)$, i.e.,
$\sum_{i=1}^m (3g_i +n_i -3)$ where $g_i$ is the genus of the $i$-th component of
$\Gamma\ba \Omega =\partial \mathcal M(\Gamma)$ and $n_i$ is the number of its punctures \cite[Thm. 5.13]{Mar}.
Moreover when $\Gamma$ is rigid acylindrical,
the interior of $AH(\Gamma)$, modulo the orientation (in other words,
modulo the conjugation by elements of $\op{Isom}(\bH^3)$, rather than by elements of $G=\op{Isom}^+(\bH^3)$),
is connected, and hence is equal to $\mathfrak T(\Gamma)$; this can be deduced from \cite{CM}, as explained to us by Y. Minsky.
Therefore $\op{Int} AH(\Gamma)/\pm =\mathfrak T(\Gamma)=\op{Teich}(\Gamma\ba \Omega)$.

\subsection{Examples of geometrically infinite groups} \label{gi}
Not every finitely generated Kleinian group is geometrically finite. An important class of finitely generated geometrically infinite Kleinian groups is given by the fundamental groups of $\mathbb Z$-covers of closed hyperbolic $3$-manifolds.
The virtual fibering theorem, proved by Agol, building upon the previous work of Wise, says that
 every closed hyperbolic $3$-manifold
is a surface bundle over a circle, after passing to a finite cover \cite[Sec 6.4]{Mar}. This implies that, up to passing to a subgroup of finite index,
 any uniform lattice  $\Gamma_0$ of $G$ contains
 a normal subgroup $\Delta$ such that $\Gamma_0/\Delta\simeq \z$ and $\Delta$ is a surface subgroup, i.e., isomorphic to the fundamental group of a closed hyperbolic surface. Note that $\Delta$ is  finitely generated (being a surface subgroup) but geometrically infinite; as no normal subgroup of a geometrically finite group of infinite index is geometrically finite. 
  In fact, any finitely generated, geometrically infinite, subgroup of a uniform lattice of $G$ arises in this way, up passing to a subgroup of finite index  (cf. \cite{Ca}). These manifolds give examples of degenerate hyperbolic $3$-manifolds with $\Lambda=\bS^2$. We mention that there are also degenerate hyperbolic manifolds with $\Lambda\ne \bS^2$.



\section{Mixing and classification of $N$-orbit closures}
Let  $\Gamma <G=\PSL_2(\c)$ be a Zariski dense geometrically finite  Kleinian group, and  
$M:=\Gamma\ba \bH^3$  the associated hyperbolic $3$-manifold. 
We denote by $$\pi:\bH^3\to M=\Gamma\ba \bH^3$$ the quotient map.

We fix $o\in \bH^3$ and  a unit tangent vector $v_o\in \T_o(\bH^3)$ so that
$K=\SU(2)$ and $M_0=\{\op{diag}(e^{i\theta}, e^{-i \theta}):\theta\in \br\}$
are respectively the stabilizer subgroups of $o$ and $v_o$. The action of $G$ on $\bH^3$ induces  identifications $G/K \simeq \bH^3 $, 
$G/M_0 \simeq \T^1(\bH^3)$, and  $G\simeq \op{F}(\bH^3)$, where $\T^1(\bH^3)$ and $\op{F}(\bH^3)$
denote respectively the unit tangent bundle and the oriented frame bundle over $\bH^3$.

\begin{figure} 
  \includegraphics [height=5.5cm]{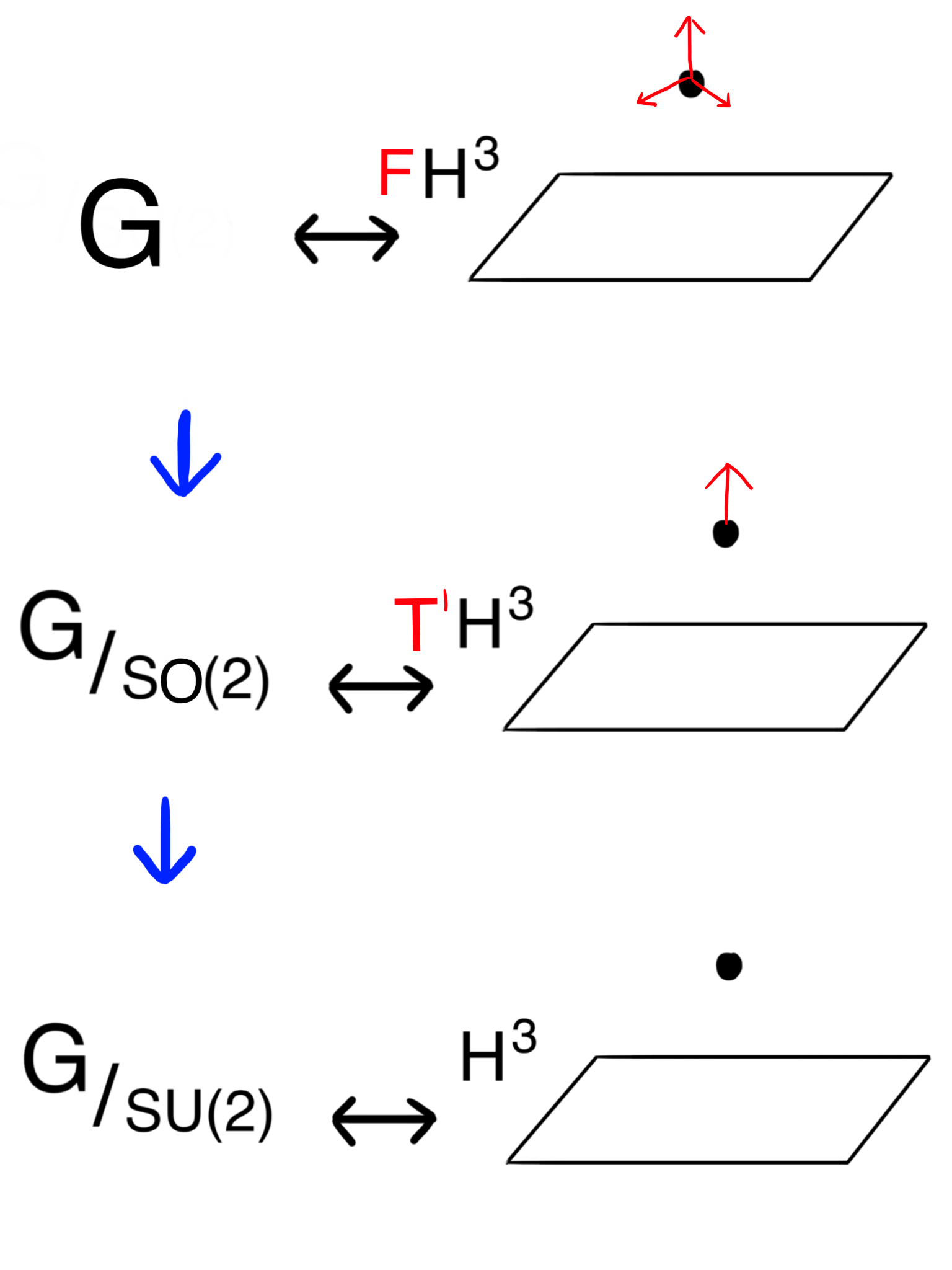}
   \caption{Frame bundle of $\bH^3$}
\end{figure}

Thus we may understand the oriented frame bundle $\FM$ as the homogeneous space $\Gamma\ba G$. Denote by 
$$p: \Gamma\ba G\to M$$ the base-point projection map.

Unless $\Gamma$ is a lattice, the $G$-invariant measure on $\Gamma\ba G$ is infinite, and
dissipative for natural geometric flows such as the geodesic flow and horospherical flow.
 Two locally finite measures on $\Gamma\ba G$, called the Bowen-Margulis-Sullivan measure, and the Burger-Roblin measure,
play important roles, and they are defined using the Patterson-Sullivan density on the limit set of $\Gamma$.

\subsection{Patterson-Sullivan density} We denote by $\delta$ the critical
exponent of $\G$, i.e., the infimum over all $s\ge 0$ such that the Poincare series
$\sum_{\gamma\in \G}e^{-s d(o, \gamma(o))}$ converges. 
As $\G$ is geometrically finite,
 $\delta$ is equal to the Hausdorff dimension of $\Lambda$ \cite{Su}.

 Bishop and Jones proved that $\delta$ is strictly bigger than $1$, unless $\La$  is totally disconnected or
 contained in a circle \cite{BJ}. As  $\Gamma$ is assumed to be Zariski dense, we have:
\begin{thm} 
If $\Lambda$ is connected, then $\delta>1$.
\end{thm}

Recall that  for $x,y\in \bH^3$ and $\xi\in \bS^2$, the Busemann function $\beta_{\xi}(x,y)$ is given by
$\lim_{t\to \infty} d(x, \xi_t) -d(y, \xi_t)$
where $\xi_t$ is a geodesic ray toward $\xi$.

\begin{dfn} \rm A
$\G$-invariant conformal density of dimension $s\ge 0$ is
a family $\{\mu_x: x\in \bH^3\}$ of finite  measures on  $\bS^2$ satisfying:
\begin{enumerate}
 \item For any $\gamma\in \G$ and $x\in \bH^3$, $\gamma_*\mu_x=\mu_{\gamma(x)}$;
\item For all $x,y\in \bH^3$ and $\xi\in \bS^2$, $\frac{d\mu_x}{d\mu_y}(\xi)=e^{s \beta_{\xi}(y,x)}$.
\end{enumerate}
\end{dfn}

\begin{thm}[Patterson-Sullivan] 
There exists a $\G$-invariant conformal density $\{\nu_x: x\in \bH^3\}$ of dimension $\delta$, unique up to a scalar multiple.
\end{thm}
We call this Patterson-Sullivan density. Denoting by $\Delta$ the hyperbolic Laplacian on $\bH^3$,
the Patterson-Sullivan density is closely related to the bottom of the spectrum of $\Delta$ for its action on smooth
functions on $\G\ba \bH^3$.  The function $\phi_0$ defined by $$\phi_0(x):=|\nu_x|$$ for each $x\in \bH^3$
is $\Gamma$-invariant, and hence we may regard $\phi_0$ as a function on the manifold $\Gamma\ba \bH^3$.
It is the unique function (up to a constant multiple)
satisfying  $\Delta \phi_0 =\delta (2-\delta)\phi_0$;
so we call $\phi_0$ the base eigenfunction.

Set $\nu:=\nu_o$ and call it the Patterson-Sullivan measure (viewed from $o$).
When $\Gamma$ is convex-cocompact,
the Patterson-Sullivan measure $\nu_o$
is simply proportional to the $\delta$-dimensional Hausdorff measure on $\Lambda$ in the spherical metric of $\bS^2$.

\subsection{Mixing of the BMS measure}
Consider the following one-parameter subgroup of $G$:
 $$ A:=\left\{a_t=\begin{pmatrix}  e^{t/2} & 0\\ 0& e^{-t/2}\end{pmatrix}  :t\in \br\right\}.$$
The right translation action of $A$ on  $\op{F} \bH^3= G$ induces  the frame flow:
if $g=(e_1, e_2,e_3)$, then $g a_t$ for $t>0$ is the frame given by translation in  direction of $e_1$ by hyperbolic distance $t$.
Let $v_o^{\pm}\in \bS^2$ denote  the forward and backward end points of the geodesic given by $v_o$ respectively.
In the upper half space model of $\bH^3$,  choosing $v_o$ to be the upward normal vector at $o=(0,0,1)$,
we have $v_o^+=\infty$ and $v_o^-=0$.

For $g\in G$, we define $$g^+= g(v_o^+) \in \bS^2  \quad \text{and}\quad  g^- =g(v_o^-) \in \bS^2 .$$
 The map $g\mapsto (g^+, g^-, s=\beta_{g^-}(o, g))$
 induces a homeomorphism between $\T^1(\bH^3)$ and $(\bS^2\times \bS^2 -\text{diagonal}) \times \br$; called the Hopf-parametrization.

We define a locally finite measure $\tilde m^{\BMS}$ on $\T^1(\bH^3)=G/M_0$  as follows:
\begin{align*}
d \tilde m^{\BMS}(g)& =
e^{\delta \beta_{g^+}(o, g)}\; e^{\delta \beta_{g^-}(o,g) }\;d\nu(g^+) d\nu(g^-) ds
\end{align*}
where  $ds$ is the Lebesgue measure on $\br$.

Denote by $m^{\BMS}$  the unique $M_0$-invariant measure on $\Gamma\ba G$ which is
 induced by $\tilde m^{\BMS}$; we call this the Bowen-Margulis-Sullivan measure (or the BMS measure for short).
 
 Sullivan showed that $m^{\BMS}$ is a {\it finite} $A$-invariant measure. The following is due to Babillot \cite{Ba}  for $M_0$-invariant functions 
  and to Winter \cite{Win} for general functions:
\begin{thm}\label{im} The frame flow on $(\Gamma\ba G, m^{\BMS})$ is mixing, that is,
for any $\psi_1, \psi_2\in L^2(\G\ba G, m^{\BMS})$,
$$\lim_{t\to\infty}\int_{\G\ba G} \psi_1(ga_t)\psi_2(g)\; dm^{\BMS}(g)=\frac{1}{|m^{\BMS}|} m^{\BMS}(\psi_1)\cdot m^{\BMS}(\psi_2) .$$
\end{thm}

We define  the   renormalized
frame bundle  of $M$ as:
$$\RFM=\{[g]\in \GG: g^{\pm}\in \Lambda\}.$$

This  is a closed $A$-invariant subset of $\Gamma\ba G$
which is precisely the support of $m^{\BMS}$, and an immediate consequence of Theorem \ref{im} is the topological mixing
of the $A$-action on $\RFM$: for any two open subsets $\cal O_1, \cal O_2$ intersecting $\RFM$,
$\cal O_1 a_t\cap \cal O_2\ne \emptyset$ for all sufficiently large $|t|$.

\subsection{Essential unique-ergodicity of the BR measure}
We denote by $N:=\{ g\in G: a_{-t}  g a_{t}\to e\text{ as $t\to +\infty$}\}$
the contracting horospherical subgroup for the action of $A$, which is explicitly given as
$$ N =\left\{ u_t=\begin{pmatrix} 1 & t\\ 0 &1\end{pmatrix}:t \in \c\right\}.$$

The projection $\pi(gN)$ in $\bH^3$ is a Euclidean sphere tangent to $\bS^2$ at $g^+$
and $gN$ consists of frames $(e_1, e_2, e_3)$ whose last two vectors $e_2, e_3$ are tangent to $\pi(gN)$.
That $N$ is a contracting horospherical subgroup means geometrically that $\pi(gNa_t)$ for $t>0$ is a Euclidean
sphere based at $g^+$
but shrunk toward $g^+$ by the hyperbolic distance $t$.

We define  $\tilde m^{\BR}$ on $G/M_0=\T^1(\bH^3)$  as follows:
$$d \tilde m^{\BR}(g) =
e^{\delta \beta_{g^+}(o, g)}\; e^{2 \beta_{g^-}(o,g) }\;d\nu(g^+) dg^- ds $$
where $dg^-$  is the  Lebesgue measures on $\bS^2$.
We denote by $m^{\BR}$ the unique $M_0$-invariant measure on $\Gamma\ba G$ which is
 induced by  $\tilde m^{\BR}$. We call this measure the Burger-Roblin measure (or the BR measure for short).
If $\Gamma$ is a lattice, $m^{\BR}$ is simply the $G$-invariant measure. Otherwise
 $m^{\BR}$ is an infinite, but locally finite, Borel $N$-invariant measure whose support is given by
 \begin{equation*}\label{rfpm} \RF_+M:=\{[g]\in \GG: g^{+}\in \Lambda\}=\RFM \cdot N .\end{equation*}

The projection of the BR measure to $M$ is an absolutely continuous measure on 
$M$ with Radon-Nikodym derivative given by $\phi_o$:
if $f\in C_c(\Gamma\ba G)$ is $K$-invariant, then
$$m^{\BR}(f)= \int_{\Gamma\ba G}  f(x) \phi_o(x) 
\, dx$$
where $dx$ is a $G$-invariant measure on $\Gamma\ba G$.
Using Theorem \ref{im}, Roblin and Winter showed the following measure classification of $N$-invariant locally finite measures, extending an earlier work of Burger \cite{Bu}:
\begin{thm} (\cite{Ro}, \cite{Win}) \label{im2}
 Any locally finite $N$-ergodic invariant measure on $\Gamma\ba G$
 is either supported  on a closed $N$-orbit or proportional to $m^{\BR}$.
\end{thm}

\subsection{Closures of $N$-orbits} If $x\notin \RFPM$, then $xN$ is a proper immersion of $N$ to $\Gamma\ba G$ via the map $n\mapsto xn$, and hence
$xN$ is closed.
In understanding the topological behavior of $xN$ for $x=[g]\in \RFPM$, the relative location of $g^+$ in the limit set becomes relevant.
The hypothesis that $\Gamma$ is geometrically finite implies that
any $\xi \in \Lambda$ is either radial (any geodesic ray $\xi_t \in M$  converging to $\xi$ accumulates on a compact subset)
or parabolic (it is fixed by some parabolic element of $\Gamma$).  Since this property is $\Gamma$-invariant, we will
say that $x^+$ is radial (resp. parabolic) if $g^+$ is for $x=[g]$. When $\Gamma$  is convex cocompact, $\Lambda$ consists only of radial limit points.

The topological mixing of the $A$-action on $\RFM$ implies
 the following dichotomy for the closure of an $N$-orbit:
\begin{thm} (\cite{F}, \cite{Win}) \label{nc} For $x\in \RFPM$,
 $xN$ is  closed (if $x^+$ is parabolic) or dense in $\RFPM$ (if $x^+$ is radial).
\end{thm}

\section{Almost all results on orbit closures}
Let  $\Gamma <G=\PSL_2(\c)$ be a Zariski dense geometrically finite  Kleinian group, and  
$M:=\Gamma\ba \bH^3$  the associated hyperbolic $3$-manifold. 

We are mainly interested in the action of the following two subgroups on $\Gamma\ba G$:
 $$H:=\PSL_2(\br) \text{ and }$$
\be \label{uuu} U :=\{u_t=\begin{pmatrix} 1 & t \\ 0 & 1\end{pmatrix}:t\in \br\}.\ee

Any one-parameter unipotent subgroup of $G$ is conjugate to $U$, and
any connected closed subgroup of $G$ generated by unipotent one-parameter subgroups is
conjugate to either $N$, $H$ or $U$. 
We also note that the subgroups $N$, $H$ and $U$ are normalized by the subgroup $A$, which is an important  point for the following discussion,
as the measures $m^{\BMS}$ and $m^{\BR}$ are invariant and quasi-invariant under $A$ respectively.

The first question is whether there exist almost all results for the closures of these orbits for appropriate measures.

We recall:
\begin{thm}[Moore's ergodicity theorem]
Let $\Gamma<G$ be a lattice. For any unbounded subgroup  $W$ of $G$, 
 $xW$ is dense in $\Gamma\ba G$ for almost all $x\in \Gamma\ba G$.
\end{thm}

When $\Gamma$ is geometrically finite but not a lattice in $G$, no  orbit of a proper connected subgroup $W$
 is dense in $\Gamma\ba G$. Moreover, it is easy to verify that if  $\partial(gW)\subset \bS^2$  does not intersect $\La$,  then
 the map $W\to [g]W\subset \Gamma\ba G$ given by $w\mapsto [g]w$ is a proper map, and hence $[g]W$ is closed
 \footnote{For a subset $S\subset G$,
 we use the notation $\partial S$ to denote $\overline{\pi(S)}\cap \bS^2$ under the projection $\pi: \op{F}\bH^3 \to \bH^3\cup \bS^2$.}.

Hence if $W$ has the property that $\partial (gW)=(gW)^+$, for instance, if $W=H$ or $U$,
then the non-trivial  dynamics of the action of $W$ on $\Gamma\ba G$ exists only inside
the closure of ${\RFPM }\cdot W$. 

We will see that $\RFPM\cdot  H$ is always closed;
it is useful to understand the geometric description of $\RFPM \cdot H$ in order to understand its closedness.

\subsection{Geodesic planes and almost all $H$-orbits} A  geodesic plane in $\bH^3$ is a totally geodesic imbedding of $\bH^2$, which is simply
 either a vertical plane or a vertical hemisphere in the upper half space model.

\begin{figure}
 \begin{center}
    \includegraphics[height=6cm]{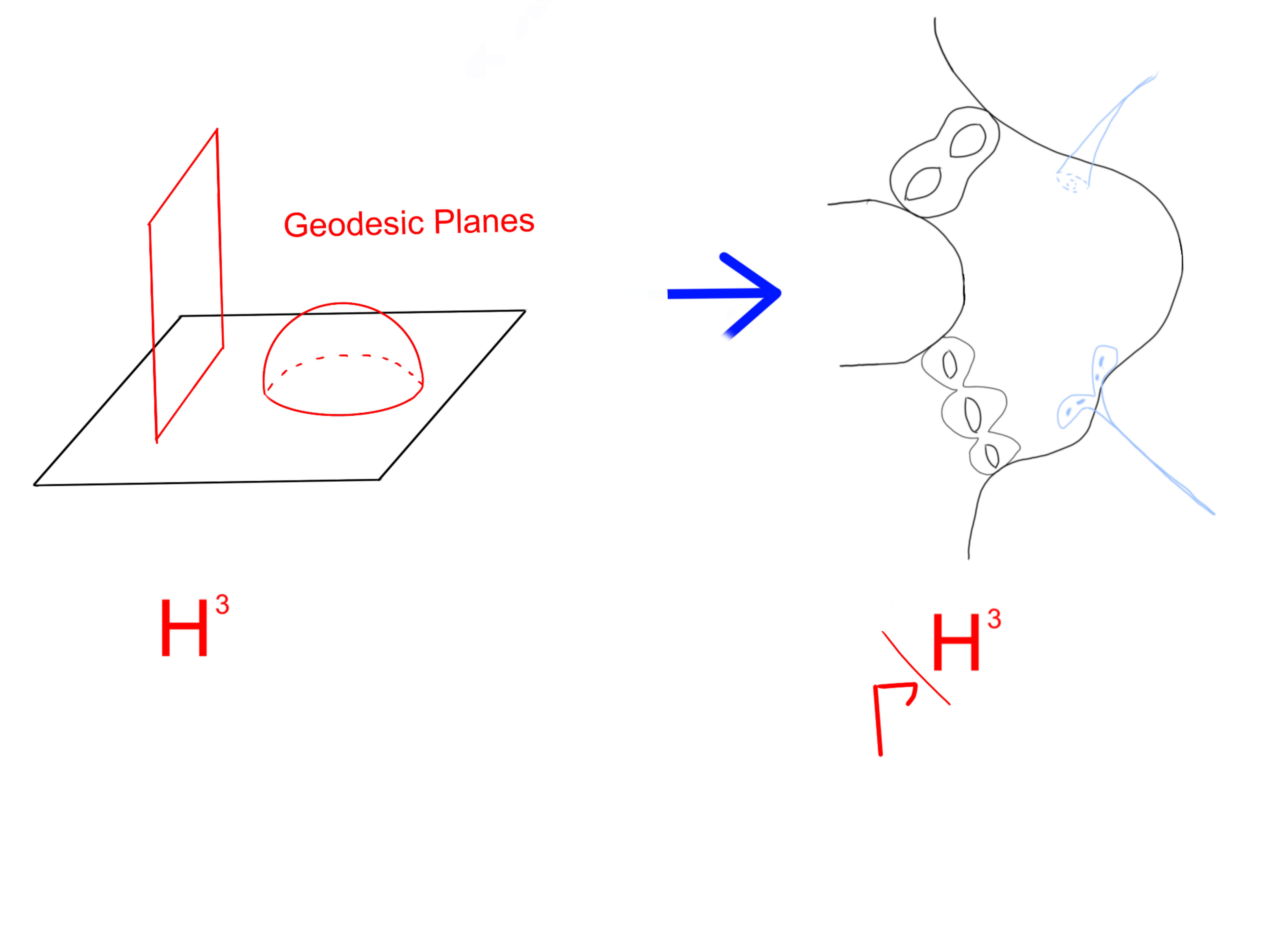}
 \end{center}\caption{Geodesic planes in $M$}
\end{figure}

Let $\mathcal P$ denote the set of all oriented geodesic planes
of $\bH^3$, and  $\mathcal C$ the set of all oriented circles in $\bS^2$.
The map $P\mapsto \partial P$ gives an isomorphism between $\mathcal P$ and $\mathcal C$.

On the other hand, the map
$$gH \mapsto P_g:=\pi(gH)$$
 gives an isomorphism between the quotient space $G/H$ and  the set $\mathcal P$, whose inverse can be described as follows:
for  $P\in \mathcal P$, the set of frames $(e_1, e_2, e_3)$ based in $P$
such that $e_1$ and $e_2$ are tangent to $P$ and $e_3$ is given by the orientation of $P$ is
precisely a single $H$-orbit.  Consequently, the map $$gH\to C_g:=\partial P_g$$
gives an isomorphism between $G/H$ and $\mathcal C$.

\begin{Def} An oriented  geodesic plane $P\subset M$ is a totally geodesic immersion of an oriented hyperbolic plane
$\bH^2$ in $M$, or equivalently, $P$ is the image of an oriented geodesic plane of $\bH^3$ under $\pi$.
\end{Def}
In this paper,  geodesic planes and circles are always considered to be oriented.
Note that any  geodesic plane $P\subset M$ is of the form:
$$P=p (gH)\quad\text{for some $g\in G$.}$$
 Therefore the study of $H$-orbits on $\Gamma\ba G$ has a direct implication on the behavior
of geodesic planes in the manifold $\Gamma\ba \bH^3$.

 We set
\begin{equation}\label{ff} F_\La:=\RFPM \cdot H\quad \text{and}\quad \C_\La:=\{C\in \mathcal C: C\cap \La \ne\emptyset\}.\end{equation}
It follows from the compactness of $\La$ that $\mC_\La$ is a closed subset of $\mathcal C=G/H$.
As 
$$F_\La/H=\Gamma\ba \C_\La ,$$
we deduce:
\begin{lem} The set $F_\La$ is a closed $H$-invariant subset of $\Gamma\ba G$.
\end{lem}

\begin{prop}\label{ah} For $m^{\BMS}$-a.e. $x\in \RFM$,
$$\cl{xH}=F_\La ;$$
in particular, the geodesic plane $p(xH)$ is dense in $M$. \end{prop}
\begin{proof}
We have $\overline{\RFM \cdot U}=\RFPM$ \cite{MS} and hence $\overline{\RFM \cdot H}=F_\La$.
Theorem \ref{im} implies that $m^{\BMS}$ is ergodic, and hence by the Birkhoff ergodic theorem,
for almost all $x$, $xA$ is dense in $\RFM$. Since $A\subset H$,  we deduce
 $$\overline{xH}\supset \overline{\RFM}\cdot H=F_\La.$$
\end{proof}

\begin{figure}\label{ttt5} \begin{center}
 \includegraphics [height=5cm]{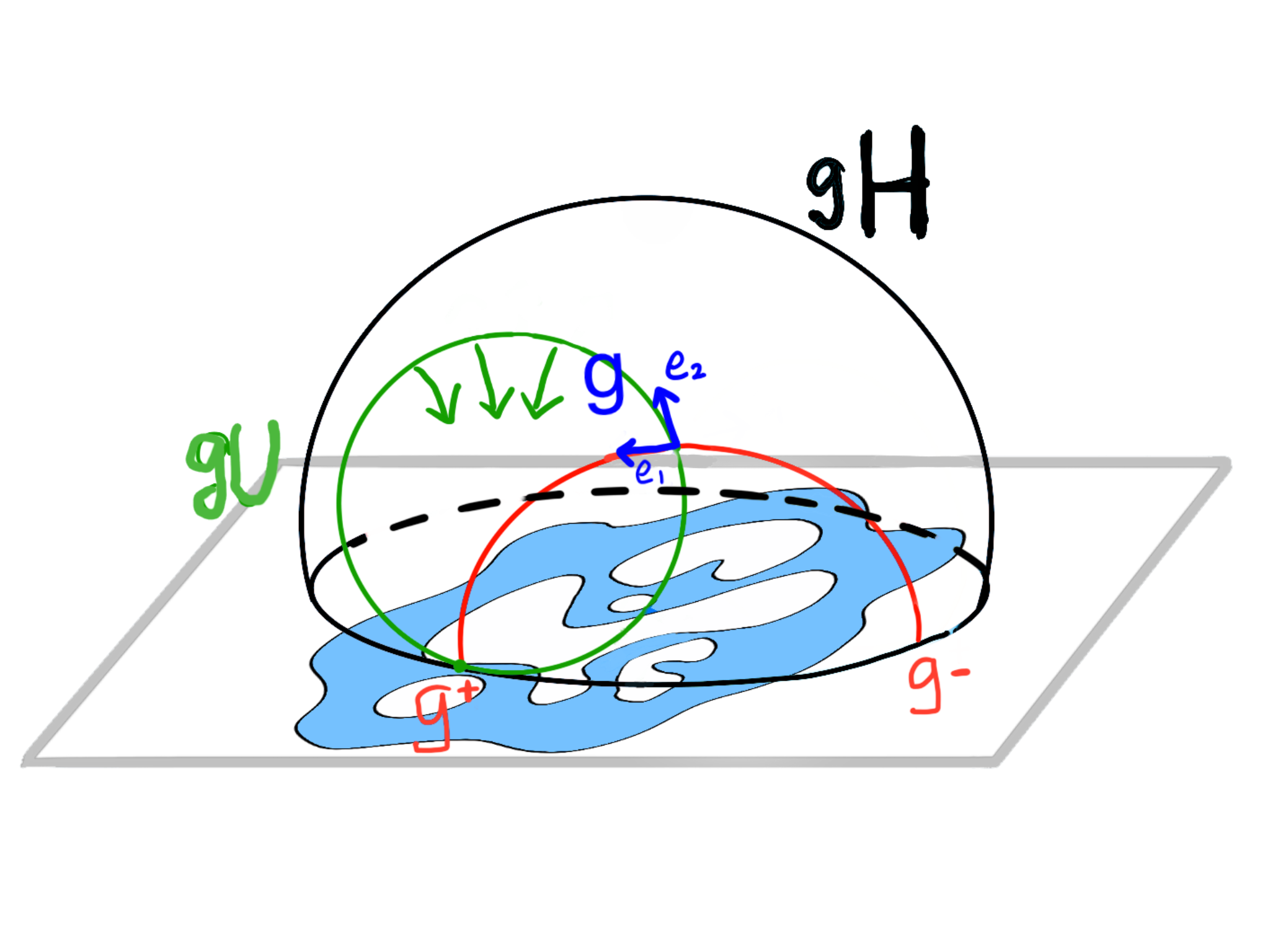}
\end{center}\caption{Orbits under $A$, $U$ and $H$} \end{figure}

 \subsection{Horocycles and almost all $U$-orbits}
 
  A horocycle in $\bH^3$ is a Euclidean circle tangent to $\bS^2$.

\begin{Def}   A  horocycle $\chi$ in $M$ is  an isometrically immersed copy of $\br$ with zero torsion and geodesic curvature $1$, or
equivalently, $\chi$ is the image of a horocycle of $\bH^3$ under $\pi$.
\end{Def}

The right translation action of $U$ on $\Gamma \ba G$ is the horocyclic action: if $g=(e_1, e_2,e_3)$, then $g u_t$ for $t>0$ is the frame given by translation in the direction of $e_2$ by Euclidean distance $t$. 
In fact, any horocycle $\chi\subset M$ is of the form
$$\chi=p(gU) \quad\text{for some $g\in G$.}$$

Note that both $gA$ and $gU$ have their trajectories inside the plane $P_g=\pi(gH)$. In particular, $\pi(gU)$ is a Euclidean circle lying
on $P_g$ tangent  to $\bS^2$ at $g^+$.

We now discuss the almost all results for $U$-orbits in terms of the Burger-Roblin measure.
 It turns out that the size of the critical exponent $\delta$ matters in
this question. The following was proved in joint work with Mohammadi for $\Gamma$ convex cocompact \cite{MO}
and by Maucourant and Schapira \cite{MS}  for geometrically finite groups.
\begin{thm} \label{erg}  If $\delta>1$, $m^{\BR}$ is $U$-ergodic and conservative. 
\end{thm}

\begin{prop} \label{au} Let $\delta >1$ (e.g., $\La$ is connected).
Then for $m^{\BR}$-a.e. $x\in \RFPM$,
$$\cl{xU}=\RF_+M ;$$ In particular, the horocycle $p(xU)$ is dense in $M$.
\end{prop}
\begin{proof}
Since $m^{\BR}$ is an infinite measure, unless $\Gamma$ is a lattice, the Birkhoff ergodic theorem does not apply.
Instead we use the Hopf ratio theorem which applies by Theorem \ref{erg}, and hence the claim follows.
\end{proof}
In \cite{MS}, it was  proved that if $\delta< 1$, $m^{\BR}$ is totally
$U$-dissipative and hence almost all $U$-orbits
are divergent (cf. \cite{Du}).
Whether $m^{\BR}$ is ergodic or not at $\delta=1$ remains an open question.

\subsection{Orbit closure theorem for lattices}
The almost all results on orbit closures in Propositions \ref{ah} and \ref{au} do not describe the topological behavior
of a given individual orbit.
In the lattice case, we have the following remarkable classification of all possible orbit closures, due to Ratner \cite{R} and Shah \cite{Sht}
independently:
\begin{thm} \label{rrr} Let $\Gamma<G$ be a lattice, and $x\in \Gamma\ba G$. 
\begin{enumerate}  
\item The closure $\overline{xH}$ is either $xH$ or $\Gamma\ba G$.
\item The closure $\cl{xU}$  is either $xU$, $xv^{-1} H v$, for some $v\in N$, or $\Gamma\ba G$.
\end{enumerate}
\end{thm}

This theorem immediately implies the first part of the following theorem; the rest  follows from the results in the same paper loc.cit.
\begin{thm}
If $M$ has finite volume, the closures of a geodesic plane and a horocycle are properly immersed submanifolds of $M$.
Moreover, 
\begin{enumerate} 
\item any properly immersed geodesic plane has finite area;
\item there are at most countably many properly immersed geodesic planes in $M$;
\item any infinite sequence of properly immersed geodesic planes $P_i$  becomes dense in $M$, i.e,
$\lim_{i\to \infty} P_i= M$ \footnote{For a sequence of closed subsets $Y_i$ of a topological space $X$, we write
$\lim_{i\to\infty}Y_i=Y$ if $\limsup_{i\to \infty}Y_i=\liminf_{i\to \infty} Y_i=Y$ }. 
\end{enumerate}
\end{thm}
 The density statement (3) above, which is a topological version of Mozes-Shah theorem \cite{MSh},  implies
that every properly immersed geodesic plane $P$ is topologically isolated, in the sense that there exists an open neighborhood of $P$
which does not contain any other properly immersed geodesic plane in its entirety.

\subsection{Topological obstructions to orbit closure theorem}

In this section, we describe a family of quasi-fuchsian manifolds some of whose geodesic planes have
fractal closures, in particular, they have non-integral dimensions. These geodesic planes pass through the interior
of the convex core of $M$ but their boundaries meet the limit set $\Lambda$ only at two points.

These examples can be seen easily for Fuchsian manifolds, and by performing
 a small bending deformation along a simple closed geodesic far away
from our fractal closures of a fixed plane, we will obtain quasi-fuchsian manifolds keeping the fractal closure intact.

\subsubsection{Fuchsian $3$-manifolds} 
\begin{figure} \begin{center}
 \includegraphics [height=4cm]{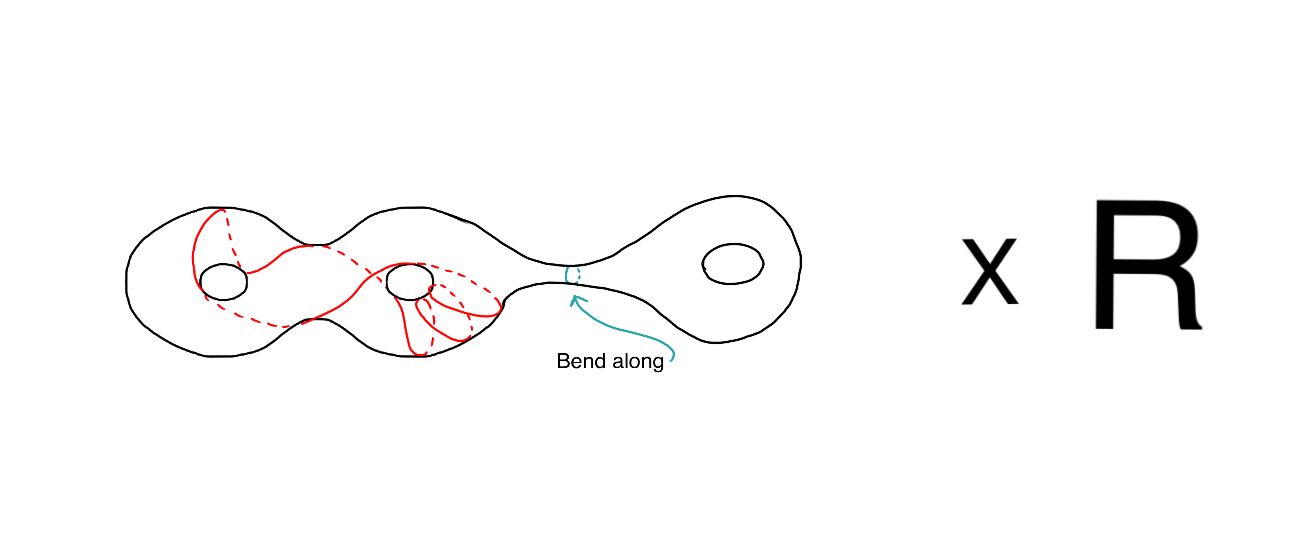}
\end{center}

\caption{Bending deformation}
\end{figure}

Consider a Fuchsian $3$-manifold $M$ which can be expressed
as $$M= S\times \br$$ in cylindrical coordinates where $S$ is a closed hyperbolic surface of genus at least $2$.
Or equivalently, take a torsion-free uniform lattice $\Gamma<\PSL_2(\br)$, and consider $\Gamma$ as a subgroup
of $G$, so that  $M=\Gamma\ba \bH^3=\left( \Gamma\ba \bH^2 \right) \times \br$. We have $\core M=S$.

It is well-known that geodesics on a closed hyperbolic surface $S$ can behave as wild as we wish for, in particular, for any $\beta\ge 1$,
there exists a geodesic whose closure has Hausdorff dimension precisely $\beta$.
\begin{enumerate}
\item The closure of a geodesic plane needs not be a submanifold:
if $\gamma \subset S$ is a geodesic and $P$ is a geodesic plane orthogonal to $S$ with $P\cap S=\gamma$, then
$$\overline{P}\simeq \overline{\gamma}\times \br.$$  Therefore if we take a geodesic $\gamma\subset S$ whose closure $\overline{\gamma}$ is wild, then $\overline P$ is very far from being a submanifold. 

\item There are uncountably many properly immersed geodesic planes  intersecting $\core M$; if $\gamma\subset S$ is a closed geodesic  and $P$ is a geodesic plane with $P\cap S=\gamma$, then
 $P$ is a properly immersed geodesic plane. By varying angles between $P$ and $S$, we obtain a continuous family of such $P$.
\end{enumerate}

We can now use a small bending deformation of $M$ to obtain quasifuchsian manifolds
in which the same phenomenon persists.

\subsubsection{Quasifuchsian hyperbolic $3$-manifolds}

Let $\gamma_0\in \Gamma$ be a primitive hyperbolic element representing a separating simple closed geodesic $\beta$ in $S$.
Without loss of generality, we assume $\gamma_0\in A$, up to conjugation.
If $S_1$ and $S_2$ are components of
$S-\beta$, then each $\Gamma_i:=\pi_1(S_i)$  is a subgroup of $\Gamma$  
and
$\Gamma$ can be presented as the amalgamated free product
$$\Gamma=\Gamma_1*_{\langle \gamma_0\rangle} \Gamma_2 .$$
Setting $m_\theta=\text{diag}(  e^{i\theta}, e^{-i\theta})$, note that $m_\theta$ centralizes $\gamma_0$.
For each non-trivial $m_\theta$, we have
$\Gamma_1\cap m_\theta^{-1} \Gamma_2 m_\theta=\langle \gamma_0\rangle$
and the map which maps $\gamma$ to  $\gamma$ if $\gamma\in \Gamma_1$ and
to $ m_\theta^{-1}\gamma m_\theta$ if $\gamma\in \Gamma_2$ extends to an isomorphism
$\Gamma\to \Gamma_\theta$ where
$$\Gamma_\theta:={\Gamma_1} *_{\langle \gamma_0\rangle} m_\theta^{-1} \Gamma_2 m_\theta .$$

If $\theta$ is sufficiently small, then
\begin{itemize}
\item $\Gamma_\theta$ is a discrete subgroup of $G$;
\item $M_\theta:=\Gamma_\theta\ba \bH^3$
is a quasifuchsian manifold and 
\item there is a path isometric embedding $j_\theta: S\to  \partial \op{core} M_\theta $
such that its image $S_\theta$ is bent with a dihedral angle of $\theta$ along the image of $\beta$ and otherwise totally geodesic.
\end{itemize}

Fix $\e>0$ sufficiently small that $\beta$ has an embedded annular collar neighborhood in $S$ of width $2\e$.
Let $\gamma \subset S_1$ be a geodesic whose closure $\overline \gamma$ is disjoint from a $2\e$-neighborhood
$\cal O (\beta, 2\e)$ of $\beta$. Now if we set
$S_1(\e):=S_1 -\cal O (\beta, 2\e)$, then
there is a unique orientation-preserving isometric immersion
 $$J_\theta: S_1(\e)\times \br \to M_\theta$$ which extends $j_\theta|_{S_1(\e)}$ and sends geodesics normal to $S_1(\e)$ to geodesics normal to
 $j_\theta(S_1(\e))$.
Now, if $\theta$ is small enough (relative to $\e$), then
$$\text{$J_\theta$ is a proper isometric embedding. }$$ This can be proved using the following observation.
Let $\alpha=[a, b_1]\cup [b_1, b_2]\cup \cdots \cup  [b_{n-1}, b_n]\cup [b_n, c]$ be a broken geodesic in $\bH^3$, which is a union of geodesic segments and which bends by angle $0\le \theta<\pi/2$ at each $b_i$'s. Suppose the first and the last segments have length at least $\e>0$
and the rest have length at least $2\e$. Let $P_i$ denote the geodesic plane orthogonal to $[b_i, b_{i+1}]$ at $b_i$.
If $\theta=0$, then the distance among $P_i$'s are at least $\e$.
Now if $\theta$ is small enough so that $\sin (\theta/2)< \tanh \epsilon$, then
the planes $P_i$ remain a positive distance apart,  giving a nested sequence of half-planes in $\bH^3$.
This implies that $J_\theta$ is a proper imbedding.

It now follows that for the plane $P:=\gamma\times \br\subset S_1(\e) \times \br $, its image
$P_\theta:=J_\theta(P)\subset M_\theta $ is an immersed geodesic plane whose closure $\overline P_\theta$ is isometric to $\overline P\simeq \overline \gamma\times \br$.
Therefore by choosing $\gamma$ whose closure is wild, we can obtain a geodesic plane $P_\theta$ of $M_\theta$
with wild closure (cf. \cite{MMO} for more details).

This example demonstrates that  the presence of an essential cylinder in $M$ gives an obstruction to
the topological rigidity of geodesic planes. For the behavior of an individual geodesic plane $P$,
it also indicates that the {\it finite} intersection $\partial P\cap \La$ can be an obstruction.

\section{Unipotent blowup and renormalizations}
The distinguished property of a unipotent flow on the homogeneous space
$\Gamma\ba G$ is the polynomial divergence of nearby points.
Given a sequence $zg_n\in \Gamma\ba G$ where $g_n \to e$ in $G$,
the transversal divergence between two orbits $zg_n U$ and $zU$ can be understood by studying the double coset
$Ug_n U$ in view of the equality:
$$zg_n u_t= z u_s (u_s^{-1} g_n u_t)$$
and the behavior of  rational maps $t\mapsto u_{\alpha_n (t)} g_n u_t$ for certain reparametrizations $\alpha_n:\br \to \br$
so that $\limsup_{n\to \infty} \{ u_{\alpha_n (t)}g_n u_t: t\in \br\} $ contains a non-trivial element of $G-U$.\footnote{If $Q_n$ is a sequence of subsets of $G$, $q\in \limsup_{n\to \infty} Q_n$ if and only if 
every neighborhood of $q$ meets infinitely many $Q_n$,
and $q\in \liminf_{n\to \infty} Q_n$ if and only if every neighborhood of $q$ meets all but finitely many $Q_n$.
If $\limsup_n Q_n=Q_\infty=\liminf Q_n$, then $Q_n$ is said to be convergent and $Q_\infty$ is the limit of $Q_n$ \cite{HY}.}

We denote by $V$ the transversal subgroup 
$$V=\{u_{it}: t\in \br\}$$ to $U$ inside $N$, so that $N=UV$. Note that the normalizer $\op{N}(U)$ of $U$
is equal to $AN$, and the centralizer $\op{C}(U)$ of $U$ is equal to $N$.

\begin{figure}\label{flow} \begin{center}
 \includegraphics [height=6cm]{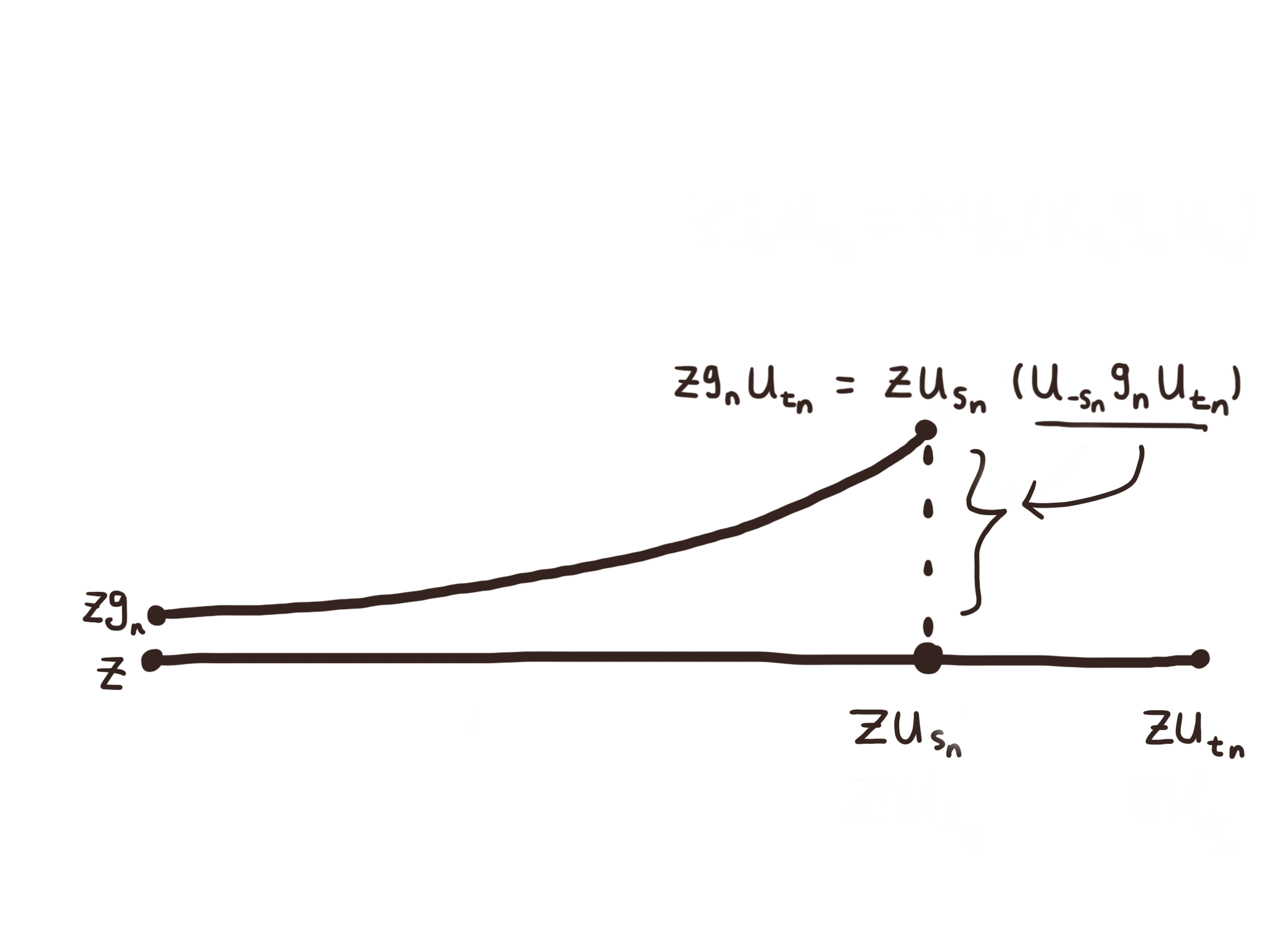}
\end{center}\caption{Divergence of $U$-orbits of two nearby points} \end{figure}

The following unipotent blowup lemma (though stated in the setting of $\SL_3(\br)$) was first observed by Margulis \cite[Lemma 5]{Ma}, in his proof of Oppenheim conjecture.

\begin{lem}\label{ub}
\begin{enumerate}
\item If $g_n\to e$ in $G-AN$, then $\limsup_{n\to\infty} U g_n U$ contains a sequence of elements of $AV-\{e\}$ tending to $e$.
\item If $g_n\to e$ in $G - VH$,  then $\limsup_{n\to\infty} U g_n H$ contains a one-parameter semigroup \footnote{A one-parameter semigroup of $V$ is given by $\{\exp t\xi: t\ge 0\}$ for some non-zero $\xi\in \op{Lie}(V)$}of $V$.
\end{enumerate}
\end{lem}

\subsection{Use of unipotent blowup in the compact $\Gamma\ba G$ case}
In order to demonstrate the significance of this lemma, we present a proof of the following orbit closure theorem,
which uses the notion of $U$-minimal subsets.
A closed $U$-invariant subset $Y\subset \Gamma\ba G$ is called $U$-minimal
if every $U$ orbit in $Y$ is dense in $Y$. By Zorn's lemma, any compact $U$-invariant subset of $\Gamma\ba G$
contains a $U$-minimal subset.
\begin{theorem}\label{pop}
Let $\Gamma <G$ be a uniform  lattice. For any $x\in \Gamma\ba G$, $xH$ is either closed or dense.
\end{theorem}
\begin{proof}
Set $X:=\overline{xH}$. Suppose that $X\ne xH$. By the minimality of the $N$-action on $\Gamma\ba G$ (Corollary \ref{nc}),
it suffices to show that $X$ contains an orbit of $V$.

\noindent{\bf Step 1:}  For any $U$-minimal subset $Y\subset X$,
$$YL= Y\quad\text{ for a one-parameter subgroup $L<AV$.}$$ 
It suffices to show that $Yq_n = Y$ for some sequence $q_n \to e$ in $AV$.
Fix $y_0\in Y$. As $Y$ is $U$-minimal, there exists $t_n \to \infty$ such that $y_0u_{t_n}\to y_0$.
Write $y_0u_{t_n}=y_0g_n$ for $g_n\in G$.
Then $g_n\to e$ in $G-U$, because if $g_n$ belonged to $U$,  the orbit $y_0U$ would be periodic, which is a contradiction
to the assumption that $\Gamma$ is a uniform lattice and hence contains no parabolic elements. 
If $g_n=a_nv_nu_n \in AN=AUV$, then we may take $q_n=a_nv_n$. If $g_n \notin AN$, then by
Lemma \ref{ub}, $\limsup_{n\to \infty} U g_n U$ contains a one-parameter semigroup $L$ of $AV$.
Hence for any $q\in L$, there exist $t_n, s_n\in \br $ such that $q=\lim u_{t_n} g_n u_{s_n}$.
\begin{equation} \label{nes} \text{Since $Y$ is compact, $y_0 u_{-t_n}$ converges to some $y_1\in Y$, }\end{equation}
by passing to a subsequence.
Therefore $y_0 g_n u_{s_n}=y_0 u_{-t_n} (u_{t_n} g_n u_{s_n})$ converges to $y_1 q\in Y$. Since $q\in \op{N}(U)$ and $Y$
is $U$-minimal,
we have $$\cl{y_1qU}=\cl{y_1U} q=Yq=Y.$$
This proves the claim.

\medskip

\noindent{\bf Step 2:}  There exists a $U$-minimal subset $Y\subset X$ 
such that  $X-y_0H$ is not closed for some $y_0\in Y$.

If $xH$ is not locally closed, i.e., $X-xH$ is not closed, then let $Y$ be any $U$-minimal subset  of $X$.
If $Y\subset xH$, then for any $y_0\in Y$, $X-y_0H=X-xH$ is not closed. If $Y\not\subset xH$, then choose $y_0\in Y-xH$.
If $xH$ is locally closed, then let $Y$ be a $U$-minimal subset of $X-xH$. Then $X-y_0H$ is not closed for any $y_0\in Y$.

\medskip

\noindent{\bf Step 3:} For $Y$ from Step (2), we have
$$Yv\subset X\quad\text{ for some non-trivial $v\in V$.}$$ 
By Step (2), we have $y_0g_n \in X$ for some $y_0\in Y$ and a sequence $g_n\to e$ in $G-H$.
If $g_n\in VH$ for some $n$, then the claim follows.
If $g_n\notin VH$ for all $n$, then by Lemma \ref{ub}(2), $\limsup_{n\to \infty} Ug_n H$ contains a non-trivial element $v\in V$.
Since $v=\lim u_{t_n} g_n h_n$ for some $t_n\in \br$ and $h_n\in H$, we deduce $Yv\subset X$ as in Step (1).

\medskip

\noindent{\bf Step 4:} $X$ contains a $V$-orbit.

It suffices to show  that $X$ contains $x_0V_+$ for a one-parameter semigroup $V_+$ of $V$; because
if $v_n\to \infty$ in $V_+$ and $x_0v_n \to x_1$, then 
$$x_1V= x_1\cdot  \limsup_n (v_{n}^{-1}V_+)\subset \overline{x_0V_+} \subset X.$$
Let $Y\subset X$ be a $U$-minimal subset from Step 2. By Step 1, $YL\subset Y$  where
$L$ is either $V$ or $v_0Av_0^{-1}$ for some $v_0\in V$. If $L=V$, this finishes the proof.
If $L=A$, then by Step 3, we get $Yv  =Yv(v^{-1} A v)\subset X$.
Hence we get $X\supset x_0 v^{-1}Av A$ for some $x_0\in X$ and a non-trivial $v\in V$. Since $v^{-1}AvA$ contains a
one-parameter semigroup of $V$, this finishes the proof.
\end{proof}

\begin{figure}
 \begin{center}
    \includegraphics[height=3cm]{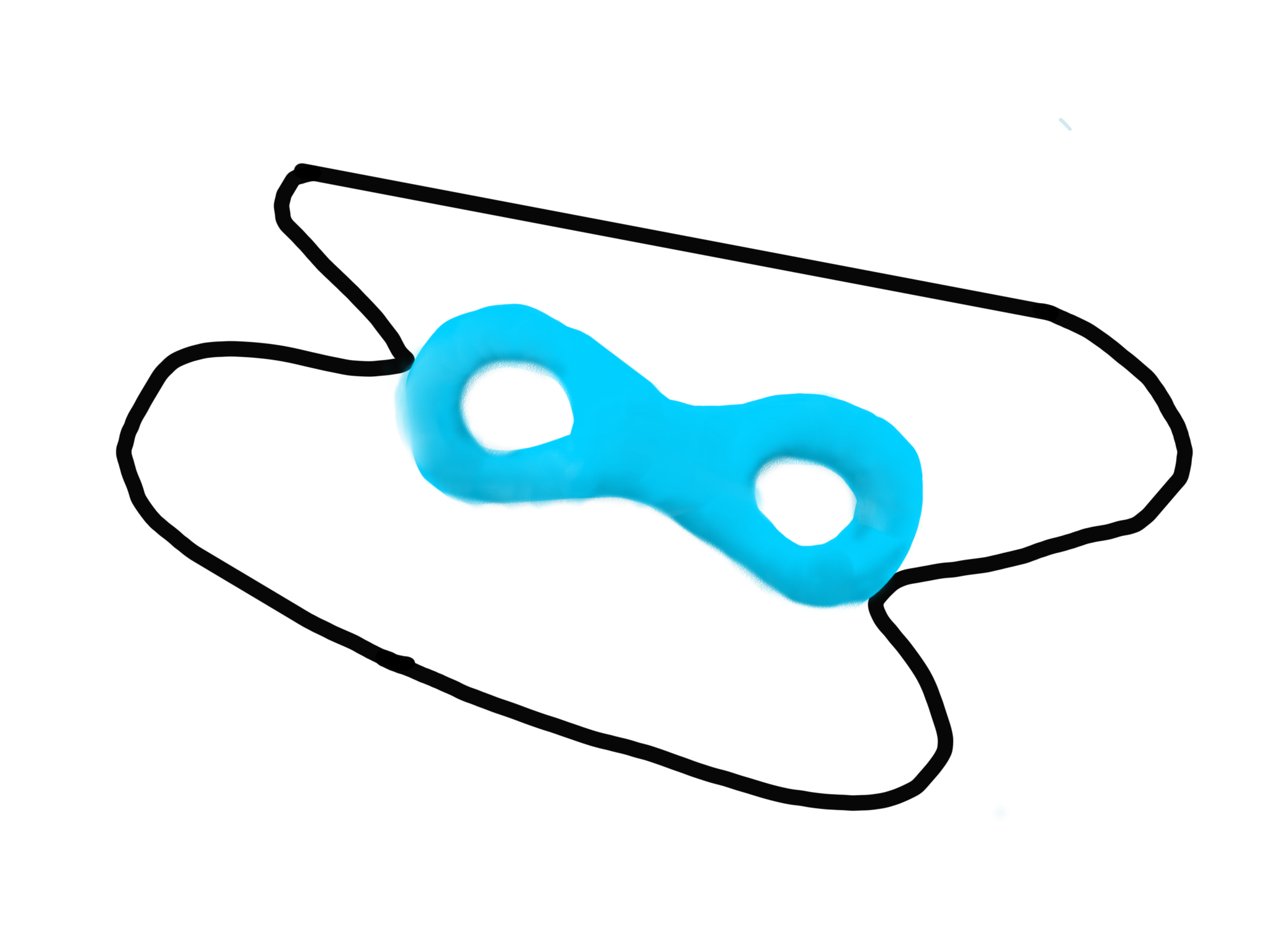} \includegraphics[height=3cm]{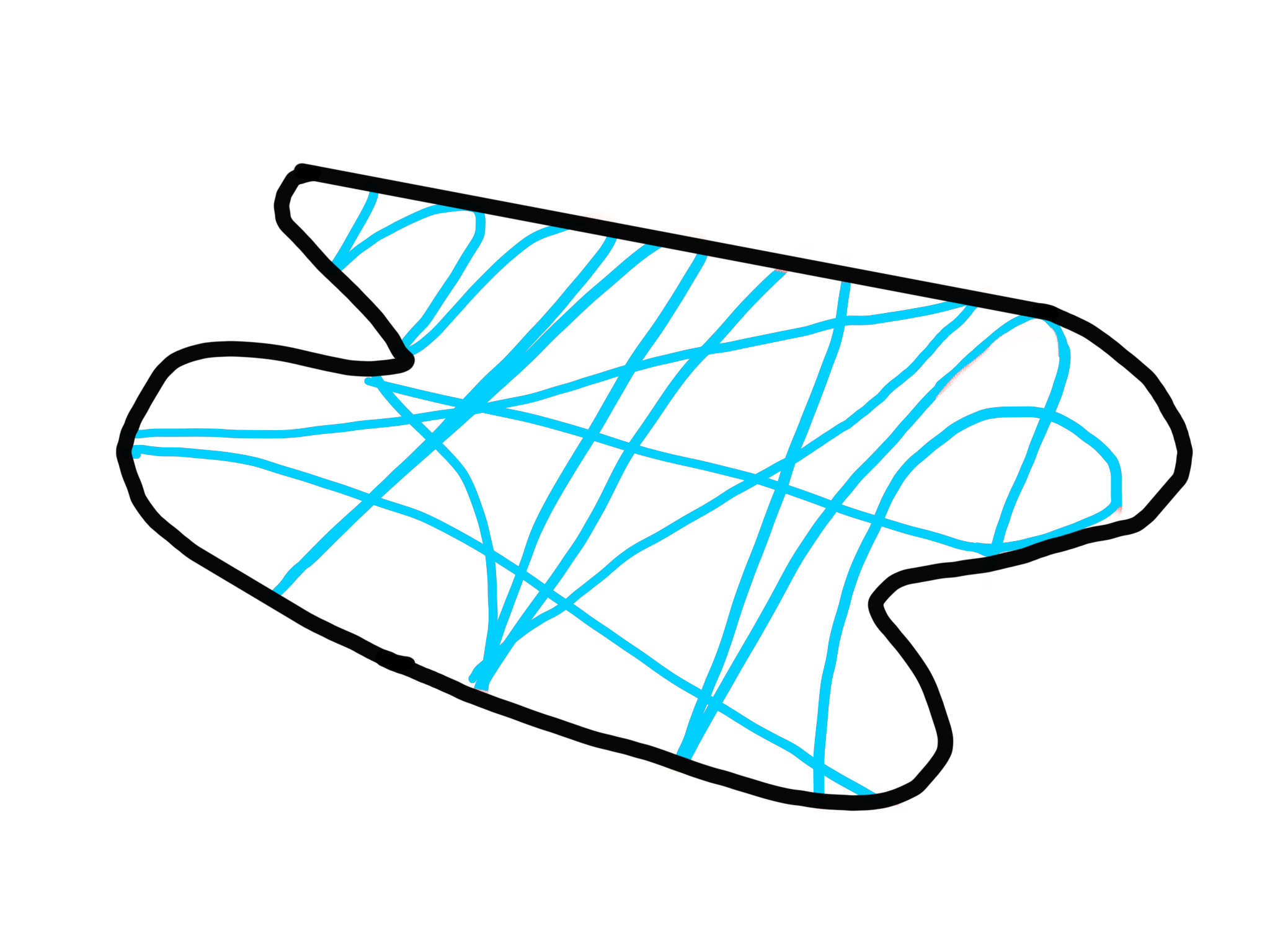}
    \caption{Closed or dense}
 \end{center}
\end{figure}

 We highlight the importance of \eqref{nes} from the above proof:
if  $q$ belongs to the set $ \limsup_{n\to \infty} Ug_nU$ in Lemma \ref{ub}, i.e.,
  $q=\lim_{n\to \infty} u_{t_n} g_n u_{s_n}$ for some $t_n, s_n\in \br$, then
  the size of $t_n$ and $s_n$ are essentially determined by the sequence $g_n \to e$, up to multiplicative constants.
 On the other hand, we need the convergence of the sequence
 $y_0u_{-t_n}$  in order to derive $Yq\subset Y$. That is, if $y_0u_{-t_n}$ diverges, which will be typical when $\Gamma\ba G$ has infinite volume,
Lemma \ref{ub}, whose proof depends on the polynomial property of unipotent action, does not lead anywhere in the study of orbit closure problem.
 
\subsection{Unipotent blowup and renormalizations of the return time} Loosely speaking,  for a given $y_0\in \Gamma\ba G$,
we now would like to understand  the set $$\limsup_{n\to \infty} \mT g_n U$$
 where $\mT$ is the recurrence time of $y_0U$ into a fixed compact subset of $\Gamma\ba G$. 
Most of time, $\limsup \mT g_nU$ may be empty.  In order to make sure that this set is non-trivial enough for our purpose,
we need a certain polynomial $\phi(t)$ (cf. proof of Lemma \ref{U})
 not to vanish  on the renormalized set  $\limsup \lambda_n^{-1} \mT$ where $\lambda_n>0$ is a sequence whose size is dictated by
the speed of convergence of the sequence $g_n\to e$. Since we do not have a control on $g_n$ in general,
the following condition on $\mT$, or more generally on a sequence $\mT_n$,  is necessary for an arbitrary sequence $\lambda_n\to \infty$.
 
 

 \begin{Def} We say that a sequence $\mT_n\subset \br$ has accumulating renormalizations
 if for any sequence $\lambda_n \to \infty$,
 $$\mathsf T_\infty := \limsup\limits_{n\to\infty} \lambda_n^{-1} {\mathsf T}_n$$
accumulates both at $0$ and $\infty$. 
 \end{Def}
That is,  $\mT_\infty$ contains a sequence tending to $0$, as well as a sequence tending to $\infty$.
  We allow a constant sequence $\mT_n$ in this definition.
 

The following lemma is immediate:
 \begin{lem} \label {kappa} If there exists $\kappa>1$ such that each $\mT_n$ is $\kappa$-thick in the sense
 that for all $r>0$, $\mT_n\cap \pm [r, \kappa r]\ne \emptyset$, then
the sequence $\mT_n$ has accumulating renormalizations.  \end{lem}
 
 We now present a refined version of Lemma \ref{ub}, which will be a main tool in the study of $U$-orbits
 in the infinite volume homogeneous space: via the map $t\mapsto u_t$, we identify $\br \simeq  U$.
 
 We write $\frak g= \frak h^\perp\oplus \frak h$ where $\frak h=\frak{sl}_2(\br)$ is the Lie algebra of $H$ and $\frak h^\perp=i\frak{sl}_2(\br)$; note that
$\frak h^\perp$ is $H$-invariant under conjugation.

\begin{lemma}[Unipotent blowup] \label{lem.UgU}\label{U}
Let $\mT_n\subset U $ be a sequence with  accumulating renormalizations.
\begin{enumerate}
\item For  any $g_n\to e$ in $G-AN$, the subset $AV \cap \big(\limsup_{n\to \infty} \mT_n g_n U\big)$
accumulates at $e$ and $\infty$.

\item For any $g_n\to e$ in $G -VH$,
 the subset  $V\cap\big( \limsup_{n\to \infty} \mT_n g_n H \big)$ accumulates at $e$ and $\infty$.
\item For any $g_n \to e$ in $\exp \frak h^\perp - V$,
  the subset  $V\cap \big(\limsup_{n\to \infty} \{u_t g_n u_{-t}: t\in \mT_n \}\big)$ accumulates at $e$ and $\infty$.

\end{enumerate}
\end{lemma}
\begin{proof} For (1), we will find a sequence $\lambda_n \to \infty$ (depending on $g_n$) and
a rational map $\psi: \br \to AV$ such that for $\mT_\infty:=\limsup_{n\to \infty} \lambda_n^{-1} \mT_n$,
\begin{itemize}\item $\psi(\mT_\infty)\subset \limsup_{n\to \infty} \mT_n g_n U$;
\item
$\psi (\mT_\infty)$ accumulates at $e$ and $\infty$.
\end{itemize}

The construction of $\psi$ follows the arguments of Margulis and Tomanov \cite{MaT}.
 Since $U$ is a real algebraic subgroup of $G$, by Chevalley's theorem, there exists an $\br$-regular representation ${{G}}\to \op{GL}(W)$ with a distinguished point $p\in W$ such that $U=\op{Stab}_G(p)$. Then $p{{G}}$ is locally closed, and
\begin{equation}\label{eq.char1}
\op{N}(U)=\{g\in {{G}} : p gu =p g\text{ for all }u\in U\}.
\end{equation}

Set $\mathcal L:=V AMN^+$ where $N^+$ is the transpose of $N$. Then 
 $U\mathcal L$ is a Zariski dense open subset of $G$ and $p\mathcal L$ is a Zariski open neighborhood of $p$ in the Zariski closure
 of $pG$. We choose a norm on $W$ so that $B(p,1)\cap \overline{pG}\subset p \mathcal L$, where
 $B(p, 1)\subset W$ denotes the closed ball of radius $1$ centered at $p$.

Without loss of generality, we may assume $g_n \in U\mathcal L$ for all $n$.
 For each $n$,
define $\tilde \phi_n : \br \to W$  by 
$$\tilde \phi_n(t)=p g_nu_t,$$ 
which is  a polynomial
of degree uniformly bounded for all $n$. 
Define $\lambda_n\ge 0$ by
$$\lambda_n:=\sup \{\lambda\ge 0: \tilde \phi_n [-\lambda, \lambda] \subset B(p,1)\}.$$ 
As $g_n\not\in\op{N}(U)=AN$, $\tilde \phi_n$ is non-constant, and hence $\lambda_n<\infty$. As $\tilde \phi_n (0)=pg_n \to p$, we have $\lambda_n\to\infty$.
We reparametrize $\tilde \phi_n$ using $\lambda_n$:
$$\phi_n(t):=\tilde \phi_n (\lambda_n t).$$

Then for all $n$, $$\phi_n[-1,1]\subset {B(p,1)} .$$  Therefore
the sequence $\phi_n$ forms an equicontinuous family of polynomials, and hence, after passing to a subsequence,
 $\phi_n$ converges to a  polynomial $$\phi:\br \to \overline{pG}\subset W $$ uniformly on every compact subset of $\br$.
 Note that $\phi$ is non-constant, since $\phi(0)=p$ and $\max \|\phi(\pm 1)\|=1$.
 As the map $\rho: \mathcal L\to p\mathcal L$ defined by $\ell\mapsto p\ell$ is a regular isomorphism, 
 and $p\mathcal L$ is a Zariski open neighborhood
  of $p$ in the Zariski closure of $pG$,
 we now get a rational map
 $\psi:\br \to \mathcal L$ given by
 $$\psi (t)= \rho^{-1} ( \phi(t) ).$$

 If we define $\psi_n(t)$ as the unique $\mathcal L$-component of $g_nu_t$ in the $U\mathcal L$ decomposition, that is,
 $g_nu_t= u_{s_n}\psi_n(t)$ for some $s_n\in \br$,  then
$$\psi(t)=\lim_{n\to \infty} \psi_n(\lambda _n t)$$
where  the convergence is uniform on compact subsets of $\br$.
It is easy to check that
 $\op{Im} \psi \subset \op{N}(U)\cap \mathcal L= AV$ using \eqref{eq.char1}.
 
Set \begin{equation*}
\mathsf T_\infty := \limsup\limits_{n\to\infty} \lambda_{n}^{-1} {\mathsf T}_n.
\end{equation*}

By the hypothesis on $\mT_n$, $\mT_\infty$ accumulates at $0$ and $\infty$.
Since $\psi: \br \to AV$ is a non-constant rational map with $\psi(0)=e$, $\psi(\mT_\infty)$ accumulates at $e$ and $\infty$.

Letting $t\in \mT_\infty$, 
 choose a sequence ${t}_n\in {}{\mathsf T}_n$  such that $\lim_{n\to \infty}\lambda_n^{-1} {t}_n = {t}$.
  Since $\psi_n\circ \lambda_n\to \psi$ uniformly on compact subsets,
\begin{equation*}
\psi({t})=\lim_{n\to\infty} (\psi_n\circ \lambda_n )\left( \lambda_n^{-1}{t_n}\right)=\lim_{n\to \infty} \psi_n({t}_n)
=\lim_{n\to\infty}u_{s_n}g_nu_{{t}_n}.
\end{equation*}
for some sequence $s_n\in \br$.
Hence
$$\psi(\mT_\infty) \subset \limsup_{n\to \infty} Ug_n \mT_n.$$
By applying this argument to $g_n^{-1}$, we may switch the position of $U$ and $\mT_n$, and hence finish the proof of (1). 

To prove (2), 
by modifying $g_n$ using an element of $H$, we may assume that $g_n=\exp q_n\in \exp \mathfrak{h}^\perp -V$. Hence
(2) follows from (3).  We define
a polynomial $\psi_n : \br \to \mathfrak{h}^\perp$ by
\begin{equation*}
\psi_n(t)= u_{t} q_nu_{-t} \quad\text{ for all ${t}\in \br$}.
\end{equation*}
Since  $g_n\notin V$ and hence does not commute with $U$,  $\psi_n $ is a nonconstant polynomial.
Define $$\lambda_n :=\sup \{ \lambda\ge 0:  \psi_n ([-\lambda, \lambda]) \subset B(0, 1)\}$$
where $B(0,1)$ is the closed unit ball around $0$ in $\mathfrak h^\perp$. Then $0<\lambda_n< \infty$ and $\lambda_n\to \infty$.

Now the rescaled polynomials 
 $\phi_n=\psi_n\circ \lambda_n: \br \to \mathfrak{h}^\perp $
 form  an equicontinuous family of polynomials of uniformly bounded degree and
 $\lim_{n\to \infty} \phi_n(0)= 0$. Therefore
$\phi_n$  converges
 to a  non-constant polynomial $$\phi : \br \to \mathfrak{h}^\perp$$ uniformly on compact subsets.

We claim that $\op{Im}(\phi)\subset \op{Lie}(V)$.
For any fixed ${s}, t\in \br$, 
\begin{align*} u_{s} \phi({t}) u_{-s} & =\lim_{n\to \infty} u_{\lambda_n t +s} q_n u_{-\lambda_n{t} -s }\\
&=
\lim_{n\to \infty} u_{\lambda_n ( t + \lambda_n^{-1} s)} q_n  u_{-\lambda_n ({t} +\lambda_n^{-1} s )}
\\&=
\lim_{n\to \infty} u_{\lambda_n t}  q_n  u_{-\lambda_n {t}}
=\phi(t).
\end{align*}

Hence $ \phi({t})$ commutes with $U$. 
Since the centralizer of $U$ in $\mathfrak h^\perp$ is equal to $\op{Lie}V$,  the claim follows.
Define $\psi:\br \to V$ by $\psi(t)=\exp (\phi(t))$, noting that $\exp:
\op{Lie}V\to V$ is an isomorphism.
Setting
\begin{equation*}
\mathsf T_\infty:=\limsup\limits_{n\to\infty}\lambda_n^{-1} \mathsf T_n,
\end{equation*}
we deduce that $\psi(\mT_\infty)$ accumulates at $e$ and $\infty$.
For any ${t}\in \mathsf T_\infty$, we choose ${{t}_n}\in \mathsf T_n$ so that
$t=\lim \lambda_n^{-1}t_n$. Then
$$
\psi( t)=\lim_{n\to\infty} {u_{{t}_n}} g_n {u_{-{t}_n}}.
$$
as $\phi_n (t) \to \phi (t) $ uniformly on compact subsets.
Hence,
$$ \psi(\mT_\infty) \subset V\cap \big( \limsup_{n\to \infty} \{u_t g_n u_{-t}: t\in \mT_n \} \big).$$ This completes the proof of (3). \end{proof}

\subsection{Relative minimal sets and additional invariance}
Let $\Gamma <G$ be a discrete subgroup. 
Let $X\subset \Gamma\ba G$ be a closed $H$-invariant subset with no periodic $U$-orbits\footnote{The case when $X$ contains a periodic $U$-orbit turns out to be more manageable; see \cite[Prop. 4.2]{BO}}.
Let 
$W\subset \Gamma\ba G$ be a compact subset such that $X\cap W\ne \emptyset $. We suppose that for any 
$y\in X\cap W$, 
\be\label{ttt} \text{$\mT(y):=\{t\in \br: yu_t\in W\}$
has accumulating renormalizations}.\ee

Under this hypothesis, we can obtain analogous steps to Step (1) and (3) in the proof of Theorem \ref{pop} for  relative 
$U$-minimal subsets of $X$.
Since $X$ is not compact in general, a $U$-minimal subset of $X$ may not exist. Hence we consider  a relative $U$-minimal subset of $X$ instead.

\begin{Def} A closed $U$-invariant subset $Y\subset X$ is called $U$-minimal with respect to $W$, if $Y\cap W\ne \emptyset$ and
 $yU$ is dense in $Y$ for every $y\in Y\cap W$.\end{Def}
As $W$ is compact, it follows from Zorn's lemma that $X$ always contains a $U$-minimal subset with respect to $W$.

\begin{Lem}  [Translates of $Y$ inside of $Y$]  \label{YL2}
Let $Y\subset X$ be a $U$-minimal subset with respect to $ W$.
Then $$YL\subset Y$$ for some one-parameter semigroup $L<AV$.
\end{Lem}
 \begin{proof} It suffices to find a sequence $q_n\to e$ in $AV$ such that  $Yq_n \subset Y$. 
 
 Fix $y_0\in Y\cap W$. We claim that there exists $g_n\to e$ in $G-U$ such that  $y_0g_n\in  Y\cap W$.
By the minimality assumption on $Y$, there exists $t_n\to \infty$ in $\mT(y_0)$ so that
$y_0u_{t_n}$ converges to $ y_0 \in Y\cap W$ (cf. \cite[Lemma 8.2]{BO}). Hence there exists $g_n \to e$
such that $$y_0u_{t_n}=y_0g_n.$$
Then $g_n\notin U$, because if $g_n$ belonged to $U$,  $y_0U$ would be periodic,  contradicting
 the assumption that $X$ contains no periodic $U$-orbit.
  
 {\noindent{\it Case (1): $g_n\in AN$.}} By modifying $g_n$ with elements of $U$, 
 we may assume that $g_n\in AV$. Since $g_n\in \op{N}(U)$ and $y_0\in Y\cap W$,
 we get $y_0Ug_n= y_0g_n U \subset Y$ and hence $\overline{y_0U} g_n= Yg_n\subset Y$.
 
 {\noindent{\it Case (2): $g_n\notin AN$.}} By Lemma \ref{U}, for any neighborhood $\cal O$ of $e$,
 there exist $t_n \in \mT(y_0)$ and $s_n \in \br$ such that $u_{-t_n}g_nu_{s_n}$ converges to some
 $ q\in (AV-\{e\})\cap \cal O$.
 Since
 $y_0u_{t_n}\in W$ and $W$ is compact,
 $y_0u_{t_n}$ converges to some $y_1\in Y\cap W$, by passing to a subsequence.
 Therefore as $n\to \infty$, $$y_0g_n u_{-s_n}= (y_0 u_{t_n}) (u_{-t_n} g_n u_{s_n}) \to y_1 q \in Y .$$
  As $y_1\in Y\cap W$ and $q\in \op{N}(U)$, it follows $Yq\subset Y$.  Since such $q$ can be found in any neighborhood of $e$, this finishes the proof.

 \end{proof}


\begin{Lem} [One translate of $Y$ inside of $X$] \label{Yv}
 Let $Y\subset X$ be a $U$-minimal subset with respect to $W$ such that $X-y_0H$ is not closed for some $y_0\in Y\cap W$.
   Then  $$Yv\subset  X\quad\text{ for some non-trivial $v\in V .$}$$
\end{Lem}

\begin{proof}
By the hypothesis, there exists $g_n\to e$ in $G-H$ such that
 $y_0g_n\in X$.

If $g_n\in VH$ for some $n$, the claim is immediate as $X$ is $H$-invariant.
If $g_n\notin VH$ for all $n$, 
by Lemma \ref{U}, there exist  $t_n\in \mT(y_0)$ and $h_n\in H$
so that $u_{t_n}^{-1} g_n h_n $ converges to some non-trivial $v\in V$.
Since $y_0u_{t_n}$ belongs to the compact subset $W$, by
 passing to a subsequence $y_0u_{t_n}$ converges to some $y_1\in Y\cap W$.
Hence  $y_0g_n h_n =y_0u_{t_n} (u_{t_n}^{-1}g_nh_n)$ converges to $y_1 v$. By the minimality of $Y$ with respect to $W$, we get $Yv\subset Y$, as desired.
\end{proof}

For a subset $I\subset \br$, we write $V_{I}=\{u_{it}: t\in  I\}$.
When the conditions for Lemmas \ref{YL2} and \ref{Yv} are met, we can deduce that $X$ contains some interval of an $V$-orbit:

\begin{lem}\label{qqq} Let $X$ be a closed $H$-invariant subset of $\Gamma\ba G$ containing a compact $A$-invariant subset $W$.
Let $Y\subset X$ be a $U$-minimal subset with respect to $W$.
Suppose 
\begin{enumerate}
\item
$YL\subset Y$ for some one-parameter semigroup $L<AV$; 
\item $ Yv\subset X$ for some non-trivial $v\in V$.
\end{enumerate}
Then 
$X$ contains $x_0V_I$ for some $x_0\in W$ and an interval $0\in I$.
\end{lem}
\begin{proof} 
Any one-parameter semigroup $L<AV$ is either
a one-parameter semigroup $V_+<V$ or $ v_0A_+v_0^{-1}$ for some $v_0\in V$ and a
one-parameter semigroup $A_+<A$.
 
   \medskip

 {\noindent{\it Case}} (a). 
If $L=V_+$, we are done. 
   \medskip

 {\noindent{\it Case}} (b).  If $L=v_0A_+v_0^{-1}$ for a non-trivial $v_0\in V$,
then $$X\supset Y(v_0A_+v_0^{-1})A .$$
Since $v_0A_+v_0^{-1}A$ contains $V_I$ for some interval $0\in I$, the claim follows.

  \medskip

 {\noindent{\it Case}} (c).  If $L=A_+$,
 we first note that $YA\subset Y$; take any sequence $a_n\to \infty$ in $A_+$, and $y_0\in Y\cap W$.
 Then $y_0a_n \in Y\cap W$ converges to some $y_1\in Y\cap W$. Now $\limsup_{n\to \infty} a_{n}^{-1} A_+=A$.
 Therefore $Y\supset y_1A$. Since $\overline{y_1U}=Y$, we get $Y\supset YA$.
Since $AvA$ contains a semigroup $V_+$ of $V$, we deduce  $$X\supset Yv A \supset YA v A\supset YV_+.$$
\end{proof}

In the next section, we discuss the significance of the conclusion that
$X$ contains a segment $x_0V_I$, depending on the relative
location of $x_0$ to  $\partial \core M$.

\section{Interior frames and boundary frames}
Let $\Gamma<G$ be a Zariski dense geometrically finite group, and let $M=\Gamma\ba \bH^3$.
When $M$ has infinite volume, its convex core has a non-empty boundary, which makes the dynamical behavior of a frame
under geometric flows
different depending on its relative position to $\partial \core M$.

 Recall the notation $F_\La$ from \eqref{ff}. We denote by $F^*$ the interior of $F_\La$, and $\partial F$ the boundary of $F_\La$.
We explain that in order to show that a given closed $H$-invariant subset $X\subset F_\La$ with no periodic $U$-orbits is equal to $F_\La$, 
it suffices to show that $X$ contains $x_0V_I$ for some $x_0\in F^*\cap \RFPM$ and an interval $0\in I$ (Lemma \ref{fs}). It is important
to get $x_0\in F^*$, as the similar statement is not true if $x_0\in \partial F$.
For example, in the rigid acylindrical case, if $x_0\in \partial F$, then
$x_0H V_+ H$ is a closed $H$-invariant subset of $\partial F$
for a certain semigroup $V_+<V$ (cf. Theorem \ref{ha} below), and hence if $V_I$ belongs to $V_+$, we cannot use $x_0V_I$ to obtain useful information on $X\cap F^*$.
 

\subsection{Interior frames} In this section, we assume that $\Lambda$ is connected.
Under this hypothesis, the closed $H$-invariant set
$F_\La=\RFPM \cdot H$ has non-empty interior which can be described as follows:
\begin{align*}F^*&=\{[g]\in \Gamma\ba G: \pi(P_g) \cap M^*\ne \emptyset \}\\
&=\bigcup \{xH \subset \Gamma\ba G:{p(xH)} \cap  M^* \ne \emptyset\} \end{align*}
where  $M^*$ denotes the interior of $\core M$.

The condition $\pi(P_g)\cap M^*\ne \emptyset$ is equivalent to the condition that the circle $C_g=\partial P_g$
separates the limit set $\La$, that is, both components of $\bS^2-C_g$ intersects $\La$ non-trivially.
If we set
$$\mC^*:=\{C\in \mC: \text{$C$ separates $\La$}\},$$ 
we have $$F^* /H=\Gamma\ba  \mC^* .$$

We observe that the connectedness of $\La$ implies the following two equivalent statements:
\begin{enumerate}
\item For any $C\in \mC^*$, $\# C\cap \La \ge 2$;
\item $F^*\cap \RFPM \subset \RFM \cdot U$
\end{enumerate}

By the openness of $F^*$ and (2) above, for any $x\in F^*\cap \RFPM$, there exists a neighborhood $\cal O$ of $e$ in $G$
such that \be\label{op} x\cal O \cap \RFPM\subset \RFM \cdot U .\ee


Thanks to this stability, we have the following lemma:
\begin{Lem} \label{fs} Let $X\subset F_\La$ be a closed $H$-invariant subset intersecting $\RFM$ and
 with no periodic $U$-orbits.
If $X$ contains $x_0V_I$ for some $x_0\in F^*\cap \RF_+M$ and an interval $0\in I$, then
 $$X= F_\Lambda.$$
\end{Lem}
\begin{proof} It suffices to find $z_0V$ inside $X$ for some $z_0\in \RFM$ by Theorem \ref{nc}.
  Without loss of generality, we may assume $I=[0,s]$ for some $s>0$. We write $v_t:=u_{it}$.
Since $x_0\in F^*\cap \RF_+M$, there exists $0<\e<s$ such that $x_0 v_\e\in X \cap \RFM \cdot U$ by \eqref{op}.
Hence there exists $x_1\in x_0 v_\e U\cap \RFM\cap X$; so
$x_1 v_\e^{-1} V_I= x_1V_{[-\e, s-\e]}\subset X$.
Since $X$ has no periodic $U$-orbit, $x_1^+$ is a radial limit point of $\Lambda$, and hence there exists $t_n\to +\infty$ such that
$x_1a_{t_n}$ converges to some  $ z_0\in \RFM$.
Since $$\limsup_{n\to \infty} a_{t_n}^{-1} V_{[-\e, s-\e]} a_{t_n} =V$$
and $x_1  V_{[-\e, s-\e]}a_{t_n}= x_1a_{t_n} (a_{t_n}^{-1} V_{[-\e, s-\e]}a_{t_n}) \subset X$,
we obtain $z_0 V\subset X$ as desired.
\end{proof}

\subsection{Boundary frames} The geometric structure
of the boundary $\partial F=F_\La -F^*$ plays an important (rather decisive) role in the rigidity study.
For instance, unless $xH$ is bounded, $xH$ is expected to accumulate on $\partial F$. In the most dramatic situation,
all the accumulation of $xH$ may fall into the boundary $\partial F$ so that $\overline{xH}\subset xH\cup \partial F$. Unless
we have some analysis on what  possible closed $H$-invariant subsets of $\partial F$ are, there isn't too much more
we can say on such situation.

A geodesic plane $P\subset \bH^3$ is called a supporting plane if
it intersects $\hull (\La)$ and one component of $\bH^3-P$ is disjoint from $\hull (\La)$, or equivalently,
the circle $C=\partial P$ is a supporting circle in the sense that $\#C\cap \La \ge 2$ and $C$ does not separate $\La$.

For $C\in \cal C$, we denote by $\Gamma^C$  the stabilizer of $C$ in $\Gamma$.
The theory of bending laminations yields:
\begin{thm}\cite[Theorem 5.1]{MMO1} \label{mmo}
For any supporting circle $C\in \mC$,
\begin{enumerate}
\item $\Gamma^C$ is a finitely generated Fuchsian group;
\item there exists a finite subset
$\La_0\subset C\cap \Lambda$ such that
$$C\cap \La=\La(\Gamma^C)\cup \Gamma^C \La_0$$
where $ \La(\Gamma^C)$ denotes the limit set of $\Gamma^C$.
\end{enumerate}
\end{thm}

\begin{Def} We call $x\in \partial F$ a boundary frame, and  call $x=[g]\in \partial F$ a {\it thick} boundary frame 
if there exists  a supporting circle $C$ with  non-elementary stabilizer $\Gamma^C$ such that
 $C_g=C$ or $C_g$ is tangent to $C$ at  $g^+\in \Lambda(\Gamma^C)$.
\end{Def}

\begin{thm}\label{ddd1}  If $x\in \partial F$ is a thick boundary frame such that $xU$ is not closed, then
$\overline{xU}\supset x vA v^{-1}$ for some $v\in V$. If $x\in \RFM$ in addition, then $\overline{xU}\supset x A $.
\end{thm}

\begin{proof} Choose $g\in G$ so that $[g]=x$. By the hypothesis on $x$,
there exists a supporting circle $C$ with $\Gamma^C$ non-elementary and $g^+\in \Lambda(\Gamma^C)$.
The circle $C_g$ is equal to $C$ or tangent to $C$ at $g^+$. It follows that there exists $v\in V$ such that
$C=C_{gv}$.
By Theorem \ref{mmo}, the stabilizer $\Gamma^C$ is finitely generated and non-elementary.
It now follows from a theorem  Dalbo \cite{Da}  that $xvU$ is either periodic (if $g^+=(gv)^+$ is a parabolic fixed point of $\Gamma^C$)
or $\overline{xvU}$ contains $xvH\cap \RFPM\supset xvA$. Since $v$ commutes with $U$, the first claim follows.
If $x\in \RFM$ in addition, then  $C_g$ must be equal to $C$, and hence $v=e$.
\end{proof}

\begin{lem}\label{bd} Let $X\subset F_\La$ be a closed $H$-invariant subset intersecting $\RFM$ and
 with no periodic $U$-orbits.
If $X\cap F^*$ contains $zv_0$ for some thick boundary frame $z\in \partial F\cap \RFM$ and  $v_0\in V-\{e\}$,
then $$X=F_\La.$$
\end{lem}

\begin{proof} By Lemma \ref{fs}, it suffices to find $x_0V_I$ inside $X$ for some $x_0\in F^*$ and an interval $0\in I$.
By Theorem \ref{ddd1}, we have $\overline{z U}\supset z A$.
Therefore $$X\supset \overline{z v_0 UA }= \overline{zUv_0 A }\supset\overline{ z A v_0 A } \supset zV_+$$
where $V_+$ is the one-parameter semigroup contained in $V\cap Av_0A$.
Since  $zv_0\in zV_+\cap F^*\ne \emptyset$, and $F^*$ is open,
$zv_0V_I\subset zV_+ \cap F^*$ for some interval $0\in I$, as desired.
\end{proof}

\section{Rigid acylindrical groups and circular slices of $\La$}

Let $\Gamma <G$ be a rigid acylindrical Kleinian group, and 
$M:=\Gamma\ba \bH^3$  the associated hyperbolic $3$-manifold. 
We assume $\op{Vol}(M)=\infty$.

\subsection{Boundary frames for rigid acylindrical groups}
In this case, we have a complete understanding of the orbit closures in the boundary $\partial F$;
which makes it possible to give a complete classification for {\it all} orbit closures in $F$.

When $\Gamma$ is rigid acylindrical, {\it every} supporting circle $C$ is contained in the limit set, so that $C\cap \La=C$.
It follows that $\Gamma^C$ is a uniform lattice of $G^C$ and the orbit $xH=p(gH)$ is compact whenever $C_g$ is a supporting circle.
This implies:
\begin{thm}\cite{MMO} \label{ha} Let $\Gamma$ be rigid acylindrical, and let $x\in \partial F$ be a boundary frame.
\begin{enumerate}
\item If $x\in \RFPM$, then
$$\overline{xU}=x v H v^{-1} \quad\text{for some $v\in V$.}$$
\item If $x\in \RFM$,
then $xH$ is compact.
\item If $x\in \RFPM -\RFM$,
 there exist a one-parameter semigroup $V_+$ of $V$ and a boundary frame $x_0\in \partial F$ with $x_0H$ compact such that
$$\overline{xH}= x_0HV_+H.$$
\end{enumerate}
\end{thm}

\begin{figure}\label{sc}  \begin{center}
 \includegraphics [height=6cm]{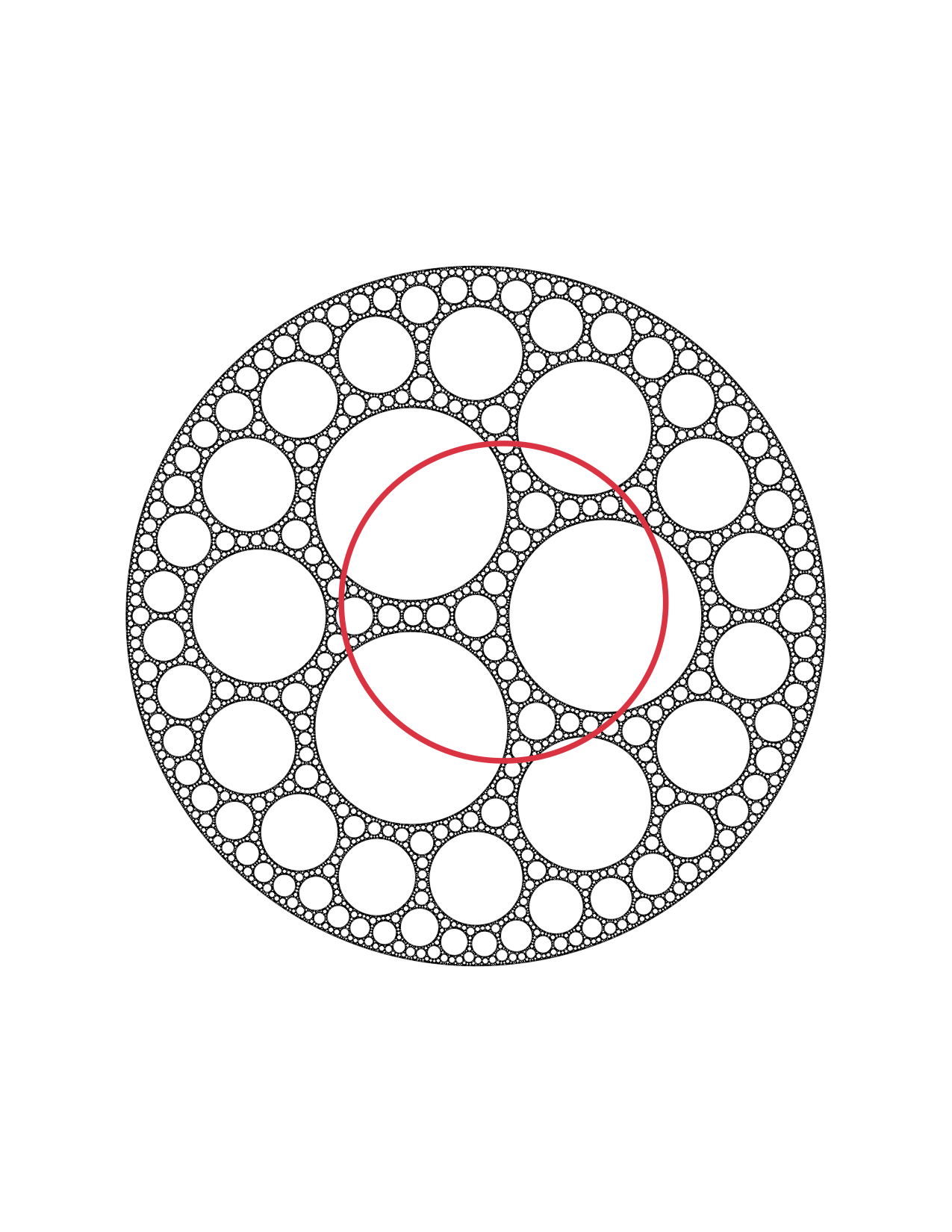}
\caption{Circular slice of $\La$}
\end{center}\end{figure}

\subsection{Circular slices of $\La$}

%
 Circular slices of the limit set $\La$ control the recurrence time of $U$-orbits into the compact subset $\RFM$.
For $x\in \RFM$, set
$$\mT(x):=\{t\in \br: xu_t\in \RFM\} .$$

If $x=[g]$, then $(gu_t)^+=g^+\in C_g\cap \La$ and hence
$$\text{$t\in \mT(x)$ if and only if $(gu_t)^-\in C_g\cap \Lambda$.}$$

We will use the following geometric fact for a rigid acylindrical manifold $M$: if we write $S^2-\Lambda=\bigcup B_i$ where $B_i$'s  are round open disks, then
\begin{equation}\label{sys} \inf_{i\ne j} d(\op{hull} (B_i), \op{hull} (B_j))\ge \e_0\end{equation}
where $2\e_0$ is the systol of  the double of the convex core of $M$.
This follows because a geodesic in $\bH^3$ which realizes the distance $d(\op{hull} (B_i), \op{hull} (B_j))$ is
  either a closed geodesic  in $M$ or the half of a closed geodesic in the  double of $\op{core}(M)$.

\begin{Prop}\label{locally} Let  $\Gamma$ be rigid acylindrical.
There exists $\kappa>1$ 
such that  for all $x\in \RFM$,
$\mT(x)$ is $\kappa$-thick.
In particular, for any sequence $x_i\in \RFM$, $\mT(x_i)$ has accumulating renormalizations.

\end{Prop}

\begin{proof}
For $\e_0>0$ given by \eqref{sys}, consider the upper-half plane model of $\bH^2=\{(x_1,0, y):x_1\in \br, y>0\}$.
For $a<b$, $\hull_{\bH^2}(a, b)\subset \bH^2$ denotes the convex hull of the interval  connecting $(a,0,0)$ and $(b, 0,0)$.
Define  $\kappa>1$ by the equation
$$d_{\bH^2}(\hull (-\kappa, -1), \hull (1, \kappa))=\e_0/2 ;$$
since $\lim_{s\to \infty} d_{\bH^2}(\hull (-s, -1), \hull (1, s))= 0$, such $\kappa>1$ exists.

Since $z\mapsto tz$ is a hyperbolic isometry in $\bH^2$ for any $t>0$, we have
$$d_{\bH^2}(\hull (-\kappa  t, -t), \hull (t, \kappa t))=\e_0/2 .$$

 We now show that $\mT(x)$ is $\kappa$-thick for $x\in \RFM$.  It suffices to show the claim for $x= [g]$ where
 $g=(e_1, e_2, e_3)$ is based at $(0,0,1)$ with $e_2$ in the direction of the positive real axis and
 $g^+=\infty, g^-=0$.
 Note that $gu_t\in \RFM$ if and only if $t=(gu_t)^-\in \Lambda$ and hence
 $$\mT(x)=\br \cap \Lambda .$$
 
 Suppose that $\mT(x)$ is not $\kappa$-thick. Then for some $t>0$,
 $\mT(x)$ does not intersect $[-\kappa t, -t]\cup [t, \kappa t]$, that is, $[-\kappa t, -t]\cup [t, \kappa t]\subset \bigcup B_i$.
 Since $B_i$'s are convex and $0\in \Lambda$, there exist $i\ne j$ such that
 $[-\kappa t, -t]\subset B_i$ and  $[t, \kappa t]\subset B_j$.
 Hence $$d(\hull (-\kappa t, -t), \hull (t, \kappa t))=\e_0/2 \ge  d(\hull(B_i), \hull (B_j))\ge \e_0 ,$$
yielding contradiction.
The second claim follows from Lemma \ref{kappa}.
\end{proof}

\subsection{Closed or dense dichotomy for $H$-orbits}
In Theorem \ref{ha}, we have described all possible orbit closures for $H$ and $U$-action inside $\partial F$.
It remains to consider orbits of $x\in F^*$. 

\begin{thm}\cite{MMO}  \label{cc2} For any $x\in F^*$,
  ${xH}$ is either closed or dense in  $F_\La$.
\end{thm}

\begin{proof} Set $X:=\cl{xH}$, and assume that $X\ne xH$.
We then need to show $X=F_\La$.
Since $F^*\cap \RFPM \subset \RFM \cdot U$, and $xH\subset F^*\cap \RFPM\cdot  H$,
 we may assume without loss of generality that
$x=[g]\in \RFM$.

Set $W:=X\cap F^*\cap \RFM$.

\noindent{\bf Case 1: $W$ is not compact}.
In this case, there exists $x_n\in W$ converging
to some $z\in \partial F\cap \RFM$. 
Write $x_n= zg_n$ with $g_n \to e$ in $G-H$.

Suppose that $g_n =h_nv_n \in HV$ for some $n$.
Since $ zh_n \in zH\subset \partial F\cap \RFM$ and $(zh_n) v_n\in F^*\cap \RFM$,
 the claim follows from Lemma \ref{bd}.

Now suppose that $g_n \notin HV$ for all $n$.
By Lemma \ref{U}, there exist  $t_n \in \mT(x_n) $ and $h_n\in H$ such that
$h_ng_nu_{t_n}$ converges to some $v\in V-\{e\}$. Since $zH$ is compact, $zh_n^{-1}$ converges to some $z_0\in \partial F\cap \RFM$
by passing to a subsequence.
Hence, as $n\to \infty$, $$x_n u_{t_n}= zh_n^{-1} (h_n  g_n u_{t_n}) \to z_0 v .$$
Since $z_0\in \partial F\cap \RFM$ and $z_0v\in \RFM$, we get  $z_0v\in F^*$; hence the claim follows by Lemma \ref{bd}.

\noindent{\bf Case 2: $W$ is compact}.
It follows from the definition of $W$ that for any $x\in  W$, $$\mT(x)=\{t:xu_t\in W\}.$$

\medskip

 We claim that $X$ contains a $U$-minimal subset $Y$ with respect to $W$
such that $X-y_0H$ is not closed for some $y_0\in Y\cap W$.
 We divide our proof into two cases:
   
   \medskip
   {\noindent{\it Case}} (a).  Suppose that $xH$ is not locally closed, i.e., $X-xH$ is not closed. In this case, any $U$-minimal subset
   $Y\subset X$ with respect to $W$ works.
First,  if $Y\cap W \subset xH$, then choose any $y_0\in Y\cap W$. Observe that $\oxH-y_0H=\oxH-xH$ is not closed, which implies the claim.
 If $Y\cap W \not\subset xH$,
  choose $y_0\in (Y\cap W)-xH$. Then $\oxH-y_0H$ contains $xH$, and hence cannot be closed.

   \medskip

  {\noindent{\it Case}} (b). 
Suppose that $xH$ is locally closed, and $X-xH$  intersects $W$ non-trivially.
 Therefore  $X-xH$ contains  a $U$-minimal subset $Y$ with respect to $W$.
Then any $y_0\in Y\cap W$ has the desired  property; since $y_0\in X-xH$, there exists $h_n\in H$
such that $xh_n \to y$. If we write $xh_n=yg_n$, then  $g_n\to e$ in $G-H$, since $y\notin xH$.

\medskip

By Lemmas \ref{YL2}, \ref{Yv} and \ref{qqq}, $X$ contains $x_0V_I$ for some $x_0\in W$ and for an interval $0\in I$; since $x_0\in F^*$,
this finishes the proof by
 Lemma \ref{fs}. \end{proof}

\subsection{Topological rigidity of geodesic planes} 
In (\cite{MMO}, \cite{MMO2}), the following theorem was also obtained:

\begin{thm}\label{mt3}\label{rs} Let $M$ be a rigid acylindrical hyperbolic $3$-manifold.
Then
\begin{enumerate}
\item any geodesic plane $P$ intersecting $\core M$  is either properly immersed or dense;
\item  the fundamental group of a properly immersed $P$  intersecting $\core M$ 
is a non-elementary Fuchsian subgroup;
\item there are at most countably many properly immersed geodesic planes in $M$ intersecting $\core M$;
\item any infinite sequence of geodesic planes $P_i$ intersecting $\core M$  becomes dense in $M$, i.e,
$\lim P_i= M$.
\end{enumerate}
\end{thm}

\begin{Rmk}\label{noo} \rm 
\begin{enumerate}
\item There exists a closed arithmetic hyperbolic $3$-manifold $\Gamma\ba \bH^3$
without any properly immersed geodesic plane, as shown by Maclachlan-Reid \cite{MRR}. 

\item  When $M$ has finite volume and has at least one properly immersed geodesic plane, then
 $M$ is arithmetic if and only if there are infinitely many properly immersed geodesic planes (\cite{MM}, \cite{Fi}).
 
\item 
 A natural question is whether a rigid acylindrical hyperbolic $3$-manifold
 $M$ necessarily covers an arithmetic hyperbolic $3$-manifold if there exists infinitely many properly immersed (unbounded) geodesic planes intersecting
 its core. The reason for the word ``unbounded" in the parenthesis is that
 in any geometrically finite hyperbolic $3$-manifold of infinite volume, there can be only finitely many bounded geodesic planes (\cite{MMO}, \cite{BO}). In view of the proofs given in (\cite{MM}, \cite{Fi}), the measure-theoretic equidistribution of infinitely many closed
 $H$-orbits needs to be understood first.
\end{enumerate}
\end{Rmk}

\subsection{Classification of $U$-orbit closures}
In the rigid acylindrical case, the complete classification of the $U$-orbit closures inside $\partial F$ given in Theorem \ref{ha}
 can be extended to
the whole space $\RFPM$:
\begin{thm} \label{mc} \cite{MMO2} 
For any $x\in \RF_+M$,
$$\overline{xU}=xL\cap \RF_+M$$
where  $L$ is either $v^{-1} Hv $ for some $v\in N$, or $G$.
\end{thm}

There are two main features of  a rigid acylindrical group which our proof is based on.
 The first  property is
that $$\text{ there exists a  compact $H$-orbit in $\RFPM$, }$$
namely those $[g]H$ whose corresponding plane $P_g$ is a supporting plane.
This very important feature of $M$ is a crucial ingredient of our proof. In particular,
the following {\it singular} set  is non-empty:
$$\mS(U)=  \bigcup  zHV  \cap \RFPM .$$
where the union is taken over all closed $H$ orbits $zH$.

 We set $$\mG(U):=\RFPM -\mS(U)$$ and call it the generic set.
 Note that $$\mG(U)\subset F^*.$$

The second property is the following control on the pre-limiting behavior of $\RFM$-points, whose proof is based
on the totally geodesic nature of $\partial \core M$.

\begin{lem}\label{yy} \cite[Lemma 4.2]{MMO2} If $x_n \in F^*$ converges to some $y\in \RFM$,
then there exists a sequence $x_n'\in x_n U\cap \RFM$ converging to $y$, or
converging to some boundary frame $y'\in \partial F\cap \RFM$.
\end{lem}

For $x\in \mS(U)$, Theorem \ref{mc} follows from a theorem of Hedlund \cite{He}
and Dalbo \cite{Da} on the minimality of horocyclic action on the Fuchsian case. 

\begin{prop} \cite{MMO2}\label{quote}
If $x\in \mG(U)$, then
 $$\cl{xU}=\RFPM.$$
\end{prop}

\begin{proof} 

Setting $X:=\cl{xU}$, we first claim that
\begin{equation}\label{co} X\cap \mS(U)\ne \emptyset.\end{equation}

If $X\cap \partial F\ne \emptyset$, the claim follows from Theorem \ref{ha}(1).
Hence we assume that $X\subset F^*$. Let $Y$ be a $U$-minimal subset of $X$ with respect to $\RFM$.
By Lemma \ref{YL2}, $YL\subset Y$ where $L<AV$ is a one-parameter semigroup.
If $L$ is a semigroup of $V$, say, $V_+$, then take a sequence $v_n \to \infty$ in $V_+$. Since $YV^+\subset Y\subset F^*\subset \RFM \cdot U$, up to passing to a subsequence, there exists $y_n\in Y$ such that
$y_nv_n $ converges to an $\RFM$-point, say $y_0$. 
Then $$y_0 V=\lim_{n\to \infty} (y_nv_n) \cdot \limsup_{n} ( v_n^{-1}V_+)\subset  \overline{YV_+}\subset Y.$$
Hence $Y=X=\RFPM$, proving the claim.
If $L=vA_+v^{-1}$ for some semigroup $A_+$ of $A$, since $\mS(U)$ is $V$-invariant, we may assume that $L=A_+$.
Take a sequence $a_n \to \infty$ in $A_+$.  Then for any $y\in Y$,
$ya_n$  converges to an $\RFM$-point, say $y_0\in Y$, by passing to a subsequence.
So $$y_0A=
\lim_{n\to \infty} (ya_n) \cdot \limsup_{n} ( a_n^{-1}A_+)\subset  \overline{YA_+}\subset Y.$$
On the other hand, either $y_0\in \mS(U)$ or $\cl{y_0H}=F_\La$ (Theorem \ref{cc2}).
In the latter case, $\cl{y_0H}$ contains a compact $H$-orbit $zH$.
Since $\cl{y_0AU} M_0 =\overline{y_0H}$,  it follows that  $\cl{y_0AU} \cap zH\ne \emptyset$, proving the claim \eqref{co}.

Therefore $X$ contains $\overline{yU}=yvH v^{-1} \cap \RFPM$ for some $y\in \mS(U)$.
Without loss of generality, we may assume $X\supset yH \cap \RFPM$ by replacing $x$ with $xv$.
Set $Y:=yH\cap \RFPM$, which is a $U$-minimal subset.
There exists $s_n\in \br $ such that $y=\lim_{n\to \infty} xu_{s_n}$.
In view of Lemma \ref{yy}, we may assume that $xu_{s_n}\in \RFM$ for all $n$. Write $xu_{s_n}=yg_n$ for some
sequence $g_n\to e$ in $G$.
Since $y\in \mS(U)$ and $x\in \mG(U)$, it follows that $g_n\notin HV$ for all $n$.
Hence by Lemma \ref{U}, there exist $t_n \in \mT(yg_n)$ and $h_n\in H$ such that
$h_n g_n u_{t_n}$ converges to some $v\in V$; moreover $v$ can be taken arbitrarily large.
By passing to a subsequence, $yg_n u_{t_n}$ converges to some $y_0\in \RFM$, and hence $yh_n^{-1}$ converges
to $y_1:=y_0v^{-1}\in yH\cap \RFPM=Y$.
Therefore $X\supset y_0=y_1 v$ and hence $X\supset Y v$. As $y_1v\in \RFM$, $Yv\cap \RFM\ne \emptyset$.
As $v$ can be taken arbitrarily large, there exists a sequence $v_n\to \infty$ in $V$ such that $X$ contains $Yv_n$.
Choose $y_n\in Y$ so that $y_nv_n\in \RFM$ converges to some $x_0\in \RFM$, by passing to a subsequence.
Since $Y$ is $A$-invariant and $\limsup_{n\to \infty} v_n^{-1} A v_n \supset V$, we deduce
$$X\supset \lim_{n\to \infty} (y_nv_n)\cdot \limsup_{n\to \infty} (v_n^{-1} Av_n) \supset x_0 V.$$
Therefore $X=\RFPM$.
\end{proof}

As an immediate corollary, we deduce:
\begin{cor} \cite{MMO2}  \label{hf} Let $M$ be a rigid acylindrical hyperbolic $3$-manifold.
Then the closure of any horocycle is either a properly immersed surface, parallel to a geodesic plane,
or equal to $M$.
\end{cor}

\subsection{Measure rigidity?}
If there exists a closed  orbit $x H $ for $x\in \RF_+M$,
then the stabilizer of $x$ in $H$ is a non-elementary convex cocompact fuchsian subgroup
and there exists a unique $U$-invariant ergodic measure supported on $xH\cap \RFPM$, called
the  Burger-Roblin measure $m^{\BR}_{xH}$ on $xH$.

\medskip

\noindent{\bf Question:}  Let $M$ be a rigid acylindrical hyperbolic $3$-manifold.
 Is any  locally finite $U$-ergodic measure on $\RF_+M$ either $m^{\BR}$ or $m^{\BR}_{xH}$
for some closed $H$ orbit $xH$, up to a translation by the centralizer of $U$?

\medskip

Theorem \ref{mc} implies the positive answer to this question at least in terms of the support of the measure:
 the support of any locally finite
$U$-ergodic measure on $\RF_+M$ is either $\RF_+M$ or $\RF_+M\cap xHv$
for some closed orbit $xH$ and $v\in N$.

\section{Geometrically finite acylindrical hyperbolic $3$-manifolds}
Let $\Gamma<G$ be a Zariski dense geometrically finite group, and let $M=\Gamma\ba \bH^3$.
We assume $\op{Vol}(M)=\infty$.
In the rigid acylindrical case, we were able to give a complete classfication of all possible
closures of a geodesic plane in $M$; this is largely due to
the {\it rigid} structure of the boundary of $\core M$. In particular, the intersection of a geodesic plane
and the convex core of $M$ is either closed or dense in $\core M$.

 In general, the convex core of $M$ is not such a natural ambient space to study the topological behavior of a geodesic plane, because
 of its non-homogeneity property.
Instead, its interior, which we denote by $M^*$, is a better space to work with;
first of all, $M^*$ is a hyperbolic $3$-manifold with no boundary (although incomplete), which is
diffeomorphic to $M$, and a geodesic plane $P$ which does not intersect $M^*$ cannot come arbitrarily close to $M^*$,
as $\cl P$ must be contained in the ends $M-M^*$.

\begin{Def} A geodesic plane $P^*$ in $M^*$ is defined to be a non-empty intersection $P\cap M^*$
for a geodesic plane $P$ of $M$.\end{Def}
Let $P=\pi(\tilde P)$ for a geodesic plane $\tilde P\subset \bH^3$,
and set $S=\op{Stab}_\Gamma(\tilde P) \ba \tilde P$.
Then the natural map $f: S\to P\subset M$ is an immersion (which is generically injective),
$S^*:=f^{-1}(M^*)$ is a non-empty convex subsurface of $S$ with $\pi_1(S)=\pi_1(S^*)$
and $P^*$ is given as the image of the restriction of $f$ to $S^*$.
The group $\pi_1(S^*)$ will be referred to as the fundamental group of $P^*$.
We note that a geodesic plane $P^*$ is always connected as $P^*$ is covered by the convex subset
$\tilde P \cap \op{Interior} (\text{hull } \Lambda)$.

\subsection{Rigidity of geodesic planes in $M^*$}
An analogous  topological rigidity of planes to Theorem \ref{rs}  continues to hold inside $M^*$, provided
$M$ is a geometrically finite {\it acylindrical} hyperbolic $3$-manifold.

The following rigidity theorem was proved jointly with McMullen and Mohammadi for convex cocompact cases
in \cite{MMO2}, and extended to geometrically finite cases jointly with Benoist \cite{BO}:
\begin{thm}Let $M$ be a geometrically finite acylindrical hyperbolic $3$-manifold.
Then geodesic planes in $M^*$ are topologically rigid in the following sense:
\begin{enumerate}
\item any geodesic plane $P^*$ in $M^*$  is either properly immersed or dense;
\item  the fundamental group of a properly immersed $P^*$  
is a non-elementary geometrically finite Fuchsian subgroup;
\item there are at most countably many properly immersed geodesic planes in $M^*$;
\item any infinite sequence of geodesic planes $P_i^*$  becomes dense in $M^*$, i.e,
$\lim P_i^*= M^*$.
\end{enumerate}
\end{thm}

This theorem is deduced from following results on $H$-orbits in $F^*$:
\begin{thm}\label{mt5} (\cite{MMO2}, \cite{BO})
Let $M$ be a geometrically finite acylindrical hyperbolic $3$-manifold.
Then
\begin{enumerate}
\item any $H$-orbit in $F^*$  is either closed or dense;
\item if $xH$ is closed in $F^*$, then $\op{Stab}_H(x)$ is Zariski dense in $H$;
\item there are at most countably many closed $H$-orbits in $F^*$;
\item any infinite sequence of closed $H$-orbits $x_iH$  becomes dense in $F^*$, i.e,
$\lim x_iH= F^*$.
\end{enumerate}
\end{thm}

\subsection{Closed or dense dichotomy for acylindrical groups}\label{s:dis}
In this section, we discuss the proof of the following
closed or dense dichotomy:

\begin{thm}\label{can}
Let $M$ be a geometrically finite acylindrical hyperbolic $3$-manifold.
Then any $H$-orbit in $F^*$ is either closed or dense.
\end{thm} 

Indeed, the proof of Theorem \ref{cc2} for the rigid acylindrical case can be modified to prove the following proposition.

\begin{prop} [Main proposition] \label{mmm} Let $\Gamma$ be a Zariski dense convex cocompact subgroup of $G$
with connected limit set.
Let $\mathcal R$ be a closed $A$-invariant subset of $\RFM$
satisfying that for any $x\in \mathcal R$,
$\mT(x):=\{t:xu_t\in \R\}$ has accumulating renormalizations.
Then for any $x\in \mathcal R\cap F^*$, $xH$ is either  locally closed  or dense in $F$.
When $xH$ is locally closed, it is closed in $ \mathcal R H \cap F^*$.
\end{prop}

\begin{proof}  

Let $X:=\overline{xH}$ for $x\in \mathcal R\cap F^*$.
Set $W:=X\cap \R\cap F^*$.
 Suppose
that either $xH$ is not locally closed or $(X-xH)\cap W\ne \emptyset$. We claim that $X=F_\La$.

\noindent{\bf Case 1:  $W$ is not compact}.
By repeating verbatim the proof of Theorem \ref{cc2},
we obtain $zv\in X\cap \R$ for some $z\in \partial F \cap \R$ and non-trivial $v\in V$.
As $z=[g]\in \R$, $\Gamma^{C_g}$ is non-elementary and hence
 $z$ is a {\it thick} boundary frame. Since $zv\in F^*$,  the claim follows from Lemma \ref{bd}.

\noindent{\bf Case 2:  $W$ is compact}.
By repeating verbatim the proof of Theorem \ref{cc2},
we show that  $X$ contains a $U$-minimal subset $Y$ with respect to $W$
 such that $X-y_0H$ is not closed for some $y_0\in Y\cap W$.
 Hence by applying Lemmas \ref{YL2}, \ref{Yv}, \ref{qqq} and \ref{fs},
 we get $X=F_{\Lambda}$.
 \end{proof}

When $\Gamma$ is rigid acylindrical, note that $\mathcal R=\RFM$ satisfies the hypothesis of Proposition \ref{mmm}.
In view of this proposition, Theorem \ref{can} for a convex cocompact case now follows from the following theorem, and
a general geometrically finite case can  be proved by an appropriately modified version, taking account of closed horoballs, which
is responsible for the non-compactness of $\RFM$.
\begin{thm}\label{con}
Let $M$ be a geometrically finite acylindrical hyperbolic $3$-manifold.
Then there exists a closed $A$-invariant subset $\R\subset \RFM$ such  that for any $x\in \mathcal R$,
$\mT(x):=\{t:xu_t\in \R\}$ accumulating renormalizations. Moreover $F^*\subset \R H$.
\end{thm}


We use the notion of a conformal modulus in order to find a closed subset $\mathcal R$ satisfying the hypothesis
of Theorem \ref{con}. An annulus $\mathcal A\subset \mathbb S^2$ is an open region whose compliment consists of two components.
If neither component is a single point, $\mathcal A$ is 
conformally equivalent to a unique round annulus
of the form $\{z\in \c: 1<|z|<R\}$. The modulus $\op{mod}(\mathcal A)$ is then defined to be $\log R$.
If $P$ is a compact set of a circle $C$ such that its complement
$C-P=\bigcup I_i$ is a union of at least two intervals with disjoint closures, we define
the modulus of $P$
as $$\mod P:= \inf_{i\ne j} \mod (I_i, I_j)$$
where  $\mod (I_i, I_j):= \mod \big( \mathbb S^2- (\cl I_i \cup \cl I_j)\big)$.

For $\e>0$, define $\mathcal R_\e\subset \RFM$ as the following subset:
$$ \mathcal R_\e:=\{[g]: C_g\cap \La \text{ contains a compact set of modulus $\ge \e$ containing
$g^{\pm}$}\} .$$ 

\begin{lem} For $\e>0$, the set $\mathcal R_\e$ is  closed. \end{lem}

\begin{proof} Suppose that $g_n\in \mathcal R_\e$
converges to some $g\in \RFM$. We need to show $g\in \mathcal R_\e$.
Let $P_n\subset C_{g_n}\cap \Lambda$ be a compact set of modulus
$\ge \e$ containing $g_n^{\pm}$. 
Since the set of all closed subsets of $\mathbb S^2$ is a compact space in the Hausdorff topology on closed subsets,
we may assume $P_n$ converges to some $P_\infty$, by passing to a subsequence. This means
that $P_\infty=\limsup_n P_n=\liminf_{n} P_n$ \cite{HY}.

Write $C_g-P_\infty=\bigcup_{i\in I} I_i$ as the disjoint union of connected components.
As $g^{\pm}\in P_\infty$, $|I|\ge 2$. Let $i\ne j\in I$, and write $I_i=(a_i, b_i)$ and $I_j=(a_j, b_j)$.
There exist $a_{i,n}, b_{i,n}, a_{j, n}, b_{j,n}\in P_n$ converging to
$a_i, b_i, a_j$ and $b_j$ respectively. Set $I_{i,n}$ and $I_{j, n}$ to be the intervals $(a_{i,n}, b_{i,n})$
and $(a_{j,n}, b_{j, n})$ respectively. Since $I_{i,n}\to I_i$, and $I_{j,n}\to I_j$, 
 $I_{i,n}\cup I_{j, n} \subset C_{g_n}-P_n$ for all large $n$. Since $a_{i,n}, b_{i, n}\in P_n$, $I_{i,n}$ is a connected component of 
$C_{g_n}-P_n$. Similarly, 
$I_{j,n}$ is a connected component of 
$C_{g_n}-P_n$.
Since $\mod (I_{i,n}, I_{j, n})\ge \e$ for all $n$, it follows that $I_i$ and $I_j$ have disjoint closures
and $\mod (I_i, I_j)\ge \e$. This shows that $P_\infty$ is a compact subset of $C_g\cap \Lambda$ of modulus at least $\e$
containing $g^{\pm}$. Therefore $g\in \mathcal R_\e$.
\end{proof}

There exists $\kappa=\kappa (\e)>1$ such that for any $x\in \mathcal R_\e$,
$\mT(x):=\{t:xu_t\in \R_\e\}$ is $\kappa$-thick (see \cite[Prop. 4.3]{MMO2}); hence $\mathcal R_\e$ satisfies the hypothesis of Theorem \ref{con}.

 In general, $\mathcal R_\e$ may be empty!
However for geometrically finite acylindrical manifolds,  there exists $\e>0$ such that 
\be \label{f} F^*\subset \mathcal R_\e H
\ee (\cite{MMO2}, \cite{BO}); hence Theorem \ref{can} follows. The inclusion \eqref{f}  is proved using {\it bridge} arguments devised
in \cite{MMO2},
and the monotonicity  of conformal moduli,
 based on the property that for a convex cocompact acylindrical manifold $M$, $\Lambda$
is a Sierpinski carpet of positive modulus, that is,
$$\inf_{i\ne j} \text{mod}(\bS^2- (\overline B_i\cup \overline B_j))>0 $$ 
where $B_i$'s are components of $\bS^2-\La$ (see \cite{MMO2} for details).

When $M$ has cusps, the closures of some components of $\bS^2-\La$ may meet each other, and hence $\Lambda$ is not even a Sierpinski carpet in general.
 Nevertheless, under the assumption that $M$ is a geometrically finite acylindrical manifold,
 $\Lambda$ is still a {\it quotient of a Sierpinski carpet of positive modulus}, in the sense that
  we can present $\bS^2-\La$ as the disjoint union $\bigcup_\ell T_\ell$ where $T_\ell$'s are maximal {\it trees of components of $\bS^2-\Lambda$}
  so that 
  $$\inf_{\ell\ne k}\;\; \text{mod}(\bS^2- (\overline T_\ell \cup \overline T_k))>0.$$

\begin{figure}\label{acp5}  \begin{center}
 \includegraphics [height=4cm]{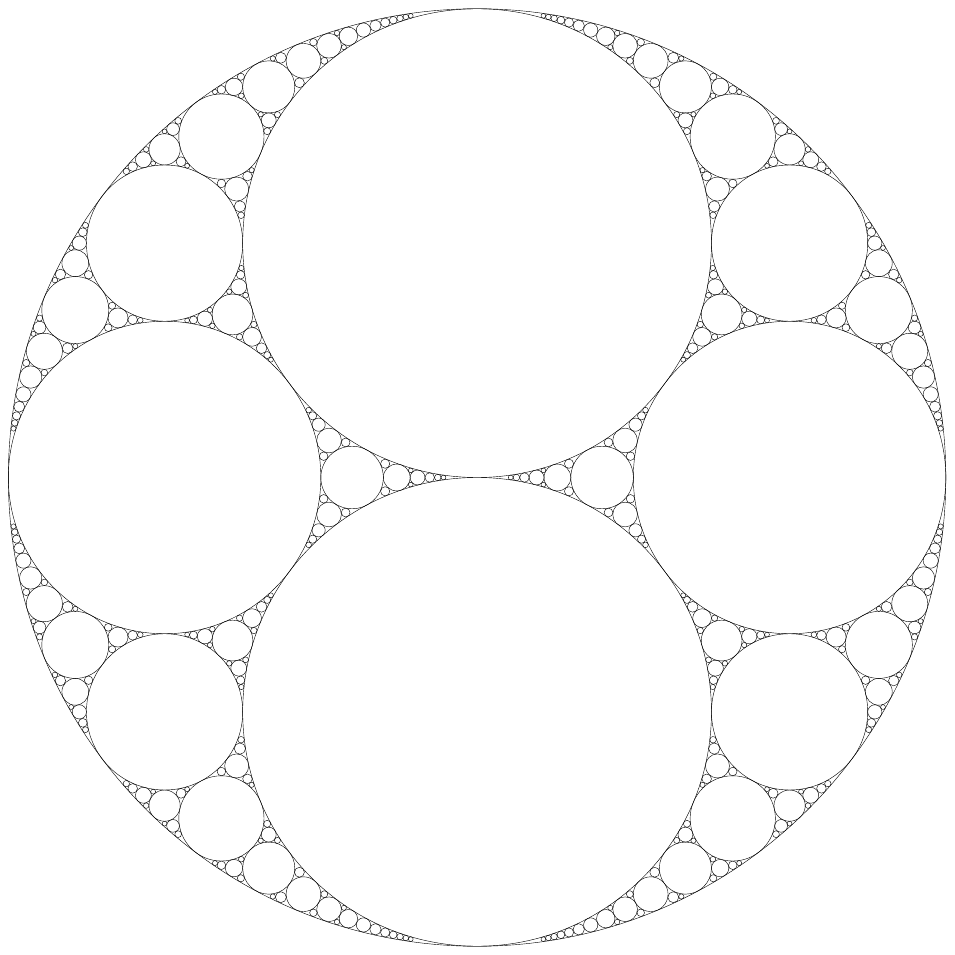}
\caption{Apollonian gasket}
\end{center}\end{figure}

\medskip
\noindent{\bf Question:} Let $\Gamma$ be a Zariski dense geometrically finite subgroup of $G$ with a connected limit set.
Let $C\in \mathcal C^*$. If $C\cap \La$ contains a Cantor set,
\begin{equation*} \label{q} \text{is $\Gamma C$ either discrete or dense in $\mC^*$?}\end{equation*}
If $C\cap \Lambda$ contains a Cantor set of positive modulus, this question has been answered affirmatively in \cite{BO}.


One particular case of interest is when $\La$ is the Apollonian gasket. The correspdoning
geometrically finite hyperbolic $3$-manifold  is not acylindrical in this case, because its compact core is a handle body of genus $2$, and hence it is not boundary incompressible; this can also be seen from the property that the Apollonian gasket contains 
a loop  of three consecutively tangent disks.

\medskip
\noindent{\bf Question:}  Can we classify all possible closures of $U$-orbits in a geometrically finite acylindrical group?
In order to answer this question, we first need to classify all possible $H$-orbit closures in $\partial F$, which is unsettled yet.

\section{Unipotent flows in higher dimensional hyperbolic manifolds}
Let $\bH^d$ denote the $d$-dimensional hyperbolic space for $d\ge 2$ with $\partial (\bH^d)=\bS^{d-1}$, and let
$G:=\SO^\circ (d,1)$, which is the isometry group $\op{Isom}^+(\bH^d)$.
Any complete hyperbolic $d$-manifold is given as the quotient $M=\Gamma\ba \bH^d$  for a torsion-free discrete subgroup $\Gamma <G$
(also called a Kleinian group).
The limit set of $\Gamma$ and the convex core of $M$ are defined just like the dimension $3$ case.  
As we have seen in the dimension $3$ case, the geometry  and topology
 of hyperbolic manifolds becomes relevant in the study of unipotent flows in hyperbolic manifolds of infinite volume,  unlike in the finite volume case. Those hyperbolic $3$-manifolds
in which we have a complete understanding of the topological behavior of unipotent flows are  rigid acylindrical hyperbolic $3$-manifolds.
\subsection{Convex cocompact hyperbolic manifolds with Fuchsian ends}   The higher dimensional analogues of  rigid acylindrical hyperbolic $3$-manifolds are as follows:
 \begin{Def} A convex cocompact hyperbolic $d$-manifold $M$ is said to have Fuchsian ends
 if  the convex core of $ M$ has non-empty interior and has  totally geodesic boundary.
 \end{Def}
The term {\it Fuchsian ends} reflects the fact that each component of the boundary of $\core M$
is a $(d-1)$-dimensional closed hyperbolic manifold, and each component of the complement $M-\core (M)$  is diffeomorphic 
to the product $S\times (0,\infty)$ for some closed hyperbolic $(d-1)$-manifold $S$. 
For $d=2$, any convex cocompact hyperbolic surface has Fuchsian ends. For $d=3$, these are precisely rigid acylindrical hyperbolic $3$-manifolds.

\begin{figure}\label{end1}  \begin{center}
 \includegraphics [height=4cm]{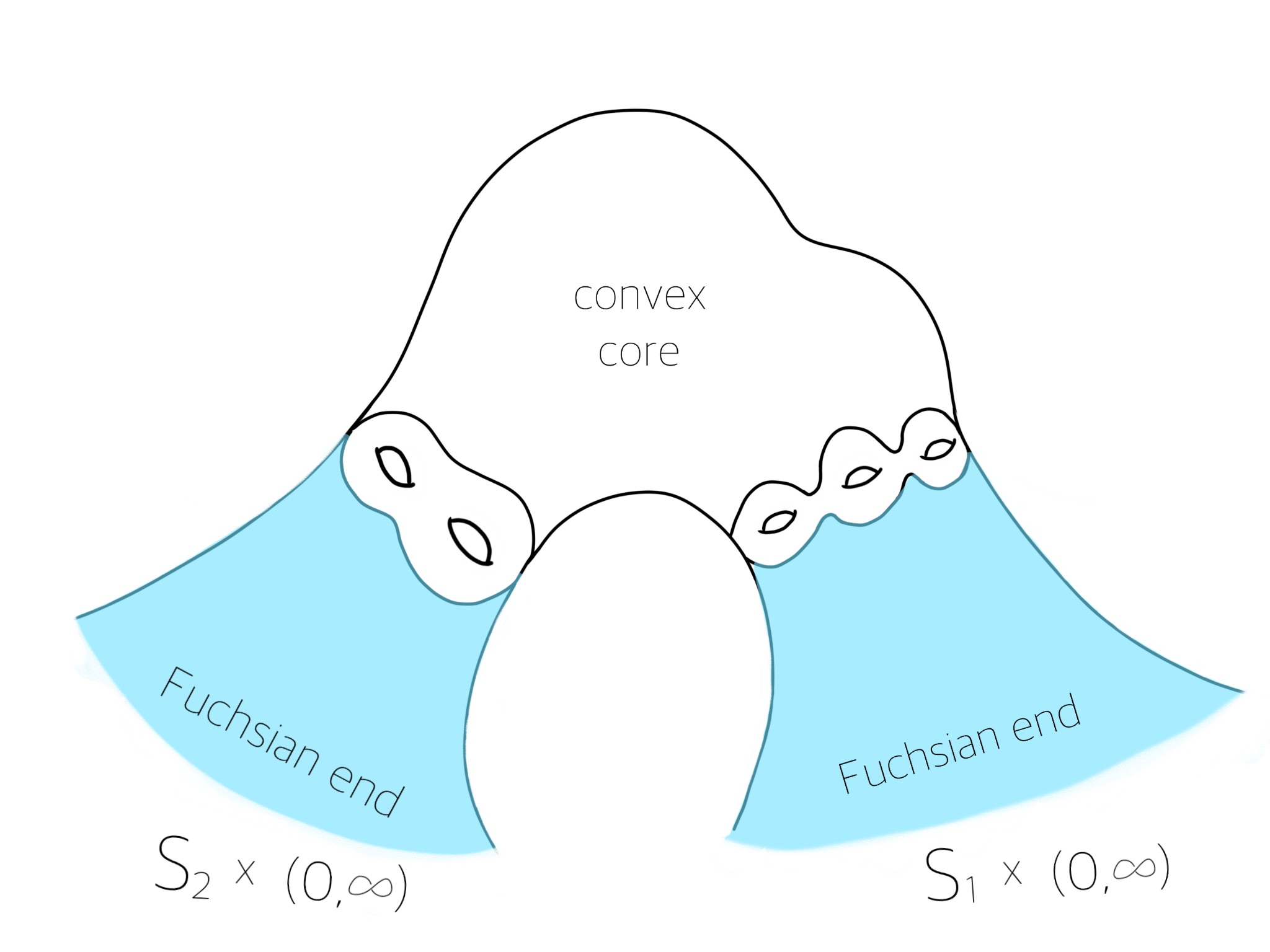}
 \end{center} \caption{Convex cocompact manifolds with Fuchsian ends.} \end{figure}

Convex cocompact hyperbolic manifolds with non-empty Fuchsian ends are constructed from closed hyperbolic manifolds
as follows.
 Begin with a closed hyperbolic $d$-manifold $N_0=\Gamma_0\ba \bH^d$ with a fixed collection of finitely many, mutually disjoint,
 properly embedded totally geodesic hypersurfaces. Cut $N_0$ along those hypersurfaces and perform the metric completion
  to obtain a compact hyperbolic manifold $W$ with totally geodesic boundary hypersurfaces.   Then $\Gamma:=\pi_1(W)$ injects to $\Gamma_0=\pi_1(N_0)$, and
 $M:=\Gamma\ba \bH^d$ is a convex cocompact hyperbolic manifold with Fuchsian ends.

Unlike  $d=3$ case, 
Kerckhoff and Storm showed that if $d\ge 4$,
 a convex cocompact
hyperbolic manifold $M=\Gamma\ba \bH^d$ with Fuchsian ends does not allow any non-trivial local deformation, in the sense that
  the representation of $\Gamma$
into $G$ is infinitesimally rigid
\cite{KS}.


\subsection{Orbit closure of unipotent flows are relatively homogeneous} 
 We let $A=\{a_t\}$ be the one parameter subgroup
 of  semisimple elements of $G$ which give the frame flow, and let $N\simeq \br^{d-1}$ denote the contracting horospherical subgroup.
 We have a compact $A$-invariant subset $\RFM=\{x\in \Gamma\ba G: \text{$xA$ is bounded}\}$.
 
 The following presents a generalization of Theorems \ref{cc2} and \ref{mc} to any dimension:
  
 \begin{thm} \label{oc}\cite{LO} Let $d\ge 2$ and $M$ be a convex cocompact hyperbolic $d$-manifold with Fuchsian ends. 
 Let $U$ be any connected closed subgroup of $G$ 
 generated by unipotent elements.
Suppose that $U$ is normalized by $A$.
Then the closure of any $U$-orbit is relatively homogeneous in $\RFM$, in the sense that  for any $x\in \RFM$,
$$\overline{xU}\cap \RFM =xL\cap \RFM $$
for  a connected closed reductive  subgroup $U<L<G$ such that $xL$ is closed.
\end{thm}

When $M$ has finite volume, this is a special case of Ratner's orbit closure theorem \cite{R}.
This particular case was also proved by Shah by topological methods \cite{Sh0}.

Theorem \ref{oc} and its refinements made in \cite{LO}  yield  the analogous topological rigidity of geodesic planes and horocycles.
A geodesic $k$-plane of $M$ is the image of a totally geodesic immersion $f:\bH^k\to M$.

\begin{thm}\label{mt6} \cite{LO}
Let $M$ be a convex cocompact hyperbolic $d$-manifold with Fuchsian ends.
Then for any $2\le k\le d-1$,
\begin{enumerate}
\item the closure of any  geodesic $k$-plane  intersecting $\core M$
is a properly immersed geodesic $m$-plane for some $k\le m\le d$;
\item a properly immersed geodesic $k$-plane  is a convex cocompact (immersed) hyperbolic $k$-manifold with Fuchsian ends;
\item there are at most countably many maximal properly immersed geodesic planes intersecting $\core M$;
\item any infinite sequence of maximal properly geodesic planes intersecting $\core M$ becomes dense in $M$.
\end{enumerate}
\end{thm}

\begin{figure} \label{limit} \begin{center}
 \includegraphics [height=4cm]{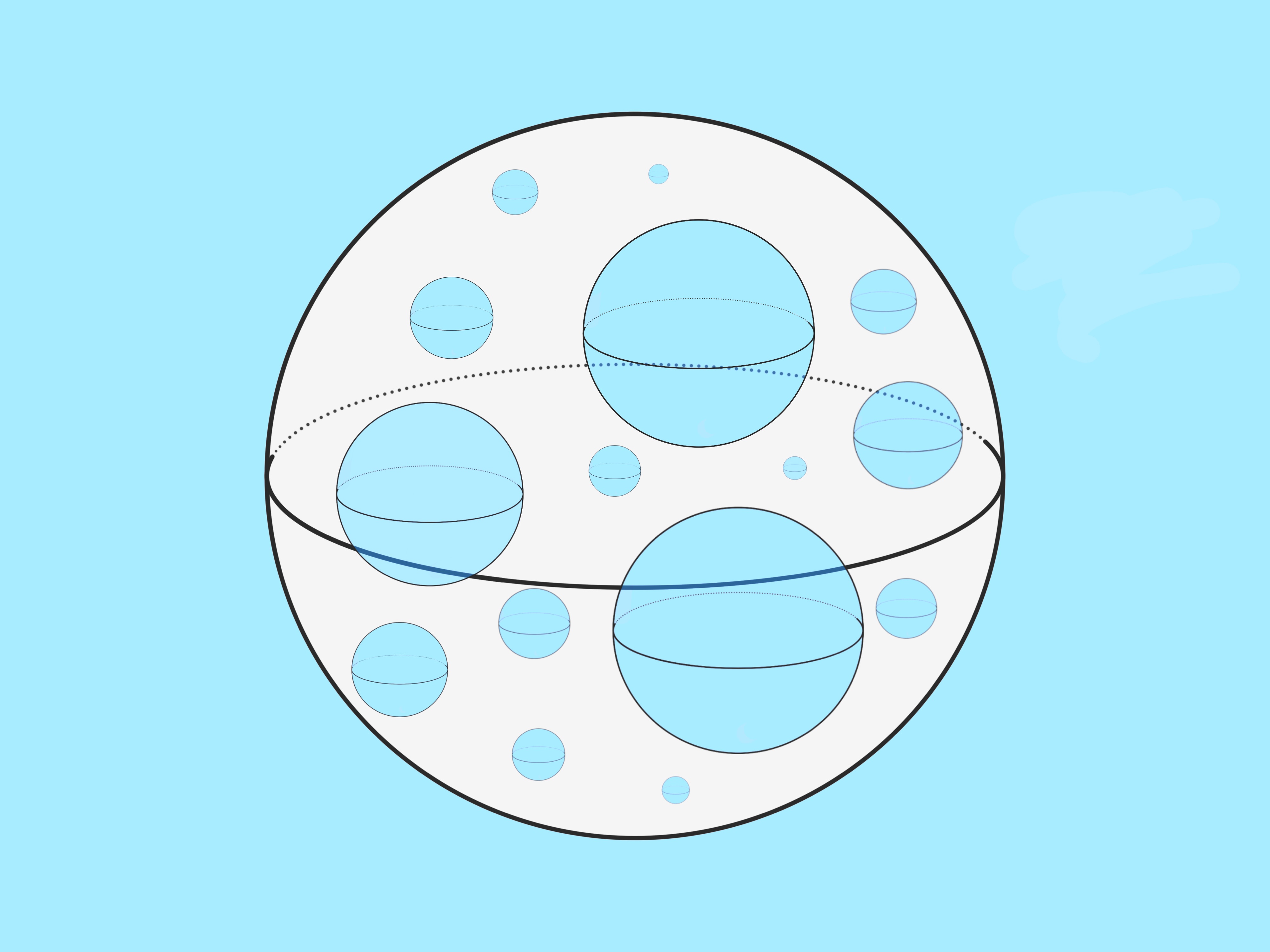}
\end{center} \caption{Limit set of a convex cocompact hyperbolic $4$-manifold with Fuchsian boundary}
 \end{figure}

A  $k$-horosphere in $\bH^d$ is a Euclidean sphere of dimension $k$ which is tangent to a point in $\mathbb S^{d-1}$. A $k$-horosphere in $M$ is simply the image of a $k$-horosphere in $\bH^d$ under the covering map $\bH^d\to M=\Gamma\ba \bH^d$.

\begin{thm} \label{geo-intro3} \cite{LO} Let $\chi$ be a $k$-horosphere of $M$ for $k\ge 1$.
Then either
\begin{enumerate}
\item $\chi$ is properly immersed; or
\item
$\cl{\chi}$ is
 a properly immersed $m$-dimensional submanifold, parallel to a convex cocompact geodesic $m$-plane of $M$ with Fuchsian ends for some $m\ge k+1$.
\end{enumerate}
\end{thm}

\subsection{Avoidance of singular set}\label{Sing}
An important ingredient of the proof of Theorem \ref{oc} which  appears newly for $d\ge 4$ is the avoidance of the singular set
along the recurrence time of unipotent flows to $\RFM$.

Let $U=\{u_t\}$ be a one parameter unipotent subgroup of $N$. Extending the definition given by Dani-Margulis \cite{DM}
to the infinite volume setting, we define the singular set $\mS(U)$
as
 \be\label{sss} \mathscr{S}(U):= \bigcup xL  \cap \RFPM \ee
where $\RFPM=\RFM \cdot N$, and the union is taken over all
 closed orbits $xL$ of proper connected closed subgroups $L$ of $G$
containing $U$.  Its complement in $\RFPM$
 is denoted by $\mathscr{G}(U)$, and called the set of {\it generic} elements of
$U$.

 The structure of $\mS(U)$ as the countable union of {\it singular tubes} is an important property which plays
crucial roles in both measure theoretic and topological aspects of the study of unipotent flows.
Let  $\mH$ denote the collection of all proper connected closed subgroups $H$ of $G$ containing a unipotent element such that
$\Gamma\ba \Gamma H$ is closed and $H\cap \Gamma$ is Zariski dense in $H$. For each $H\in \mH$, we define the singular tube:
 $$X(H, U):=\{g\in G: gUg^{-1}\subset H\} .$$

We have the following:
\begin{enumerate}
\item  $\mH$ is {\it countable};
\item $X(H_1, U)\cap g X(H_2, U)= X(H_1\cap gH_2g^{-1}, U)$ for any $g\in G$;
\item If $H_1, H_2\in \mH$ with $X(H_1\cap H_2, U)\ne \emptyset$, there exists
a closed subgroup $H_0\subset H_1\cap H_2 $ such that $H_0\in \mH$.
\end{enumerate}

In parictular $\mS(U)$ can be expressed as the union of countable singular tubes:
$$\mS(U)= \bigcup_{H\in \mH}\Gamma\ba \Gamma X(H, U) \cap \RFPM.$$


\begin{Rmk}\label{eee} \rm
If $\Gamma<G=\PSL_2(\c)$ is a uniform lattice, and $U$ is the one-parameter
subgroup as in \eqref{uuu}, then  $H\in \mH$ if and only if $H=g^{-1}\PSL_2(\br)g$  for $g\in G$ such
that $\Gamma $ intersects $g^{-1}\PSL_2(\br)g$ as a uniform lattice.
It follows that if $H_1, H_2\in \mH$ and $X(H_1, U)\cap X(H_2, U)\ne \emptyset$, then $H_1=H_2$.
 
 
We note that $\mH$ and hence $\mS(U)$ may be empty in general; see Remark \ref{noo}(1). \end{Rmk}

When the singular set $\mS(U)$ is  non-empty,
it is very far from being closed in $\RFPM$; in fact, it is dense, which is an a posteriori fact.
Hence presenting a compact subset of $\mS(U)$ requires some care, and we will be
using the following family of compact subsets $\mS(U)$ in order to discuss the recurrence of $U$-flows relative
to the singular set $\mS(U)$.
We define $\cal E=\cal E_{U}$ to be the collection of all subsets of $\mS(U)$ which are
of the form
$$\bigcup \Gamma\ba \Gamma H_i D_i \cap \RFM$$
where $H_i\in \mH$ is a finite collection, and $D_i$ is a compact subset of $X(H_i, U)$.

The following theorem was obtained by Dani and Margulis \cite{DM} and independently by Shah
\cite{Sh1} using the linearization methods, which translates the study of unipotent flows on $\Gamma\ba G$
to the study of vector-valued polynomial maps via linear representations.
\begin{thm}[Avoidance theorem for lattice case] \cite{DM}\label{ooo}
Let $\Gamma< G$ be a uniform lattice, and let $U<G$ be a one-parameter unipotent subgroup.
Then for any $\e>0$, there exists
 a sequence of compact subsets $E_1\subset E_2\subset \cdots $ in $\cal E$ such that $\mS(U)=\bigcup_{n\ge 1} E_n$
 which satisfies the following:
Let $x_j$ be a sequence converging to $x\in \mathscr{G}(U)$. For each $n\ge 1$,
 there exist a neighborhood $\cal O_n$
of $E_n$  and $j_n\ge 1$
 such that for all $j\ge j_n$ and for all  $T>0$,
\begin{equation}\label{dmg}
 \ell \{t\in [0,T] : x_ju_t\in \bigcup_{i\le n} \cal O_i\} \le \e T \end{equation}
 where $\ell$ denotes the Lebesgue measure.
\end{thm}

If we set
$$\mT_{n} := \{t\in \br: x_{j_n}u_t\notin \bigcup_{i\le n} \cal O_i\},$$
 then   for any sequence $\lambda_n\to \infty$,
$\limsup \lambda_n^{-1} \mT_{n}$ accumulates at $0$ and $\infty$; and hence
 the sequence $\mT_{n}$ has accumulating renormalizations.

When $xL$ is a closed orbit of a connected closed subgroup of $L$ containing $U$, 
the relative singular subset $\mS(U, xL)$ {\it of} $ xL\cap \RFPM$ is defined similarly by replacing $\mH$
by its subcollection of proper connected closed subgroups of $L$, and $\mG(U, xL)$ is defined as its complement inside
$xL\cap \RFPM$. And Theorem \ref{ooo} applies in the same way to $\mG(U, xL)$ with the ambient space $\G\ba G$ replaced by $xL$.

In order to explain some ideas of the proof of Theorem \ref{ooo}, we will discuss the following
(somewhat deceptively)  simple case when $G=\PSL_2(\c)$ and $\Gamma$ is a uniform lattice. 
Let $U=\{u_t\}$ be  as in \eqref{uuu}.
\begin{prop}\label{ww} Let $E\in\mathcal E_U $.
If $x\in \mG(U)$, then $xU$ spends most of its time outside a neighborhood of $E$,
more precisely,  for any $\e>0$, we can find a neighborhood $E\subset \cal O$ 
such that for all $T>0$,
\be\label{most} \ell \{t\in [0, T] : xu_t \in \cal O\} \le \e  T.\ee
\end{prop}
\begin{proof} Since $xU$ is dense in $\Gamma \ba G$ a posteriori, $xu_t$ will go into any neighborhood of $E$
 for an infinite sequence of $t$'s, but that the proportion of such $t$ is very small is the content of Proposition \ref{ww}.
 In view of Remark \ref{eee}, we may assume that $E$ is of the form $\Gamma\ba \Gamma \op{N}(H) D$
where  $H=\PSL_2(\br)$, and $D\subset V$ is a compact subset; note  $X(H, U)=\op{N}(H) V$, and $\op{N}(H)$
is generated by $H$ and $\text{diag}(i, -i)$.

 As remarked before, we prove this proposition using the linear representation and
 the polynomial-like behavior of unipotent  action. 
 As $\op{N}(H)$ is the group of real points of a connected
  reductive algebraic subgroup, there exists  an $\br$-regular
 representation $\rho:G\to \op{GL}(W)$ with a distinguished point $p\in W$ such that $\op{N} (H)=\op{Stab}(p)$ and
 $pG$ is Zariski closed. The set $pX(H, U)= pV$ is a real algebraic subvariety. \footnote{We can explicitly take $\rho$  and $p$ as follows. Consider the Adjoint representation of $G$ on its Lie algebra $\frak g$.
We then let $\rho$ be the induced representation on  the wedge product space
  $\wedge^3 \frak g$ and set  $p=w_1\wedge w_2\wedge w_3$ where $w_1,w_2, w_3$ is a basis of $\frak h$.}

Note that for $x=[g]$, 
the following are equivalent:
\begin{enumerate}
\item $xu_t\in [e] \op{N}(H) \cal O$;
\item $p\gamma g u_t \in p \cal O$ for some $\gamma\in \Gamma$.
\end{enumerate}

Therefore, we now try to find a neighborhood $p\cal O$ of $pD$ so
that  the set $$\{t\in [0, T]: xu_t\in [e]\op{N}(H) \cal O\}\subset \bigcup_{q\in p\Gamma} \{t\in [0,T]: q g u_t\in p\cal O\} $$
is an $\e$-proportion of $T$.
Each set $\{t\in [0,T]: q g u_t\in p\cal O \}$ can be controlled by the following lemma,
which is proved using the property that
the map $t\mapsto \|qgu_t\|^2$ is a  polynomial of degree uniformly bounded for all $q\in p \Gamma$, and
polynomial maps of bounded degree have uniformly {\it slow} divergence. 

\begin{lem}\label{po}\cite[Prop. 4.2]{DM}
Let $\cal A\subset W$ be an algebraic variety. Then for any compact subset $C\subset\cal A$ and any $\e>0$, there exists
a compact subset $C'\subset \cal A$ such that the following holds: for any neighborhood $\Phi'$ of $C'$ in $W$, there exists a neighbhorhood
$\Phi$ of $C$ of $W$ such that for any $q\in W-\Phi'$ and any $T>0$,
$$\ell \{t\in [0,T]: qu_t\in \Phi\}\le \e \cdot \ell \{t\in [0,T]: qu_t\in \Phi '\} .$$
\end{lem}

Applying this lemma to $\cal A=pV$, and $C=pD$, we get a compact subset $C'=pD'$ for $D'\subset V$.
Since $x\notin [e]\op{N}(H) D'$, we can find a neighborhood $\cal O'$ so that $x\notin [e]\op{N}(H)\cal O'$.
Fix a neighborhood $\Phi'$ of $C'$, so that $\Phi'\cap pG\subset p\cal O'$.
We then get a neighborhood $\Phi$ of $C$ such that if $\cal O$ is a neighborhood of $D$ such that $p\cal O\subset \Phi$,
then \be\label{bob} \ell (J_q\cap  [0, T] ) \le \e \cdot  \ell (I_q\cap  [0, T] )\ee
where $J_q:=\{ t\in \br : q gu_t \in p\cal O\}$ and $I_q:=\{ t\in \br : q gu_t \in p\cal O'\}$.

We now claim that in the case at hand, we can find a neighborhood $\cal O'$ of $D'$ so that
all $I_q$'s are  mutually disjoint:
 \be\label{dis} \text{ if $q_1\ne  q_2$ in $p\Gamma$, then $I_{q_1}\cap I_{q_2}=\emptyset$}.\ee
  Using  \eqref{bob}, this would finish the proof, since  
$$\ell \{t\in [0, T] : xu_t \in [e]\op{N}(H)\cal O \} \le \sum_{q\in p\Gamma}
\ell (J_q\cap  [0, T] ) \le \e \cdot \sum_{q\in p\Gamma}   \ell (I_q\cap  [0, T] ) \le \e T.$$

To prove \eqref{dis}, we now observe the special feature of this example, namely, {\it{no}}
singular tube $\Gamma\ba \Gamma
X(H, U)$ has {\it self-intersection}, meaning that 
\be\label{self} \text{$X(H, U)\cap \gamma X(H, U)
=\emptyset$ if $\gamma\in \Gamma -\op{N}(H)$.}\ee
If non-empty,  by Remark \ref{eee}, we must have
 $H\cap \gamma H\gamma^{-1}=H$, implying that $\gamma\in \op{N}(H)$.
Now if $t\in I_{p\gamma_1}\cap I_{p\gamma_2}$, then $gu_t\in \gamma_1^{-1} H V\cap \gamma_2^{-1} HV$
 and hence $\gamma_1\gamma_2^{-1}\in \op{N}(H)$. So $p\gamma_1= p\gamma_2$, proving \eqref{dis}.
 \end{proof}

\medskip

In the higher dimensional case, we cannot avoid self-intersections of $ \Gamma X(H, U)$; so $I_q$'s are not pairwise disjoint, which means a more careful
study of the nature of the self-intersection is required. 
Thanks to the countability of $\mH$, 
 an inductive argument on the dimension of $H\in \mH$ is used to take care of the
issue, using the fact that
the intersections among $\gamma X(H, U)$, $\gamma\in \Gamma$ are essentially of the form $X(H_0,U)$ for a proper connected  closed subgroup
$H_0$ of $H$ contained in $\mH$ (see \cite{DM} for details).

\medskip

In order to
 illustrate the role of Theorem \ref{ooo} in the study of orbit closures, we prove the following sample case:
  let $G=\SO^\circ(4,1)$, $H=\SO^\circ (2,1)$ and $L=\SO^\circ (3,1)$;
the subgroups $H$ and $L$ are chosen so that $A<H<L$ and $H\cap N$ is a one-parameter unipotent subgroup. The centralizer 
$\op{C}(H)$ of $H$ is $\SO(2)$. We set $H'=H \op{C}(H)$.
\begin{prop} \label{in} Let $\Gamma<G $ be a uniform lattice. Let $X=\overline{xH'}$ for some $x\in \Gamma\ba G$.
If $X$ contains a closed orbit $zL$ properly, then $X=\Gamma\ba G$.
\end{prop}
A geometric consequence of this proposition is as follows:
let $M$ be a closed hyperbolic $4$-manifold, and let $P\subset M$ be a geodesic $2$-plane.
If $\cl P$ contains a properly immersed geodesic $3$-plane $P'$, then the closure $\cl P$ is either $P'$ or $M$.

\begin{proof} Let $U_1=H\cap N$ and $U_2=H\cap N^+$ where $N^+$ is the expanding horospherical subgroup of $G$. Then
the subgroups  $U_1$ and $ U_2$ generate $H$, and the  intersection of the normalizers of $U_1$ and $U_2$ is equal to $A\op{C}(H)$.
Since $zL$ is compact, each $U_\ell$ acts ergodically on $zL$ by Moore's ergodicity theorem. Therefore we may
choose $z$ so that $zU_\ell$ is dense in $zL$ for each $\ell=1,2$.

It suffices to show $X$ contains either $N$ or $N^+$-orbit. 
Since $zL$ is a proper subset of $ X$, there exists $g_n\to e$ in $G-L \op{C} (H)$ such that $x_n=zg_n\in X$.
As $L$ is reductive, the Lie algebra of $G$ decomposes into $\op{Ad}(\frak l)$-invariant subspaces
$\frak l\oplus \frak l^\perp$ with $\frak l$ the Lie algebra of $L$. Hence we write $g_n=\ell_n r_n$
with $\ell_n \in L$ and $r_n \in \exp \frak l^\perp -\op{C}(H)$.
As $g_n\notin \op{C}(H)$, there exists $1\le \ell \le 2$ such that no $r_n$ belongs to the normalizer of $U_\ell$, by passing to a subsequence.
We set $U=U_\ell$.  Without loss of generality we assume $U=H\cap N$; otherwise replace $N$ by $N^+$ in the argument
below.

Note that $\overline{zU}=zL$, in particular, $z$ is a generic point: $z\in \mathcal G(U, zL)=zL-\mS(U, zL)$.
We replace the sequence $z\ell_n$ by $z\ell_{j_n}$ with $j_n$ given by Theorem \ref{ooo}.

Define \be \label{mmm2} \mT_n:= \{t\in \br  : z\ell_n u_t\notin \bigcup_{i\le n} \cal O_i\} .\ee
By Theorem \ref{ooo} applied to $zL=z\SO^\circ (3,1)$,  $\mT_n$ has accumulating renormalizations.

Now by a similar argument as in the proof of Lemma \ref{U}(3),
we can show that $$\limsup \{u_t r_n u_{-t}: t\in \mT_n \}$$
accumulates at $0$ and $\infty$ in $V$ where $V$ is the one-dimensional unipotent subgroup $(L\cap N) V= N$.
In particular, there exists $v\in V$ of arbitrarily large size such that
$v= \lim u_{-t_n} r_n u_{t_n} $ for some $t_n \in \mT_n$.

Note that $z\ell_n u_{t_n} $ is contained in the compact subset $ zL- \bigcup_{i\le n} \cal O_i$. Since $\bigcup_i \cal O_i$
is a neighborhood of $\mS(U, zL)$, $z\ell_n u_{t_n}$ 
 converges to some \be\label{zzz} z_0\in \mG(U, zL).\ee
Therefore $$zg_n u_{t_n}= z \ell_n u_{t_n} (u_{-t_n} r_n u_{t_n}) \to z_0 v.$$

Since $z_0\in \mG(U, zL)$, by Proposition \ref{quote}, we have
$$X\supset \cl{z_0v U}= \cl{z_0U} v=zL v.$$

As $v$ can be taken arbitrarily large, we get a sequence $v_n\to \infty$ in $V$ such that $X\supset zL v_n$.
Using the $A$-invariance of $X$, we get $X\supset z L(Av_nA)\supset z(L\cap N) V_+$ for some one-parameter semigroup $V_+$ of $V$.
Since $X\supset zv_n (L\cap N) v_n^{-1}V_+$, and $\limsup v_n^{-1} V_+=V$, $X$ contains an $N$ orbit, finishing the proof.
\end{proof}

Roughly speaking, if $H$ is a connected closed subgroup of $G$ generated by unipotent elements,
the proof of the theorem that $\cl{xH}$ is homogeneous uses an inductive argument on the codimension of $H\cap N$ in $N$
and involves repeating the following two steps:
\begin{enumerate}

\item Find a closed orbit $zL$ inside $\cl{xH}$ for some connected reductive subgroup $L<G$.

\item If $\cl{xH}\ne zL$, then enlarge $zL$, i.e., find a closed orbit $zL'$ inside $\cl{xH}$ with $\op{dim} (L'\cap N)>
\op{dim}(L\cap N)$.
\end{enumerate}

The proof of Proposition \ref{in} is a special sample case of the step (2), demonstrating
the importance of getting accumulating renormalizations for
the sequence of return time avoiding the exhausting sequence of compact subsets of the singular set.

The following version of the avoidance theorem in \cite{LO} is a key ingredient in the proof of Theorem \ref{oc}:

\begin{thm}[Avoidance theorem] \label{hard} Let $M=\Gamma\ba \bH^d$ be a convex cocompact hyperbolic manifold with Fuchsian ends.
Let $U<N$ be a one-parameter unipotent subgroup.
There exists
 a sequence of compact subsets  $E_1\subset E_2\subset \cdots $ in $\cal E$ such that
 $\mathscr S(U)\cap \RFM=\bigcup_{n\ge 1} E_n$  which satisfies the following:
Let $x_j\in \RFM$ be a sequence converging to $x\in \mathscr{G}(U)$. For each $n\ge 1$,
 there exist a neighborhood $\cal O_n$
of $E_n$  and $j_n\ge 1$
 such that for all $j\ge j_n$,
\begin{equation}\label{dmg2}
 \mT^\diamond(x_j):= \{t\in \br : x_ju_t\in \RFM - \cal O_n\}  \end{equation}
 has accumulating renormalizations. \end{thm}

Note that in the lattice case, one can use the Lebesgue measure $\ell$ to understand the return time
away from the neighborhoods $\cal O_n$ to prove Theorem \ref{ooo}, as was done in \cite{DM}
(also see the proof of Proposition \ref{ww}). In the case at hand,
the relevant return time  is a subset of $\{t\in \br: x_n u_t\in \RFM\}$ on which it is not clear if there exists any {\it friendly} measure.
This makes the proof of Theorem \ref{hard} very delicate, as we have to examine each return time to $\RFM$
and  handpick the time outside $\cal O_n$. First of all,
we cannot reduce a general case to the case $E\subset \Gamma\ba \Gamma X(H, U)$ for a single $H\in\mH$.
 This means that  not only do we need to understand the {\it self-intersections}
of $\Gamma X(H, U)$, but we also have to control intersections among  different $\Gamma X(H,U)$'s in $\mS(U)$, $H\in \mH$.

We cannot also use an inductive argument on the dimension of $H$.
When $G=\SO(3,1)$, there are no intersections among closed orbits
in $\mS(U)$ and the proof is much simpler in this case.
In general, our arguments are
based on the $k$-thick recurrence time to $\RFM$, a much more
careful analysis on the graded intersections of among $\Gamma X(H,U)$'s, $H\in \mH$,  and a combinatorial inductive search argument. 
We prove  that there exists $\kappa >1$, depending only on $\Gamma$ such that
$ \mT^\diamond(x_n)$ is $\kappa$-thick in the sense that for any $r>0$,
$$\mT^\diamond(x_n)\cap \pm [r, \kappa r]\ne \emptyset.$$
We remark that unlike the lattice case, we are not able to prove that
$ \{t\in \br : x_nu_t\in \RFM - \bigcup_{j\le n}\cal O_j\}  $ has accumulating renormalizations.
This causes an issue in carrying out a similar proof as in Proposition \ref{in}, as we cannot conclude
the limit of $x_n u_{t_n}$ for $t_n\in \mT^\diamond(x_n)$ belongs to a generic set as in \eqref{zzz}.

Fortunately, we were able to devise an inductive argument (in the proof of Theorem \ref{mainth} below) which involves an extra step of proving equidistribution
of translates of maximal closed orbits and overcome this difficulty. 



\subsection{Induction} For a connected closed subgroup $U<N$,
we denote by $H(U)$
the smallest  closed simple Lie subgroup of $G$
which contains both $U$ and $A$. If $U\simeq \br^{k}$, then $H(U)\simeq \SO^\circ(k+1,1)$.
A connected closed subgroup of $G$ generated by  one-parameter unipotent subgroups is, up to conjugation, of the form
 $U<N$ or $H(U)$ for some $U<N$.

We set $F_{H(U)}:=\RFPM \cdot H(U)$, which is a closed subset.  
We define the following collection of closed connected subgroups of $G$:
\begin{equation*}
\mathcal L_U:=\left\{L=H(\widehat U) C: 
\begin{array}{c}
\text{for some $z\in \RFPM$, $zL $ is closed in $\Gamma\ba G$  }\\  \text{ and $\op{Stab}_L(z)$ is Zariski
dense in $L$}
\end{array}
\right\}.
\end{equation*}
 where $U <\widehat U<N$ and $C$ is a closed subgroup of the centralizer of $H(\widehat U)$.
We also define: $$\cal{Q}_U:=\{vLv^{-1} :  L\in\cal{L}_U\text{ and } v\in N\}.$$

Theorem \ref{oc} follows from the following:
\begin{theorem} \cite{LO} \label{thm.H'UMIN}\label{mainth}
Let $M=\Gamma\ba \bH^d$ be a convex cocompact hyperbolic manifold with Fuchsian ends.
\begin{enumerate}
\item
For any $x\in \RFM$,
\begin{equation*}
\cl{xH(U)}=xL\cap F_{H(U)}
\end{equation*}
where  $xL$ is a closed orbit of some $L\in \mathcal L_{U}$.
\item Let $x_0\widehat L$ be a  closed orbit for some $\widehat L\in\cal{L}_U$ and $x_0\in \RFM$.
\begin{enumerate} 
\item  For any $x\in x_0\widehat L\cap \RFPM$, 
 \begin{equation*}
\cl{xU}=x  L\cap\op{RF}_+M
\end{equation*}
where  $x L$ is a closed orbit of  some $L\in \cal{Q}_U$.
\item  For any $x\in x_0\widehat L\cap \RFM$,
 \begin{equation*}
\cl{xAU}= xL\cap\op{RF}_+M
\end{equation*}
where  $x L$ is a closed orbit of  some $L\in \cal{L}_U$.\end{enumerate}

\item Let $x_0\widehat L$ be a  closed orbit for some $\widehat L\in\cal{L}_U$ and $x_0\in \RFM$. Let $x_iL_i 
\subset x_0\widehat L$ be a sequence of  closed orbits intersecting $\RFM$  where $x_i\in\RFPM$, $L_i\in\cal{Q}_U$. 
Assume that  no infinite subsequence of $x_iL_i$ is contained in a subset
of the form $y_0L_0D$ where  $y_0L_0$ is a closed orbit of $L_0\in \cal L_U$ with $\op{dim}L_0<\op{dim} \hat L$  and $D\subset \op{N}(U)$ is a compact subset.
Then 
$$\lim\limits_{i\to\infty}\text{ }( x_i L_i \cap\op{RF}_+M)=x_0\widehat L\cap \op{RF}_+M.$$

\end{enumerate}
\end{theorem}

We prove (1) by induction on the co-dimension of $U$ in $N$, and (2) and (3)
by induction on the co-dimension of $U$ in $\widehat L \cap N$.
Let us say $(1)_m$ holds, if $(1)$ is true for all $U$ satisfying $\op{co-dim}_N(U)\leq m$.
We will say $(2)_m$ holds, if $(2)$ is true for all $U$ and $\widehat{L}$ satisfying $\op{co-dim}_{\widehat L \cap N}(U)\leq m$ and similarly for $(3)_m$.

We  then deduce $(1)_{m+1}$ from $(2)_m$ and $(3)_m$, and  $(2)_{m+1}$
 from $(1)_{m+1}$, $(2)_m$, and  $(3)_m$  and finally 
 $(3)_{m+1}$ from  $(1)_{m+1}$, $(2)_{m+1}$ and $(3)_m$. In proving Theorem \ref{oc} for lattice case, we don't need
 $(3)_m$ in the induction proof. In the case at hand, $(3)_m$ is needed since we could not obtain a stronger version
 of Theorem \ref{hard} with $\cal O_n$ replaced by $\cup_{j\le n} \cal O_j$.

We remark that in the step of proving $(2)_{m+1}$ the following geometric feature of convex cocompact hyperbolic manifolds
$M$ of Fuchsian ends is used to insure that $\mS(U, x_0\widehat L)\ne \emptyset$.

\begin{prop} For any $2\le k\le d$, any properly immersed geodesic $k$-plane of $M$  is either compact or contains a compact geodesic $(k-1)$-plane.
\end{prop}
This proposition follows from the hereditary property that any properly immersed geodesic $k$-plane $P$ of $M$
is a convex cocompact hyperbolic $k$-manifold of Fuchsian ends; hence either $P$ is compact (when $P$ has empty ends)
or the boundary of $\core P$ provides a co-dimension one compact geodesic plane.

\end{document}